\documentclass[aps,prd,10pt,notitlepage,nofootinbib,superscriptaddress,showkeys,showpacs]{revtex4-1}

\usepackage{amsmath,amssymb,amsthm,latexsym}
\usepackage[english]{babel}
\usepackage{color}
\usepackage{xspace}
\usepackage{graphicx}
\usepackage{pifont,dsfont}
\usepackage{marvosym}
\usepackage{slashed}
\usepackage[caption=false]{subfig}
\usepackage{enumitem}

\usepackage{hyperref}

\theoremstyle{definition}
\newtheorem{lemma}{Lemma}

\newtheorem{definition}{Definition}
\newtheorem{theorem}{Theorem}

\newtheorem{corollary}{Corollary}
\newtheorem{proposition}{Proposition}

\newcommand{\cG}{\mathcal{G}}
\newcommand{\cT}{\mathcal{T}}

\begin{document}

\title{\Large \bf Maximizing the number of edges in three-dimensional colored triangulations whose building blocks are balls}

\author{{\bf Valentin Bonzom}}\email{bonzom@lipn.univ-paris13.fr}
\affiliation{LIPN, UMR CNRS 7030, Institut Galil\'ee, Universit\'e Paris 13,
99, avenue Jean-Baptiste Cl\'ement, 93430 Villetaneuse, France, EU}

\date{\small\today}

\begin{abstract}
\noindent
Colored triangulations offer a generalization of combinatorial maps to higher dimensions. Just like combinatorial maps are gluings of polygons, colored triangulations are built as gluings of special, higher-dimensional building blocks, such as octahedra, which we call colored building blocks and known in the dual as bubbles. A colored building block is fully determined by its boundary triangulation, which in the case of polygons is characterized by an integer, its length. In three dimensions, colored building blocks are thus labeled by some two-dimensional triangulations and those homeomorphic to the 3-ball are labeled by the subset of planar triangulations. Similarly to the two-dimensional case where Euler's formula provides an upper bound on the number of vertices at fixed number of polygons with given lengths, we look in three dimensions for a least upper bound on the number of edges at fixed number of given colored building blocks. In this article we solve this problem when all colored building blocks, except possibly one, are homeomorphic to the 3-ball. To do this, we find a combinatorial characterization of the way a building block homeomorphic to the ball has to be glued to other blocks of arbitrary topology in a colored triangulation which maximizes the number of edges. It implies that in the dual 1-skeleton, the building block can be excised from the triangulations by a sequence of 2-edge-cuts. We show that this local characterization can be extended to the whole triangulation as long as there is at most one building block which is not a 3-ball. The triangulations obtained this way are in bijection with trees. The number of edges is given as an independent sum over the building blocks of such a triangulation. Finally, in the case of all colored building blocks being homeomorphic to the 3-ball, we show that the triangulations which maximize the number of edges are homeomorphic to the 3-sphere. Those results were only known for the octahedron and for melonic building blocks before.

\noindent
This article is self-contained and can be used as an introduction to colored triangulations and their bubbles from a purely combinatorial point of view.

\end{abstract}

\medskip

%\noindent  Pacs numbers: 02.10.Ox, 04.60.Gw, 05.40-a

\maketitle

%%%%%%%%%%%%%%%%%%%%%%%
\section*{Introduction}
%%%%%%%%%%%%%%%%%%%%%%%

Combinatorial maps are gluings of polygons along their edges into connected surfaces. They have a remarkably wide range of applications, since they naturally appear anytime discretizations of two-dimensional manifolds are relevant. This ranges from computer rendering of surfaces to folding and coloring problems \cite{FoldingColoring-DiFrancesco},  RNA foldings \cite{RNAFolding-Orland-Zee} and two-dimensional quantum gravity \cite{MatrixReview}.

From a more mathematically oriented perspective, they have many interesting and intriguing properties. In addition to the traditional setting of discretized surfaces, maps have also been found to appear in various problems of enumerative geometry (albeit all pertaining to two dimensions) \cite{GraphsOnSurfaces} like Kontsevich-Witten intersection numbers \cite{IntersectionNumbers-Kontsevich, CountingSurfaces-Eynard} and Hurwitz numbers \cite{HurwitzGalaxies}. Moreover, they have been studied through a wide range of techniques.

They are classified by a non-negative integer, the genus. Tutte's seminal work on their enumeration unraveled remarkably simple exact counting formulas, e.g. the number of planar quadrangulations with $n$ edges, see any book with chapters devoted to maps like \cite{CountingSurfaces-Eynard}. The simplicity of such formulas has since been explained in numerous cases in terms of bijections \cite{SchaefferBijection, Mobiles}: the combinatorial structures of maps can be re-encoded in different objects whose enumeration is straightforward, e.g. well-labeled trees in the case of planar quadrangulations.

Many other approaches have been developed in the context of map enumeration, bijections being only one of the most modern approaches. The generating functions for different families of planar maps typically satisfy equations with a catalytic variable and some divided differences \cite{CatalyticVariables}. Solving those kinds of equations is an important topic in combinatorics, with relations to quadrant walks in particular. While Tutte wrote those equations first, they were rediscovered and investigated in mathematical physics as saddle point equations for matrix integrals. Furthermore all genera Tutte equations can be rederived as Schwinger-Dyson equations of matrix integrals, i.e. by suitable integration by parts \cite{CountingSurfaces-Eynard, MatrixReview}.

As an alternative to solving those Schwinger-Dyson equations, matrix integrals offer the opportunity to use the method of orthogonal polynomials \cite{MatrixReview}, which for instance provides a simple derivation of Harer-Zagier formula for the generating function of unicellular (1-face) maps of any genus \cite{GraphsOnSurfaces}. Orthogonal polynomials further unraveled the fascinating connections between random maps and integrable hierarchies.

Nowadays, a modern way to solve the Schwinger-Dyson equations for various families of maps is the topological recursion \cite{CountingSurfaces-Eynard}. Designed in the context of matrix models \cite{EynardOrantinReview, AbstractLoopEquations, TR}, the topological recursion has now gone beyond its original frame and motive, i.e. to provide an intrinsic, universal solution to Schwinger-Dyson equations of matrix integrals, and has become a fascinating and powerful framework in enumerative geometry \cite{HurwitzTR, InvariantsOfSpectralCurves-Eynard, WeilPeterssonTR, GromovWitten-Eynard-Orantin, HyperbolicKnotsTR-Borot-Eynard}.

%Remarkably, maps have proved to be relevant beyond their traditional setting and in particular in enumerative geometry, like for generating functions of intersection numbers and Hurwitz numbers. This is also a context where the topological recursion has been successfull.

Maps can also be described as factorizations of permutations. This formulation enables the use of methods from the representation theory of the symmetric group, see for instance \cite{UnicellularMaps-Pittel} and more references on unicellular maps therein. It is also used to reformulate other problems in terms of maps, such as Hurwitz numbers, formulated in terms of permutations and then reinterpreted via maps. Factorizations of permutations also provide a general framework to prove the connection to integrable hierarchies \cite{KP-Goulden-Jackson}.

It is thus an exciting perspective to develop a higher-dimensional theory of maps, i.e. random gluings of simplices, and see if all those beautiful mathematics developed for maps or in connections with maps extend to higher dimensions. This however turns out to be difficult. Even laying promising foundations which generalize the setup of combinatorial maps has proved to be challenging, the interplay between combinatorics and topology being highly non-trivial \cite{Benedetti-Ziegler}. Indeed, one might for instance want to work at fixed topology, similarly to fixing the genus in two dimensions. This turns out to be too difficult as it is not even known whether the number of triangulations of the 3-sphere is exponentially bounded \cite{SpheresAreRare}. It is possible to further constrain the set of admissible triangulations so that their number becomes exponentially bounded. This is the case for locally constructible triangulations \cite{LocallyConstructible-Durhuus-Jonsson}, introduced by Durhuus and Jonsson. Benedetti and Ziegler in \cite{Benedetti-Ziegler} proved that not all simplicial 3-spheres are locally constructible and we refer to their article for more results and discussions on that topic.

All in all, one might argue that this kind of constraints enforces too many constraints, topological and/or combinatorial, to provide a genuine generalization of combinatorial maps. The key ingredient in our case which seems to be missing in those earlier studies is the classification of triangulations of arbtirary topology, not necessarily spheres, according to the number of $(d-2)$-simplices (number of edges in three dimensions). In the case of colored triangulations as we will see, Gurau's degree measures the distance to triangulations which maximize the number of $(d-2)$-simplices at fixed number of $d$-simplices. Colored triangulations of fixed Gurau's degree have been proved to be exponentially bounded and described combinatorially \cite{Gurau-Schaeffer}.

Random tensors were proposed in the nineties \cite{Tensors} as a generalization of random matrices. In the same way matrix integrals generate combinatorial maps, tensor integrals generate combinatorial objects some of which can be interpreted as triangulations in dimension $d \geq 3$, \cite{DePietri-Petronio}. Tensor integrals also generate more singular objects whose topological interpretations can be challenging \cite{GurauVsSmerlak}. Maybe it would not be too bad if the combinatorics could be analyzed as for maps, e.g. by defining an analog of the genus and performing enumeration. This also turns out to be difficult\footnote{To briefly describe the objects generated by tensor integrals, recall that maps are graphs with additional data which is needed to reconstruct a surface. For instance, in the context of matrix integrals, this data is a ribbon structure: each edge is a ribbon bounded by two strands and ribbons meet at ``ribbon vertices'' which are homeomorphic to discs. In tensor integrals, the number of strands is equal to the dimension and vertices of those graphs can connect a strand from one edge to a strand of another edge.}, although remarkable progress has been made very recently, in the sense that a combinatorial extension of the genus can be defined in some cases \cite{TwoTensors-Gurau, SymTracelessTensors}.

A much more convenient family of triangulations in dimension $d\geq 3$ is that of {\bf colored triangulations}. They are made of simplices whose $d+1$ boundary $(d-1)$-simplices are labeled by the colors $0, 1, \dotsc, d$ and two simplices can be glued in a unique way which respects the coloring. Colored triangulations were first introduced in topology, \cite{ItalianSurvey, LinsMandel} and references therein. It was found in 2011 that tensor integrals equipped with a suitable unitary invariance generate those colored triangulations and only them \cite{Uncoloring}. From the topological perspective, colored triangulations always represent pseudo-manifolds and every PL-manifold admits a colored triangulation. Moreover, as we will explain in Theorem \ref{thm:ColoredGraphs}, they can be encoded as graphs with colored edges, the graph being the 1-skeleton of the dual, which is promising for combinatorics.

The breakthrough was Gurau's discovery that colored triangulations admit a generalization of the genus, called {\bf Gurau's degree}, in a purely combinatorial (as opposed to topological) way \cite{1/NExpansion,GurauBook,GurauRyanReview}. In two-dimensional triangulations, Euler's formula for the genus can be seen as a least upper bound on the number of vertices at fixed number of triangles. Gurau's theorem in dimension $d$, restated here as Theorem \ref{thm:Gurau}, is then an {\bf upper bound on the number of subsimplices of dimension $d-2$} at fixed number of simplices of dimension $d$, which reduces to Euler's formula at $d=2$. The combinatorial classification of colored triangulations at fixed Gurau's degree has been obtained in \cite{Gurau-Schaeffer} while their topological properties have been analyzed in \cite{TopologyTensor-Casali-Cristofori-Dartois-Grasselli}.

The pair ``tensor integrals/colored triangulations'' provides a promising extension of the framework of matrix integrals and combinatorial maps. In addition to Gurau's degree, the generating functions satisfy a system of Schwinger-Dyson equations \cite{SchwingerDysonTensor, Revisiting} although much more challenging than in two dimensions (e.g. it is not known how to use catalytic variables anymore). Since standard methods of matrix models, like saddle points or orthogonal polynomials, which rely on a distribution for the eigenvalues, do not exist for random tensors (as there are no eigenvalues), one has to resort to combinatorics only. A recent bijection between colored triangulations and some colored, stuffed hypermaps was found in \cite{StuffedWalshMaps}. It also has a matrix model interpretation and in the simplest case \cite{Quartic-Nguyen-Dartois-Eynard} it led to bilinear equations {\it \`a la} Hirota \cite{GiventalTensor-Dartois} and to the first topological recursion in dimension greater than two \cite{QuarticTR}.

Another appealing feature of colored triangulations is that they provide a framework to investigate universality, \cite{SigmaReview, Universality, StuffedWalshMaps, Octahedra, MelonoPlanar, DoubleScaling}. Indeed, maps with various prescribed polygons of bounded boundary length (faces with prescribed bounded degrees, in the usual terminology of maps) exhibit several properties which are universal at large scales. For instance, the asymptotic enumeration of planar maps in such families is \cite{MatrixReview, CountingSurfaces-Eynard, Mobiles}
\begin{equation*}
K\, \rho^{-n} n^{-5/2},
\end{equation*}
where $K$ and $\rho$ are family-dependent while the exponent $5/2$ is universal, and $n$ is the size of the map (for instance the number of $p$-angles in a planar $p$-angulation). Other examples of universal features at large scales is the Hausdorff dimension being 4 and the convergence to the Brownian sphere \cite{LeGall}.

A higher-dimensional extension of maps should be rich enough to allow for the investigation of universal features. This is the case with colored triangulations, as found in \cite{Uncoloring} (in the context of tensor integrals). Generalizing the polygons, which are simply determined by the numbers of boundary edges, there is a natural family of {\bf colored building blocks (CBBs)}, for instance the octahedron being one of them in three dimensions. CBBs are defined like colored triangulations except they have a connected boundary and that all $(d-1)$-simplices of color 0 is on the boundary and the boundary only has $(d-1)$-simplices of color 0. As a consequence of the attachment rule for colored simplices, a CBB is actually determined by its boundary and can be constructed from the latter by taking the cone, thereby generalizing the 2-dimensional case of polygons. A 3-dimensional CBB is thus labeled by a 2-dimensional triangulation.

It is thus natural to study families of colored triangulations made of prescribed CBBs. The set of colored triangulations which contains $n_i>0$ CBB of type $K_i$ for $i=1, \dotsc, N$ is denoted $\cT_{n_1, \dotsc, n_N}(K_1, \dotsc, K_N)$. In three dimensions (the case we will restrict attention to), Gurau's degree theorem provides an {\bf upper bound on the number of edges} which grows linearly with the number of tetrahedra. While this bound holds for all colored triangulations, in particular colored triangulations built from any prescribed CBBs, it is in general not optimal. In fact, it is known to be a least upper bound only for a family of CBBs called {\bf melonic} and which have a rather simple, tree-like structure, see \cite{Uncoloring} and Theorem \ref{thm:Melons} here. For other CBBs, such as a typical block homeomorphic to a ball, we expect the existence of a refined bound improving Gurau's on the number of edges.

In this article we will be interested in the case of 3-dimensional colored triangulations where at least one CBB, say $K_1$, is {\bf constrained to be homeomorphic to a 3-ball}. To improve Gurau's bound and find actual least upper bounds, we will characterize the local (i.e. the gluing around $K_1$) combinatorics of colored triangulations which maximize the number of edges at fixed number of CBBs, which form the subset $\cT^{\max}_{n_1, \dotsc, n_N}(K_1, \dotsc, K_N)$. 

Remarkably our method is a direct case-by-case combinatorial analysis and is completely independent of Gurau's initial proof of his bound \cite{1/NExpansion} (his most recent works \cite{TwoTensors-Gurau, SymTracelessTensors} on more general tensor integrals also use different methods). As already mentioned all CBBs are characterized by their boundaries. It means that CBBs homeomorphic to balls are characterized by 2-dimensional planar (colored) triangulations. This planar property will be key in our main results which are the following.
\begin{itemize}
\item {\bf Theorem \ref{thm:1Planar}} fully characterizes the way a CBB homeomorphic to a ball has to be glued to other CBBs (of arbitrary topology) if the whole triangulation is to maximize the number of edges. This local characterization, which we call the maximal 2-cut property, has the central property that the 1-skeleton dual to the CBB has to be connected to those of the other CBBs by 2-edge-cuts only. The incidence vertices of those 2-edge-cuts is dictated by the triangulations which contain this CBB only and maximize the number of edges, i.e. $\cT^{\max}_{n=1}(K_1)$. This 1-building-block problem is a generalization of unicellular maps in two dimensions. Remarkably, Theorem \ref{thm:1Planar} however does not require any properties of those 1-CBB triangulations (except their existence of course).
\item {\bf Theorem \ref{thm:OnlyPlanar}} uses the previous theorem to fully describe the set $\cT^{\max}_{n_1, \dotsc, n_N}(K_1, \dotsc, K_N)$ in the case all CBBs $K_1, \dotsc, K_N$ are homeomorphic to 3-balls except possibly one of them. The characterization is again based on the maximal 2-cut property. It trivially implies that those triangulations are in bijection with trees and their (rooted) generating functions satisfy polynomial equations which can be straightforwardly written down.
\item Finally we investigate the topology of the triangulations in $\cT^{\max}_{n_1, \dotsc, n_N}(K_1, \dotsc, K_N)$ when all CBBs $K_1, \dotsc, K_N$ are homeomorphic to balls. We prove in {\bf Theorem \ref{thm:Topology}} that they are homeomorphic to the 3-sphere. This theorem makes use of classical theorems on moves which preserve the topology of colored triangulations, Theorem \ref{thm:Dipoles} and Proposition \ref{prop:ConnectedSums} which we state without proof. In contrast with Theorem \ref{thm:1Planar}, Theorem \ref{thm:Topology} requires some combinatorial properties of the 1-CBB triangulations to show that they are spheres.
\end{itemize}

The results presented here go a long way beyond the existing results. In three dimensions, colored triangulations which maximize the number of edges at fixed number of tetrahedra were characterized in only three instances:
\begin{itemize}
\item all CBBs are melonic CBBs, which are special blocks homeomorphic to balls \cite{Uncoloring}. If one considers rooted CBBs, then the melonic ones are in bijection with rooted ternary trees \cite{Melons}, thus far from describing all (rooted) CBBs homeomorphic to the 3-ball which are in bijection with rooted colored planar triangulations which are themselves in bijection with bipartite planar maps. 
\item all CBBs are octahedra (the octahedron is a CBB made of eight colored tetrahedra) \cite{Octahedra}. It was the first instance in three dimensions of a non-melonic CBB for which the triangulations maximizing the number of edges could be fully described. It now becomes just one specific case from the present article. The proof used a bijection with colored hypermaps which can now be completely bypassed.
\item all CBBs are the same one made of six tetrahedra and whose dual 1-skeleton is the complete bipartite graph $K_{3,3}$ (treated as an application of the bjection in \cite{StuffedWalshMaps}, meaning that its boundary is the torus and therefore that its gluings are not manifolds.
\end{itemize}
In all those cases, (rooted) elements of $\cT^{\max}_{n_1, \dotsc, n_N}(K_1, \dotsc, K_N)$ were found to be in bijection with trees, and their generating functions readily found to satisfy polynomial equations.

Some CBBs can be constructed as gluings of melonic ones and octahedra, in which cases we can characterize the gluings of those CBBs which maximize the number of edges \cite{SigmaReview}. This method can be seemingly useful to extend results to much larger sets of bubbles, as done originally in \cite{DoubleScaling} then in \cite{MelonoPlanar}. It however does bring genuinely new results since it relies on the triangulations built from the new CBBs being in bijection with a subset of those with the original (melonic and octahedra) bubbles.

%A way to use those results in order to find the triangulations maximizing the numbers of edges with new CBBs themselves made out of special gluings of melonic ones and octahedra has been known for a few years, originally used in \cite{DoubleScaling}. 
%Here, we simply illustrate the method in the case of octahedra in Section \ref{sec:Octahedra}. As explained there, it does not really provide intrinsically new result and in our case the new CBBs $K'_i$ can not be arbitrary CBBs homeomorphic to the 3-ball anyway.

%All the proofs of our main theorems are recursions on the number of tetrahedra. Since all CBBs with up to six tetrahedra are melonic, except for that one with $K_{3,3}$ as dual graph, it means that the melonic case can be used to initialize our recursions.

Theorem \ref{thm:1Planar} generalizes the case of melonic CBBs and octahedra to one arbitrary CBB homeomorphic to the ball glued to any CBBs without restriction. The reason why CBBs homeomorphic to the 3-ball form a natural family from the combinatorial point of view is that they can be represented using planar maps and their planarity will be preserved by combinatorial topological moves which decrease the number of tetrahedra and thus make inductions possible.

The present article is thus a major generalization of the melonic case and the octahedron. Only the third case, with $K_{3,3}$ as dual graph, is not a special case of the present result. We however do believe that it should be possible to extend our main theorems to CBBs whose boundaries are tori by combining the results of \cite{StuffedWalshMaps} with the methods of the present article. While our core lemmas rely on properties of planar maps which would not holds as such, they may be generalizable by ``factorizing'' the non-planarity, e.g. decomposing such a CBB as coming from the one with $K_{3,3}$ as dual graph on which the same topological moves as in the planar case are used to make the CBB grow in size. This will be the subject of future investigations.

As for Theorem \ref{thm:OnlyPlanar}, it can be compared with Gurau's theorem in \cite{Universality}. The latter is a central limit theorem for probability distributions on complex tensors of large size. Due to the intimate relationship between tensor integrals and generating functions of colored triangulations with prescribed CBBs, there is an area where \cite{Universality} can be applied in our three-dimensional context. It then implies the exact same result as Theorem \ref{thm:OnlyPlanar} with weaker hypotheses: it requires all CBBs to be melonic except one arbitrary, while Theorem \ref{thm:OnlyPlanar} requires all CBBs to be homeomorphic to the 3-ball except one arbitrary.

The present article aims at being as {\bf self-contained and introductory} as possible. In addition to its new results, it intends to be a reference for researchers in combinatorics who would like to find an {\bf introduction on colored triangulations and CBBs}. Up to now, most foundational results for colored triangulations can only be found either in the tensor integral literature, which comes from mathematical physics and requires a totally different background from combinatorics, or in various articles of topology from the 80s which also happen to be the only source for the proofs of those results. In spite of previous articles studying colored triangulations and their dual colored graphs, it appeared to us that no reference which would contain proofs of the combinatorial foundations, including about CBBs, existed in the combinatorics literature.

Therefore, we have decided to write {\bf independent proofs for important and foundational combinatorial theorems} presented here (leaving however the classical theorems with topological content stated without proof). If our article is successfull as an invitation to the subject, we invite the interested reader more inclined towards topology and combinatorics than tensor integrals to have a look at \cite{TopologyTensor-Casali-Cristofori-Dartois-Grasselli, Gurau-Schaeffer, Carrance} with three different, recent points of view. Reference \cite{TopologyTensor-Casali-Cristofori-Dartois-Grasselli} focuses on applying the theory of crystallization to draw topological conclusions on Gurau's degree. 

The Gurau-Schaeffer classification \cite{Gurau-Schaeffer} is a full classification of colored graphs with respect to Gurau's degree. It is the most profound combinatorial analysis of the whole set of colored triangulations with respect to Gurau's degree, and ignores CBBs. A remarkable extension by Fusy and Tanasa \cite{Fusy-Tanasa} has provided the same classification for a more general set of graphs, 3-stranded graphs called multi-orientable graphs, coming from tensor integrals.% Previous work by the author in the combinatorics literature on improving Gurau's bound for non-melonic CBBs did not include the same amount of introductory details as here.

A recent work by Carrance takes a different approach than here, analyzing random gluings of colored simplices through their description as factorizations of permutations. Using probabilistic results on permutations, the author was able to extract various expectations for colored triangulations, such as the mean Gurau's degree. Those results do not apply to our setting since we fix some prescribed CBBs and want to identify the colored triangulations which maximize the number of edges.

The reader interested in modern methods of mathematical physics, related to the tensor integral approach, is invited to consider \cite{QuarticTR}, where the topological recursion is proved to apply to a random tensor model for the first time. This version of the topological recursion is the blobbed one due to Borot \cite{BlobbedTR-Borot}, and Borot and Shadrin \cite{BlobbedTR-Borot-Shadrin}.

It is also worth noting that Stanley introduced balanced simplicial complexes in combinatorics (without relations to colored triangulations in topology as far as we know) as simplicial complexes with colored simplices and the same colored gluing rules as in our case, \cite{Balanced-Stanley} (and \cite{Balanced-Klee-Novik} for a more recent study). A major difference with our objects is that ours are not necessarily simplicial complexes (two $d$-simplices can share more than a single $(d-1)$-simplices, like two triangles sharing their three edges so that each form a hemisphere of the 2-sphere) and a second difference with balanced complexes is that they are typically considered at fixed number of vertices instead of fixed number of $d$-simplices or CBBs. The number of vertices of colored triangulations at fixed number of $d$-simplices is not fixed. We will not give new results on the number of vertices of elements of $\cT^{\max}_{n_1, \dotsc, n_N}(K_1, \dotsc, K_N)$, but we refer to \cite{Carrance} where it is shown that a typical gluing of colored simplices has $d+1$ vertices only.

Section \ref{sec:ColoredTriangulations} is an introduction to colored triangulations including details and proofs on their representation as colored graphs in Section \ref{sec:ColoredGraphs}. This representation as graphs with colored edges will be the preferred representation throughout the article. Section \ref{sec:Bubbles} introduces CBBs and their dual graphs which are called bubbles in the tensor integral literature and also in agreement with the topological literature on colored triangulations. We have included a proof that CBBs are determined by their boundaries.

Section \ref{sec:BicoloredCycles} is the starting point of our focus on maximizing the number of edges at fixed CBBs. In particular, we define melonic building blocks, recall in Theorem \ref{thm:Melons} that only they can saturate Gurau's bound. We then state the main question of the article, i.e. maximizing the number of edges, and its formulation in terms of the objects dual to the edges of the triangulations, which are bicolored cycles.

In Section \ref{sec:BdryBubblesAndFlips} we introduce key tools used in the proofs of the main theorems. Flips (of triangle gluings, or edges in the dual) are described in Section \ref{sec:Flips}. They are transformations of a colored triangulation which only change the gluings between CBBs and does so in a way such that the variations of the number of edges are under control. If a colored triangulation is split into two components, the notion of boundary bubbles defined in Section \ref{sec:BdryBubbles} makes it possible to keep track of the edges shared by the two components using colored graphs. This provides an elementary proof, in Section \ref{sec:4Cuts}, that CBBs in a triangulation of $\cT^{\max}_{n_1, \dotsc, n_N}(K_1, \dotsc, K_N)$ cannot be glued together so that the dual 1-skeleton has a 4-edge-cut. Since the maximal 2-cut property of our main theorems in the case of CBBs homeomorphic to balls relies on eliminating some $k$-edge-cuts for $k\geq 4$, this serves as the initialization for $k=4$. We finally introduce contractions in Section \ref{sec:Contractions}. They allow for decreasing the number of tetrahedra in a CBB while preserving its topology and controlling the variations of the total number of edges. Therefore, contractions make recursions possible.

Finally, all our main theorems are exposed in Section \ref{sec:PlanarBubbles}. Triangulations with a single CBB are introduced in Section \ref{sec:1CBB}. We prove Theorems \ref{thm:1Planar} and \ref{thm:OnlyPlanar} in Section \ref{sec:Max2Cut} and Theorem \ref{thm:Topology} in Section \ref{sec:Topology}.

%%%%%%%%%%%%%%%%%%%
\section{Colored triangulations as generalization of combinatorial maps} \label{sec:ColoredTriangulations}
%%%%%%%%%%%%%%%%%%%

%%%%%%%%%%%%%%%%%%%
\subsection{Colored triangulations as colored graphs} \label{sec:ColoredGraphs}
%%%%%%%%%%%%%%%%%%%

A $k$-simplex is a $k$-dimensional simplex. Although there are often called $k$-faces in a complex in combinatorics, we will avoid this terminology as it overlaps with the traditional use of ``faces'' for the connected components of the complement of the graph in a combinatorial map. We will use this meaning of faces since planar maps will play a major role.

A colored $d$-simplex is a $d$-simplex whose boundary $(d-1)$-subsimplices have a color from $\{0, 1, \dotsc, d\}$ such that each color appears exactly once.

Colors allow us to define a {\bf canonical attaching rule} between colored simplices. First notice that in a colored $d$-simplex, each $(d-2)$-subsimplex is shared by exactly two $(d-1)$-subsimplices, say with colors $c, c'$, and is thus labeled by a pair of colors $\{c, c'\}$. Similarly, a $(d-k)$-subsimplex is identified by a $k$-uple of colors, for $k=1, \dotsc, d$. The colored attaching rule is to identify a $(d-1)$-subsimplex $\sigma_c$ of color $c\in\{0, \dotsc, d\}$ in a $d$-simplex with a $(d-1)$-subsimplex $\sigma'_c$ of the same color in another $d$-simplex in the only way which identifies all the subsimplices of $\sigma_c$ and $\sigma'_c$ which have the same color labels. In other words, it is the only attaching map which preserves all induced colorings of their $k$-subsimplices for $k=0, \dotsc, d-1$.

\begin{figure}
\subfloat[\label{fig:Triangle}A colored triangle with colors $\{0,1,2\}$]{\includegraphics[scale=0.35]{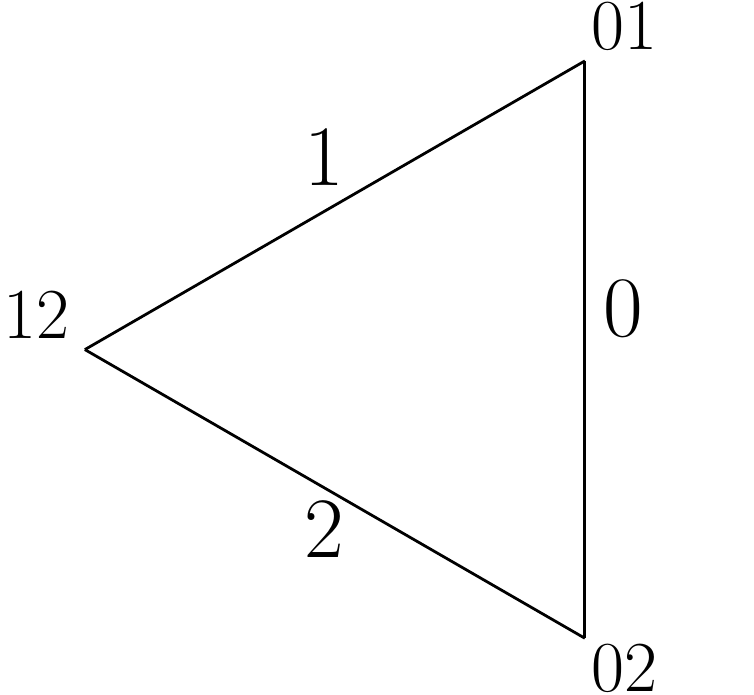}}{\hspace{4cm}}
\subfloat[\label{fig:Tetrahedron}A colored tetrahedron with colors $\{0,1,2,3\}$]{\includegraphics[scale=0.35]{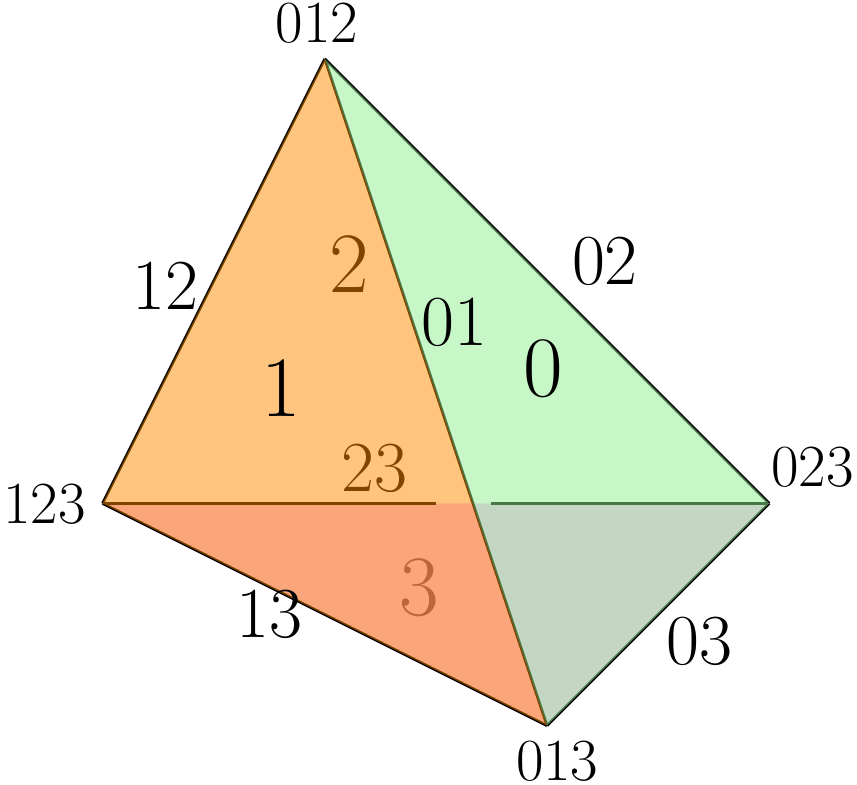}}
\caption{\label{fig:Simplex} A colored triangle and a colored tetrahedron with the induced colorings of their subsimplices by $k$-uple of colors.}
\end{figure}

In two dimensions, a colored triangle has edges with colors 0, 1, 2, and vertices with colors $\{0,1\}, \{1,2\}, \{0,2\}$ where the vertex with colors $\{c, c'\}$ is the one shared by the edges of colors $c, c'$, see Figure \ref{fig:Triangle}. Two triangles can be glued along an edge of say color 0 by identifying the vertices of colors $\{0,1\}$ of both triangles, and similarly identifying the vertices of colors $\{0,2\}$,
\begin{equation}
\begin{array}{c} \includegraphics[scale=.35]{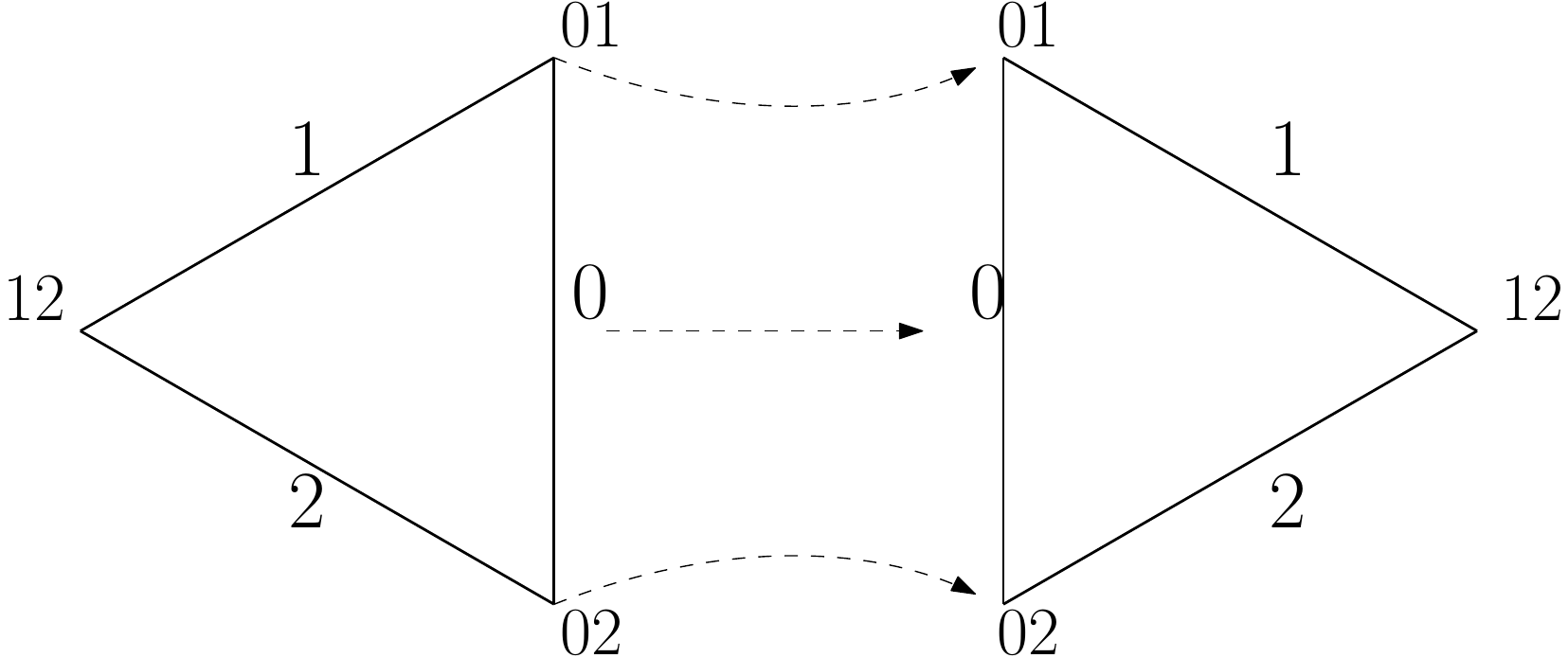} \end{array}
\end{equation}

In three dimensions, a colored tetrahedron has four triangles colored 0, 1, 2, 3, six edges colored $\{a, b\}_{0\leq a<b\leq 3}$ where $\{a, b\}$ labels the edge shared by the triangles of colors $a$ and $b$, and four vertices with labels $\{0,1,2\}$, $\{0, 1, 3\}$, $\{0, 2, 3\}$, $\{1, 2, 3\}$ where the vertex with label $\{a, b, c\}$ is the one shared by the three triangles of colors $a, b, c$, see Figure \ref{fig:Tetrahedron}. Two tetrahedra can be glued along a triangle of color say 0 by identifying pairwise the edges which have 0 in their labels, i.e. both edges of colors $\{0, a\}$ for $a \neq 0$ in the two tetrahedra are identified, and further identifying pairwise the vertices which have 0 in their labels, i.e. both vertices with colors $\{0, a, b\}$ in the two tetrahedra are identified for all $1\leq a<b \leq 3$,
\begin{equation}
\begin{array}{c} \includegraphics[scale=.35]{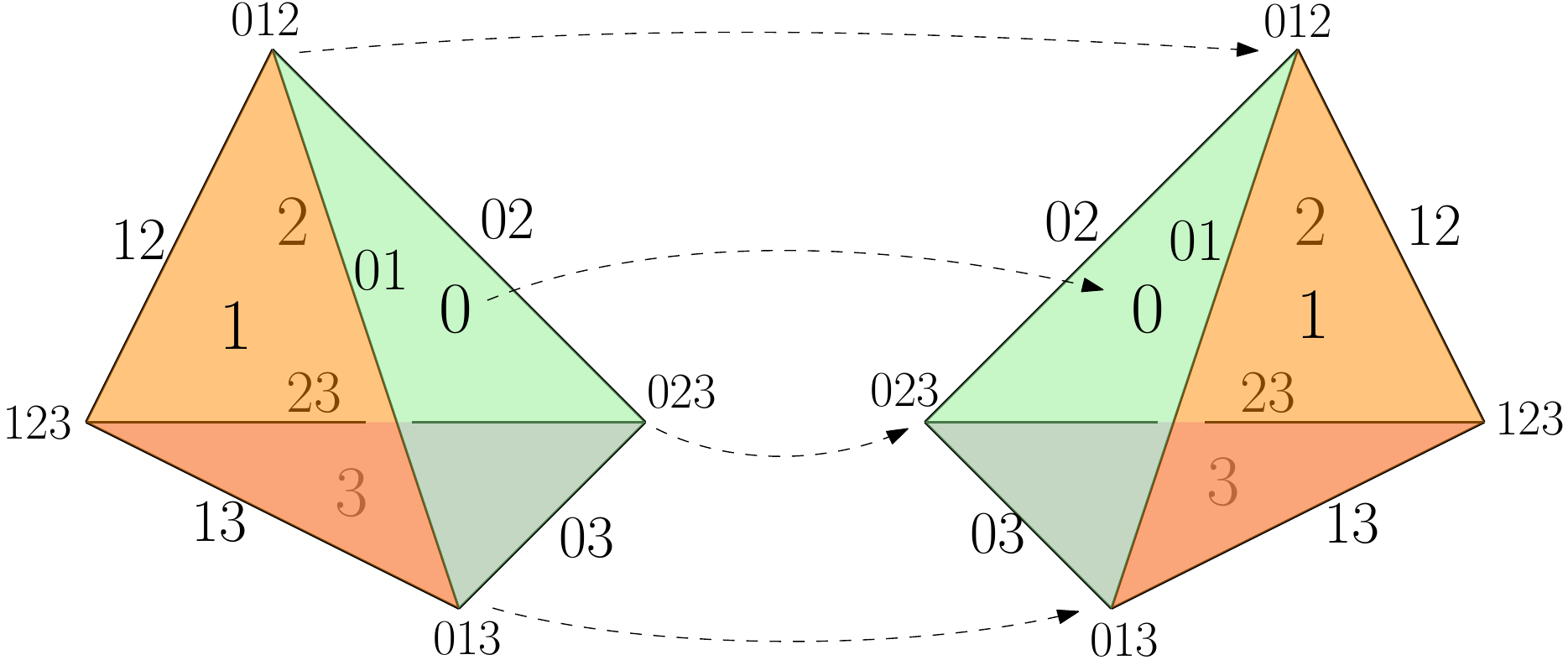} \end{array}
\end{equation}

A {\bf colored triangulation} $T$ is a connected gluing of colored $d$-simplices where all $(d-1)$-subsimplices are shared by two $d$-simplices. The canonical gluing rule for colored simplices ensures that there is a unique gluing between two colored simplices as soon as the color of their common $(d-1)$-subsimplex is specified. Therefore, a colored triangulation is completely determined by the data of which simplex is connected to which other simplex using which color. This data is fully encoded into a graph $G(T)$ with colored edges: each vertex of $G(T)$ corresponds to a $d$-simplex of $T$ and there is an edge of color $c\in\{0, 1, \dotsc, d\}$ between two vertices of $G(T)$ if the two corresponding $d$-simplices of $T$ are glued along a $(d-1)$-subsimplex of color $c$.

There is a well-known equivalence between the orientability of $T$ and the bipartiteness of $G(T)$. We will in this article restrict attention to the bipartite case, considering that $T$ has black and white simplices and $G(T)$ has black and white vertices.

The main reason why colored triangulations were introduced is thus that they encode PL-manifolds in a purely graphical way. The theorem reads

\begin{theorem} \label{thm:ColoredGraphs}
There is a one-to-one correspondence between $d$-dimensional, closed, connected, colored triangulations and {\bf colored graphs} defined as connected graphs whose vertices all have degree $d+1$ and each edge carries a color from $\{0, \dotsc, d\}$ such that the $(d+1)$ edges incident to each vertex have distinct colors. Moreover,
\begin{itemize}
\item If $T$ is a colored triangulation and $G(T)$ the corresponding colored graph, then $G(T)$ is obtained as the {\bf colored 1-skeleton} of the complex dual to $T$, i.e. the 1-skeleton of the dual where the edges of $G(T)$ carry the colors of their dual $(d-1)$-simplices in $T$.
\item $T$ can be reconstructed from $G(T)$ by applying the colored gluing rule to each pair of simplices represented in $G(T)$ by two vertices connected by an edge.
\item For any $k\in\{0, \dotsc, d-1\}$, let $\{c_1, \dotsc, c_{d-k}\} \subset \{0, \dotsc, d\}$ be a subset of colors and $G(c_1, \dotsc, c_{d-k})\subset G(T)$ be the subgraph of $G(T)$ obtained by keeping all the vertices of $G(T)$ and the edges of colors $c_1, \dotsc, c_{d-k}$ while removing the other edges. There is a one-to-one correspondence between the $k$-simplices of $T$ with color label $\{c_1, \dotsc, c_{d-k}\}$ and the connected components of $G(c_1, \dotsc, c_{d-k})$.
\end{itemize}
\end{theorem}

The connected components of $G(c_1, \dotsc, c_{d-k})$ are usually called $(d-k)$-bubbles. We will soon focus on a special case of $d$-bubbles obtained by removing the color 0. We will nevertheless use the full notion of bubbles in Section \ref{sec:Topology} to investigate topology.

\begin{proof}
The first two items are trivial. The correspondence between closed, connected colored triangulations and connected colored graphs is explained above and relies on the fact that the gluing of two colored $d$-simplices is entirely determined by the color of the $(d-1)$-simplex they share. Obviously, representing $d$-simplices by vertices and their connectivity by edges produces the 1-skeleton of the dual. The second item of the theorem is also obvious and just details the correspondence.

The third item is the most interesting and it explains why those colored graphs are relevant to study the combinatorics and topology of colored triangulations. Each $k$-dimensional subsimplex of a $d$-simplex is identified by a $(d-k)$-uple of colors. Say if $\sigma$ is a $d$-simplex, denote $\sigma(c_1, \dotsc, c_{d-k})$ the $k$-simplex with color label $\{c_1, \dotsc, c_{d-k}\}$. In $G(T)$, $\sigma$ is represented as a vertex $v_{\sigma}$. Its boundary $(d-1)$-simplices $\sigma(c)$ are represented as half-edges of color $c=0, \dotsc, d$ incident to $v_{\sigma}$. It comes that $\sigma(c_1, \dotsc, c_{d-k})$ in $\sigma$ is identified by the $(d-k)$-uple of half-edges which carry the colors $c_1,\dotsc, c_{d-k}$.

The gluing rule for colored simplices is that when $\sigma$ and $\sigma'$ are glued along a $(d-1)$-simplex $\sigma(c)$ of color $c$, they identify two by two their subsimplices whose color labels contain $c$. Therefore, when $\sigma(c_1, \dotsc, c_{d-k})\subset \sigma$ is identified with $\sigma'(c_1, \dotsc, c_{d-k})\subset \sigma'$, it translates in $G(T)$ into the fact that when $v_\sigma$ and $v_{\sigma'}$ are connected by an edge of color $c \in\{c_1, \dotsc, c_{d-k}\}$, the half-edges of colors $c_1, \dotsc, c_{d-k}$ incident to $v_{\sigma}$ and $v_{\sigma'}$ represent the same $k$-simplex of $T$.

Denote $G(c_1, \dotsc, c_{d-k})$ the subgraph of $G(T)$ which only retains the edges of colors $c_1, \dotsc, c_{d-k}$. A connected component of $G(c_1, \dotsc, c_{d-k})$ thus represents a $k$-simplex $\sigma(c_1, \dotsc, c_{d-k})$ of $T$. Moreover, two connected components represent different $k$-simplices of $T$. Indeed, the only way for a $k$-simplex with colors $\{c_1, \dotsc, c_{d-k}\}$ to be shared by two $d$-simplices $\sigma, \sigma'$ is that they are glued along a $(d-1)$-simplex of color $c\in\{c_1, \dotsc, c_{d-k}\}$. This is equivalent to $v_{\sigma}$ and $v_{\sigma'}$ being connected by an edge of color $c$ in $G(T)$.
\end{proof}

In two dimensions, the triangles of $T$ are represented as the vertices of $G(T)$, the edges of $T$ as the edges of $G(T)$ and the vertices of $T$, each having a pair of colors, as the bicolored connected subgraphs, i.e. the bicolored cycles of $G(T)$. Notice that in two dimensions, it is more common to consider the dual to $T$ to be a {\bf combinatorial map} $M(T)$. A combinatorial map is a graph equipped with a rotation system, i.e. a cyclic ordering of the edges incident to each vertex. This defines a notion of faces, which in the case of $M(T)$ are polygons dual to the vertices of $T$. The difference between $G(T)$ and $M(T)$ is thus that the vertices of $T$ are represented as bicolored cycles in $G(T)$ and as faces in $M(T)$, but they have the same set of vertices and edges. In fact, $M(T)$ can be obtained by a canonical embedding of $G(T)$ which transforms bicolored cycles to faces.

\begin{corollary} \label{cor:2D}
In two dimensions, the graph $G(T)$ has a canonical embedding as a combinatorial map $M(T)$ such that $M(T)$ is the dual to the colored triangulation $T$, the bicolored cycles of $G(T)$ are the faces of $M(T)$.
\end{corollary}

\begin{proof}
Obviously $G(T)$ and $M(T)$ have the same 1-skeleton, which is the dual graph. Each face of $M(T)$ can be labeled unambiguously with a pair of colors $\{0,1\}, \{0, 2\}$ or $\{1,2\}$: it is the pair of colors labeling the dual vertex in $T$. $M(T)$ can thus be obtained by a canonical embedding of $G(T)$: around each white vertex of $G(T)$, set the cyclic order of the three incident edges to be the colors $(012)$ and set the cyclic order around each black vertex to be the colors $(102)$, like
\begin{equation}
\begin{array}{c} \includegraphics[scale=.35]{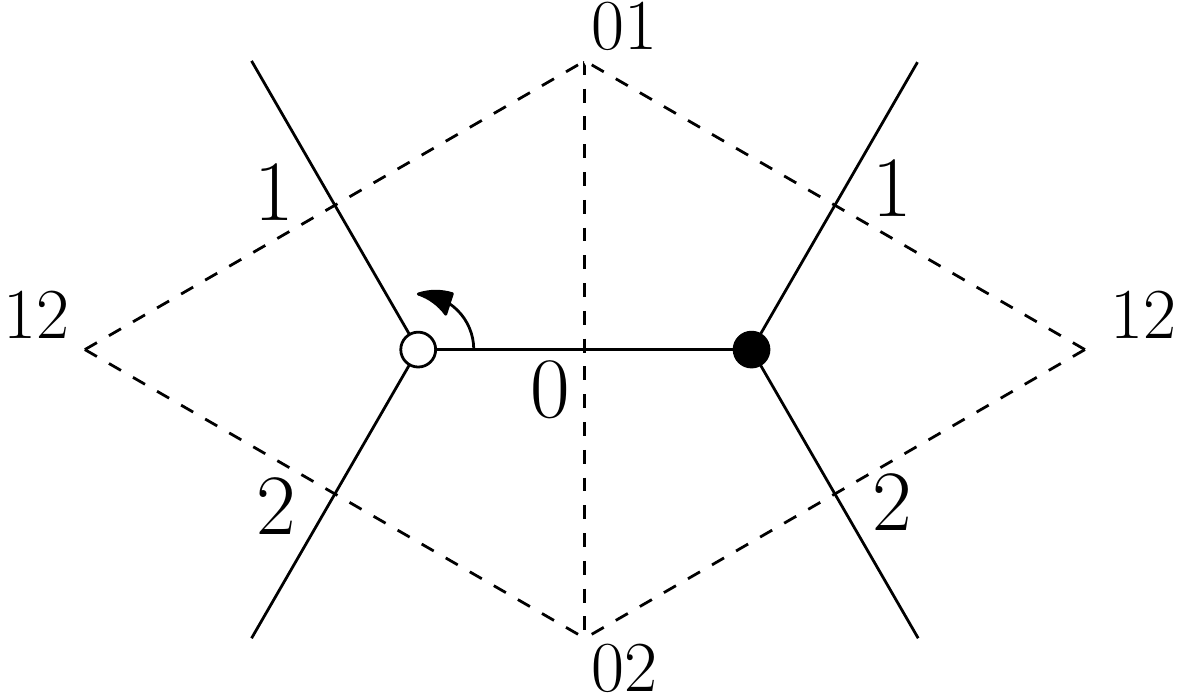} \end{array}
\end{equation}
using the counterclockwise convention. This way, $G(T)$ becomes a map whose faces are the bicolored cycles. Therefore this canonical embedding turns $G(T)$ and its bicolored cycles into $M(T)$.
\end{proof}

In three dimensions, the tetrahedra of $T$ are represented as the vertices of $G(T)$, the triangles of $T$ as the edges of $G(T)$, the edges of $T$ say of colors $\{a, b\}$ as the bicolored cycles with colors $\{a, b\}$ of $G(T)$ and the vertices of $T$ with colors $\{a, b, c\}$ as the connected components of the subgraph of $G(T)$ with the colors $a, b, c$.

The fundamental theorem of colored triangulations in topology is that they represent PL-pseudomanifolds. In two dimensions, there are only manifolds, but in three dimensions, if a connected component of a subgraph $H\subset G(T)$ with 3 colors has its canonical embedding which is a surface of non-zero genus, it means that it represents a vertex in $T$ whose neighborhood is not a 3-ball and has a conical singularity. The second important theorem for topology is that every manifold admits a representation as a colored triangulation (e.g. by barycentric subdivision of a non-colored one).

At the purely combinatorial level, the fundamental theorem is Gurau's theorem which is a combinatorial extension of the genus of a map to any $d$-dimensional colored triangulations.

\begin{theorem}\cite{1/NExpansion} \label{thm:Gurau}
Gurau's degree defined as
\begin{equation}
\omega(T) = d + \frac{d(d-1)}{4} \Delta_d(T) - \Delta_{d-2}(T)
\end{equation}
is a non-negative integer for any colored triangulation $T$, where $\Delta_k(T)$ is the number of $k$-dimensional simplices of $T$.
\end{theorem}

It is easy to check that $\omega(T)$ is equivalent to the genus of $M(T)$ in two dimensions. In higher dimensions, it is not a topological invariant, but still a genuine extension of the genus in combinatorial terms. It is a bound on the number of $(d-2)$-simplices which grows linearly with the number of $d$-simplices. This way, Gurau's degree classifies colored triangulations and the classification has been performed by Gurau and Schaeffer in \cite{Gurau-Schaeffer}.

In any dimensions, triangulations which maximize the number of $(d-2)$-simplices at fixed number of $d$-simplices are those of vanishing Gurau's degree. In two dimensions, they are the planar ones, homeomorphic to the sphere. For any $d\geq 3$ however, only melonic triangulations satisfy $\omega(T) = 0$, \cite{Melons}. Melonic triangulations are easily described as melonic graphs using the correspondence with colored graphs. They are also in bijection with $(d+1)$-ary trees. They converge as metric spaces to continuous random tree \cite{MelonsBP}. We will explain what they look like below.

%%%%%%%%%%%%%%%%%%%
\subsection{Colored building blocks and bubbles} \label{sec:Bubbles}
%%%%%%%%%%%%%%%%%%%

Combinatorial maps can be defined as gluings of polygons along their edges. The polygons are the faces of the map. Universality can be studied by comparing the asymptotic properties of families with different sets of allowed polygons. It is well-known that planar triangulations (all faces have degree three), $p$-angulations (all faces have degree $p$) and more generally all families of planar maps with a finite set of allowed polygons lie in the same universality class known as the universality class of pure 2D quantum gravity \cite{MatrixReview}.

Let us first show that the use of colored triangulations in 2D, as opposed to non-colored, allows for studying universality with the only additional constraint of bipartiteness. First, we build a bipartite map from a colored triangulation $T$. Consider a vertex $v_{12}$ of colors $\{1,2\}$ in $T$. It is the intersection of say $2p$ triangles which are glued along their edges of color 1 and color 2. This gluing of $2p$ triangles is homeomorphic to a disc and has the edges of color 0 as the boundary edges. We can associate to this disc a $2p$-gon and do so for all vertices of colors $\{1, 2\}$,
\begin{equation} \label{Polygon12}
\begin{array}{c} \includegraphics[scale=.4]{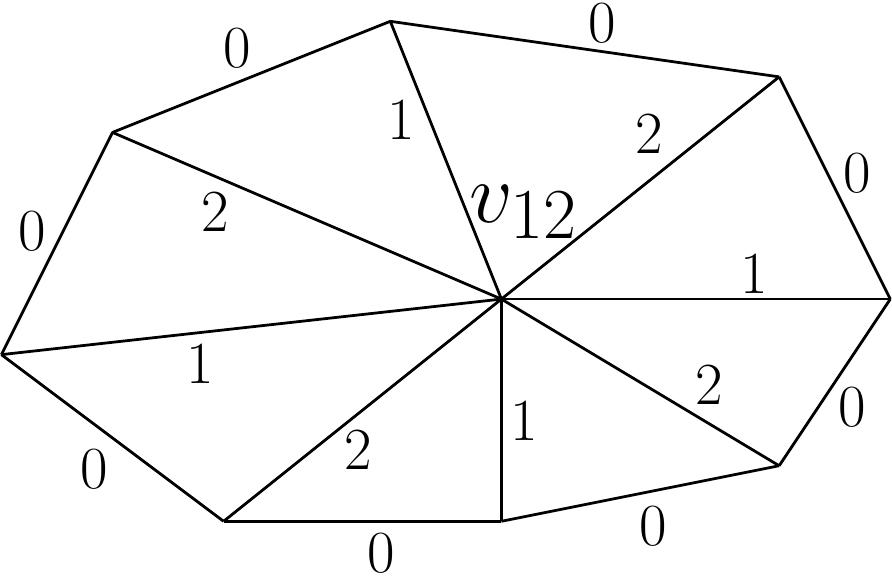} \end{array} \hspace{2cm} \leftrightarrow \hspace{2cm} \begin{array}{c} \includegraphics[scale=.4]{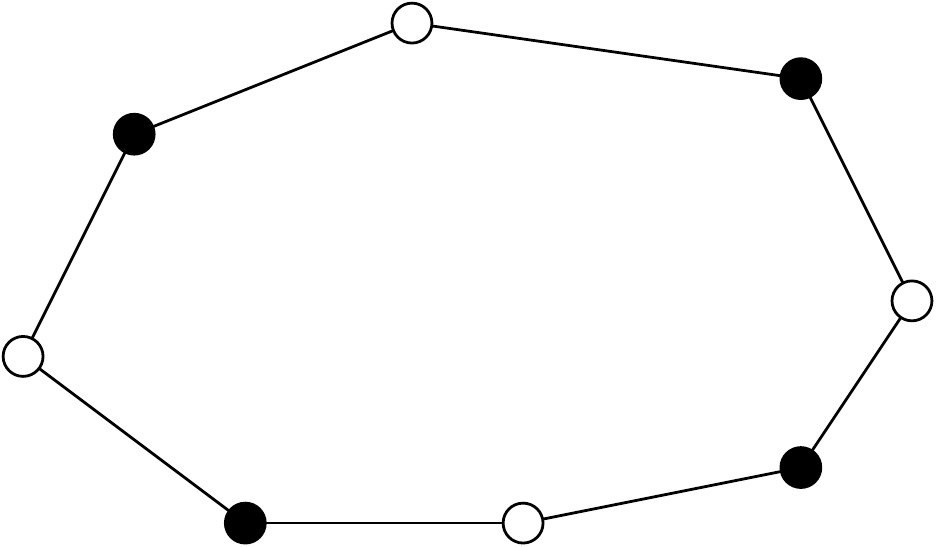} \end{array}
\end{equation}
where the vertices with colors $\{0,1\}$ become white vertices and those with colors $\{0,2\}$ become black vertices.  Since the edges of those polygons are exactly the edges of color 0 of $T$, it means that $T$ encodes a unique gluing between the edges of those polygons. We get a map which is bipartite because its vertices are the vertices of $T$ with colors $\{0,1\}$ and $\{0,2\}$. The inverse operation is simply to label the vertices of the bipartite map with colors $\{0,1\}$ and $\{0,2\}$ and its edges with the color 0, then divide each face of degree $2p$ into $2p$ colored triangles thus adding the edges of colors 1 and 2.

In the colored graph representation, the $2p$-gon is represented by a cycle alternating colors 1 and 2: the bicolored cycle dual to $v_{12}$. The color 0 lies on the boundary of the $2p$-gon: it is not necessary to describe the structure of the polygon, but only to describe the gluings between the polygons. Equivalent to building a bipartite map by gluing polygons, we can thus think dually of the colored graph as a collection of cycles with colors 1, 2 which are connected together by adding the edges of color 0, such that each vertex has exactly the three colors, see Figure \ref{fig:Cycle12}.
\begin{figure}
\includegraphics[scale=.45]{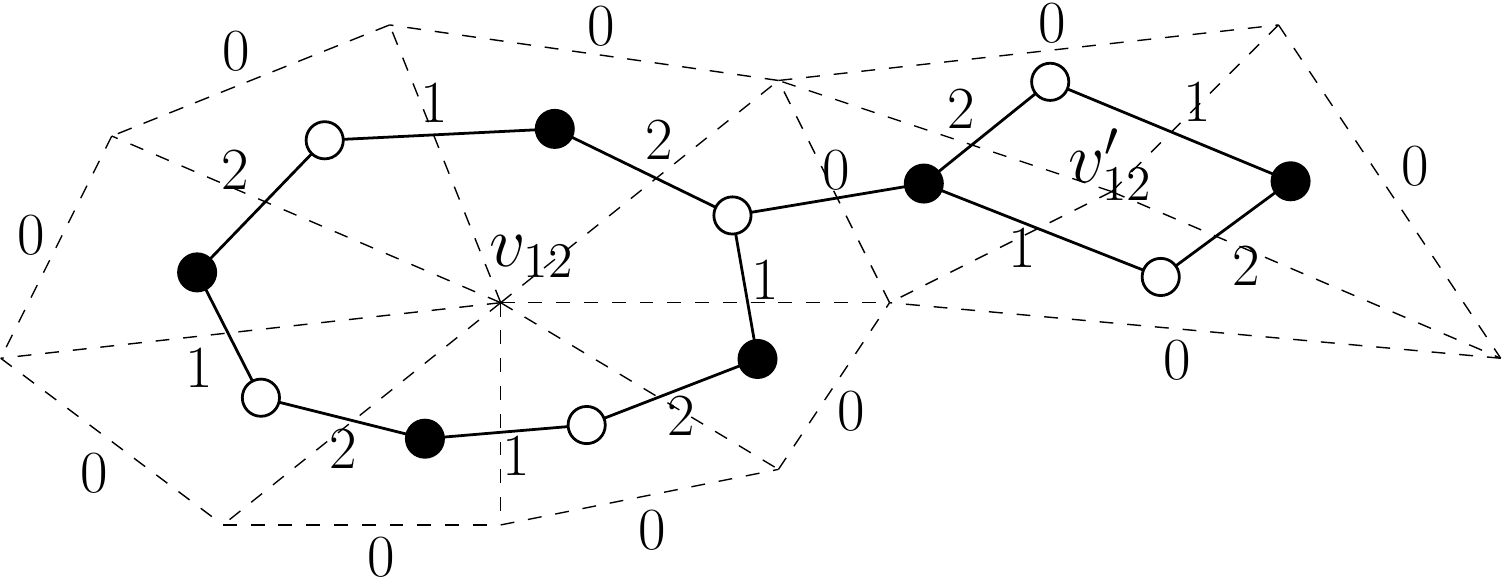}
\caption{\label{fig:Cycle12} The bicolored cycles dual to vertices with colors $\{1,2\}$ also represent polygons. The dual edges of color 0 perform the gluings of those polygons.}
\end{figure}

In higher dimensions, instead of allowing only some polygons, we want to only allow some building blocks. Colored triangulations offer a very natural set of building blocks in any dimension $d\geq 2$.

\begin{definition} [Colored building blocks]
A colored building block (CBB) is a connected colored triangulation with a boundary which consists only in $(d-1)$-simplices of color 0 and such that all $(d-1)$-simplices of color 0 lie on the boundary.
\end{definition}

\begin{figure}
\subfloat[The octahedron as a CBB]{\includegraphics[scale=.5]{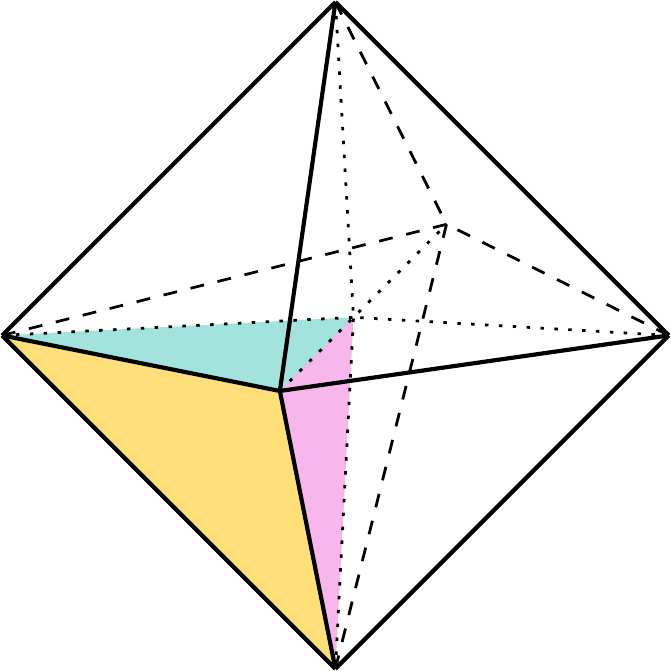}}
\hspace{4cm}
\subfloat[The bubble dual to the octahedron]{\includegraphics[scale=.6]{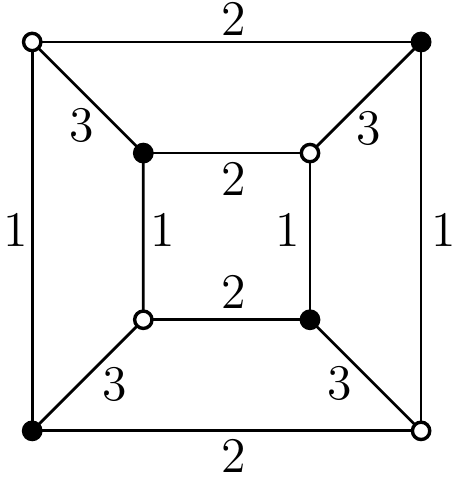}}
\caption{\label{fig:Octahedron} The octahedron is a CBB made of eight tetrahedra, and eight boundary triangles of color 0. Its bubble is the colored 1-skeleton of the dual with the color 0 removed, or equivalently the colored 1-skeleton of the dual to the boundary triangulation. The color 0 is only necessary for the gluing to other bubbles.}
\end{figure}

\begin{proposition} [Bubbles] \label{prop:CBB}
Colored building blocks satisfy the following properties.
\begin{itemize}
\item CBBs are in bijection with colored graphs with $d$ colors which we call {\bf bubbles} (as opposed to graphs with $d+1$ colors). The bubble corresponding to a CBB is the colored 1-skeleton of the dual with the edges of color 0 removed and also is the colored 1-skeleton of the dual to the boundary triangulation.
\item A CBB is the cone over its boundary triangulation.
\item Any closed colored triangulation can be obtained by gluing CBBs along their boundary subsimplices. In terms of dual colored graphs, any colored graph with colors $\{0, \dotsc, d\}$ can be obtained from a collection of bubbles $B_1, \dotsc, B_N$ with colors $\{1, \dotsc, d\}$ connected by adding edges of color 0, respecting bipartiteness and such that each vertex of the graph has an incident edge of color 0.
\end{itemize}
\end{proposition}

\begin{proof}
We start with the first item. We can construct the 1-skeleton of the dual to a CBB just like for closed colored triangulations, except for the boundary. It has a vertex $v_\sigma$ for every $d$-simplex $\sigma$ and an edge of color $c\in\{1, \dotsc, d\}$ between $v_\sigma$ and $v_{\sigma'}$ if $\sigma$ and $\sigma'$ are glued along a $(d-1)$-simplex of color $c$. Because the $(d-1)$-simplices of color 0 are not glued, they become half-edges of color 0 incident to each vertex of the 1-skeleton and those half-edges are not connected. Clearly, those half-edges of color 0 are irrelevant and can be removed without losing information. This way, we get a connected colored graph $B$ with colors in $\{1, \dotsc, d\}$ called a bubble.

Let us compare $B$ with the 1-skeleton $B'$ of the dual to the boundary triangulation and show that $B=B'$. Each $d$-simplex has a single $(d-1)$-subsimplex of color 0, and the latter must lie on the boundary of the CBB. The other way around, each boundary simplex is a $(d-1)$-subsimplex of color 0 of a $d$-simplex. Therefore there is a bijection between the $d$-simplices of the CBB and the $(d-1)$-simplices of its boundary triangulation. Both objects are represented by vertices in their dual graphs, which thus have the same set of vertices. If $\sigma$ is a $d$-simplex, we denote $\sigma_0$ its $(d-1)$-subsimplex of color 0 on the boundary and $v_\sigma$ the vertex dual to either one of them.
%it gives rise to a boundary $(d-1)$-simplex which in the dual $B'$ is represented by a vertex, just like $\sigma_d$ in $B$. Obviously each vertex of $B'$ is dual to a $(d-1)$-simplex of color 0 which belongs to a $d$-simplex. Therefore the vertices of $B$ and $B'$ are the same.

If $\sigma_c$ is a $(d-1)$-simplex of color $c\in\{1, \dotsc, d\}$ glued between $\sigma$ and $\sigma'$, it is represented in $B$ as an edge $e_c$ of color $c$ between $v_{\sigma}$ and $v_{\sigma'}$. Moreover, $\sigma_c$ intersects a $(d-1)$-simplex $\sigma_0$ of color 0 in $\sigma$ and a $(d-1)$-simplex $\sigma'_0$ of color 0 in $\sigma'$. This intersection is a $(d-2)$-simplex of colors $\{0, c\}$ which lies on the boundary of the CBB between $\sigma_0$ and $\sigma'_0$. It is thus also represented by an edge of color $c$ between $v_{\sigma}$ and $v_{\sigma'}$ in $B'$, just like $e_c$ in $B$. This shows $B=B'$ and also that the boundary is connected if and only if the CBB is. This is illustrated in 3D in Figure \ref{fig:GluedTetDuals}.

\begin{figure}
\includegraphics[scale=.55]{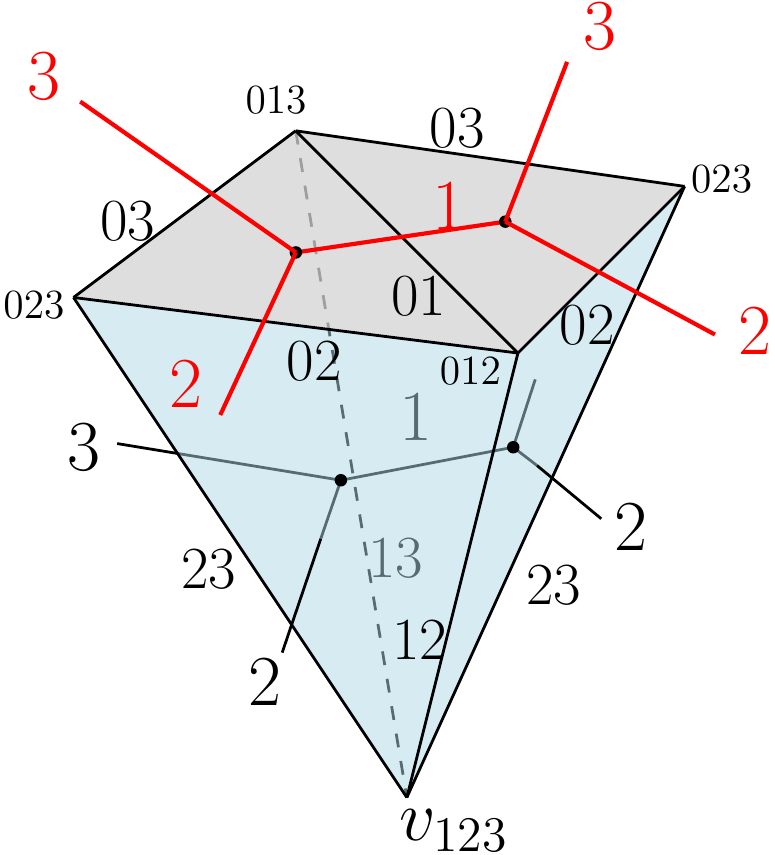}
\caption{\label{fig:GluedTetDuals} Two tetrahedra glued along their triangle of color 1. The triangles of color 0 are in gray and those of color 2 in blue (while those of color 3 are in the back). The 1-skeleton of the dual is the same as the 1-skeleton in red which is that of the dual to the boundary formed by the triangles of color 0. Clearly, one can also start from the boundary triangles in gray, take the cone and label a simplex of $k$ with the same colors as the simplex of dimension $k-1$ it comes from, add the color 0 to all color labels on the boundary and obtain this way a CBB.}
\end{figure}

We now prove the second item. As a preliminary, notice that there is a single connected component with all the colors except 0, which means (third item of Theorem \ref{thm:ColoredGraphs}) that the CBB has a single vertex $v_{1\dotsb d}$ with label $\{1, \dotsc, d\}$. All other vertices of the CBB have the color 0 in their label and thus sit on the boundary.

Let us use the notation $\sigma(c_1, \dotsc, c_k)$ for simplices with colors $\{c_1, \dotsc, c_k\}$ of a $d$-dimensional triangulation and $\sigma_0(c_1, \dotsc, c_k)$ for simplices with colors $\{c_1, \dotsc, c_k\}$ of a $(d-1)$-dimensional triangulation with colors in $\{1, \dotsc, d\}$. First, observe that the topological cone over a colored $(d-1)$-simplex $\sigma_0$ is a colored $d$-simplex $\sigma$. Indeed, $\sigma$ is clearly a $d$-simplex and its coloring is as follows. $\sigma_0$ is its boundary $(d-1)$-simplex of color 0. The $(d-2)$-simplices of $\sigma_0$ with colors $c=1, \dotsc, d$ give rise to the $(d-1)$-simplices of $\sigma$ with the same colors. $\sigma$ is thus a simplex with colored $(d-1)$-simplices, but we also need to make sure that the colorings induced on its subsimplices are consistent with the colorings of the subsimplices of $\sigma_0$. This might seem quite obvious, see Figure \ref{fig:GluedTetDuals} for instance, but let us make it formal. 

First add the color 0 to each color label of the $d$-dimensional triangulation. The subsimplices $\sigma_0(c_1, \dotsc, c_k)$ are thus given the color label $\{0, c_1, \dotsc, c_k)$ in $\sigma$ to identify them uniquely. If $\sigma_0(c, c_1, \dotsc, c_k)$ and $\sigma_0(c',c_1, \dotsc, c_k)$ are two subsimplices of $\sigma_0$, they intersect on $\sigma_0(c, c',c_1, \dotsc, c_k)$ -- for example, the edges of colors $\{0,1\}$ and $\{0,2\}$ intersect on a vertex of colors $\{0,1,2\}$ in a triangle of color 0 of Figure \ref{fig:GluedTetDuals}. Moreover they give rise upon coning to $\sigma(c,c_1, \dotsc, c_k)$ and $\sigma(c',c_1, \dotsc, c_k)$ with one more dimension in $\sigma$ -- two triangles with colors 1 and 2 in the example -- while $\sigma_0(c,c',c_1, \dotsc, c_k)$ gives rise to $\sigma(c,c',c_1, \dotsc, c_k)$ also with one more dimension in $\sigma$ -- an edge of colors $\{1,2\}$ in the example. The latter is obviously the intersection of $\sigma(c,c_1, \dotsc, c_k)$ and $\sigma(c',c_1, \dotsc, c_k)$ in $\sigma$ -- the intersection of the triangles of colors 1 and 2 in the example.

Due to the bijection between the $d$-simplices of a CBB and its boundary $(d-1)$-simplices, the coning accounts for all $d$-simplices of a CBB.
%Any $k$-simplex $\sigma(c_1, \dotsc, c_{d-k})$ of a CBB such that $k\geq 1$ and $c_i\neq 0$ has a boundary $(k-1)$-simplex $\sigma(0, c_1, \dotsc, c_{d-k})$ which has the color 0 in its label and thus lies on the boundary of the CBB. The only case where this is not true is the vertex $v_{1 \dotsb d}$ which, as we proved already, is unique. Therefore each $k$-simplex of a CBB, $k\geq 1$, is incident to a $(k-1)$-simplex on the boundary. We can use this to reconstruct the CBB from its boundary as a topological cone.
%
%, since each $(d-k-1)$-subsimplex $\sigma_0(c_1, \dotsc, c_k)\subset \sigma_0$ gives rise by coning to a $(d-k)$-subsimplex of $\sigma$ with the same colors. According to the above remark, the topological cone over the simplices of the $(d-1)$-dimensional triangulation accounts for all simplices of a CBB.
Therefore, we just have to make sure that the colored gluing of two $(d-1)$-simplices induces a consistent colored gluing of their cones. Again this might seem obvious from the Figure \ref{fig:GluedTetDuals} in three dimensions but we can make it formal for all dimensions.

Consider two $(d-1)$-simplices $\sigma_0$ and $\sigma'_0$ which are glued together along a $(d-2)$-simplex of color $c\in\{1, \dotsc, d\}$. It means that they share all their subsimplices $\sigma_0(c, c_1, \dotsc, c_k)$ whose color labels contain the color $c$. In the topological cone, $\sigma_0$ and $\sigma'_0$ give rise to colored $d$-simplices $\sigma$ and $\sigma'$. Moreover, their common subsimplices $\sigma_0(c, c_1, \dotsc, c_k)$ correspond to subsimplices $\sigma(0, c, c_1, \dotsc, c_k)$ on the boundary of the CBB with the additional label 0. In the bulk, each $\sigma(0, c, c_1, \dotsc, c_k)$ gives rise by coning to a subsimplex $\sigma(c, c_1, \dotsc, c_k)$ with one more dimension. This implies that $\sigma$ and $\sigma'$ share all their subsimplices whose color labels contain $c$, except possibly for the one subsimplex which is not on the boundary, i.e. the vertex $v_{1\dotsc d}$. The latter is in fact the ``tip'' of the cone. The colored gluing rule is thus reproduced by the coning.

Finally, let us prove the last item of the proposition using the colored graph representation. We just have to show that each colored graph with $d+1$ colors has a (unique) set of bubbles. Consider the 1-skeleton $G$ dual to a closed, connected, colored triangulation. It is a connected, bipartite, colored graph with edges colored in $\{0, 1, \dotsc, d\}$ and such that each vertex has all colors incident exactly once. Removing the edges of color 0, one gets a collection of connected colored subgraphs with all the colors except 0, and which inherits its bipartiteness from that of the whole triangulation. Those connected subgraphs are the bubbles. The graph $G$ can then be reconstructed by considering the bubbles and adding the edges of color 0, such that each vertex has an incident edge of color 0 and respecting the bipartiteness of the bubbles.
\end{proof}

We here illustrate the coning in two and three dimensions.
 
In two dimensions, a CBB is a set of $2p$ triangles glued as in the left hand side of \eqref{Polygon12}. Dropping the color 0 on the boundary, we see that we get a 1-dimensional colored triangulation, with vertices of colors 1 and 2 and bipartiteness of the edges is inherited from bipartiteness of the triangles.

The other way around, starting from a 1-dimensional colored triangulation, we can add the color 0 to the vertex labels and to the edges, add a vertex $v_{12}$ and take the cone between the 1-dimensional triangulation and $v_{12}$. This produces a two-dimensional CBB where the edges of color $c=1, 2$ connect the vertices of labels $\{0,c\}$ to $v_{12}$. This is clearly the reverse operation to restricting to the boundary of the $2p$-gon.

In three dimensions, a CBB has a single vertex $v_{123}$ with color label $\{1, 2, 3\}$ and it is shared by all tetrahedra. Edges of color labels $\{a, b\}$ for $a, b =1, 2, 3$ do not lie on the boundary; each has one vertex which is $v_{123}$ and another vertex with label $\{0, a, b\}$ on the boundary. Edges with color labels $\{0, a\}$ lie on the boundary and connect vertices with label $\{0, a, b\}$ to $\{0, a, c\}$. All triangles of color 0 are on the boundary and since each tetrahedron has a unique triangle of color 0, there is a one-to-one correspondence between boundary triangles and tetrahedra of the CBB. Bipartiteness of the boundary triangles thus follows the bipartiteness of the tetrahedra. This way, we see that removing the color 0 from all boundary labels produces a 2-dimensional colored triangulation with colors $1, 2, 3$.

The other way around, we can consider a 2-dimensional colored triangulation $T_2$ with colors 1, 2, 3, add a vertex $v_{123}$ and take the cone between the triangulation and $v_{123}$ to obtain a CBB $T_3$. This way, all vertices of $T_2$ with labels $\{a, b\}$ give rise to edges in $T_3$ with the same labels, all edges of $T_2$ of color $a = 1, 2, 3$ give rise to triangles in $T_3$ with the same colors. Finally, each triangle in $T_2$ gives rise to a tetrahedron in $T_3$.

In the rest of the article we will use the dual colored graph representation of CBBs, i.e. bubbles. The simplest bubble is the one with two vertices connected by all edges with colors $c=1, \dotsc, d$, Figure \ref{fig:2VertexBubble}. Bubbles with four vertices are characterized by an integer $q$ which is the number of parallel edges, with colors $c_1, \dotsc, c_q$ between a white and a black vertex, Figure \ref{fig:4VertexBubble}. By symmetry, $q \in\{1, \dotsc, (d-1)/2\}$ if $d$ is odd and $q\in\{1, \dotsc, d/2\}$ if $d$ is even. Notice that for $d=3$, only $q=1$ is possible, Figure \ref{fig:4VertexBubble3d}. Due to color relabeling, this leaves three distinct bubbles with four vertices at $d=3$.

\begin{figure}
\subfloat[\label{fig:2VertexBubble}]{\includegraphics[scale=.35]{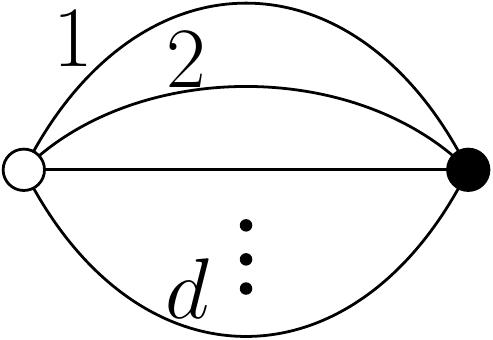}}
\hspace{2cm}
\subfloat[\label{fig:4VertexBubble}]{\includegraphics[scale=.35]{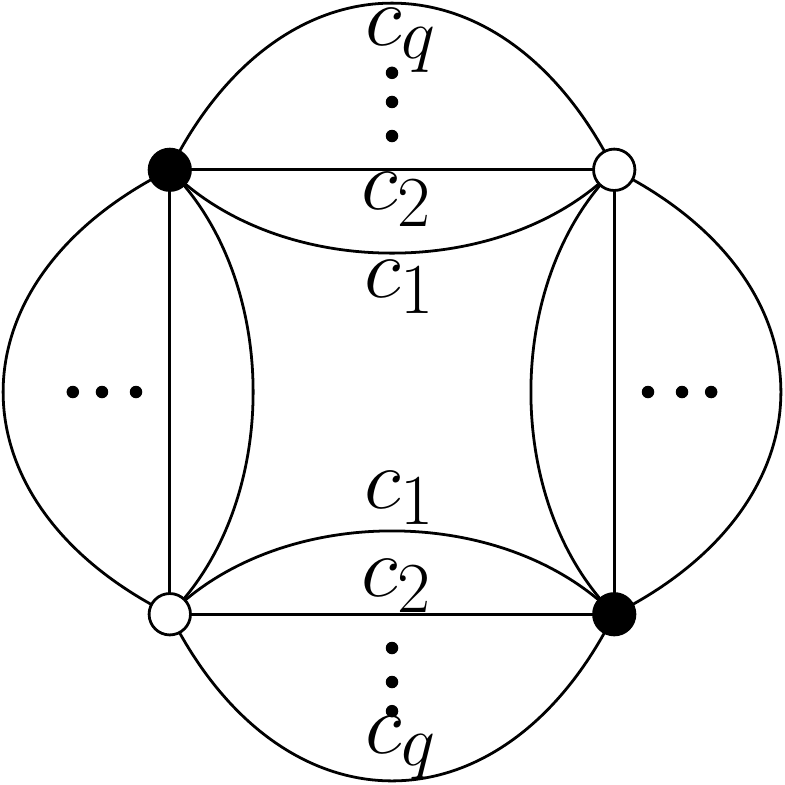}}
\hspace{2cm}
\subfloat[\label{fig:4VertexBubble3d}]{\includegraphics[scale=.35]{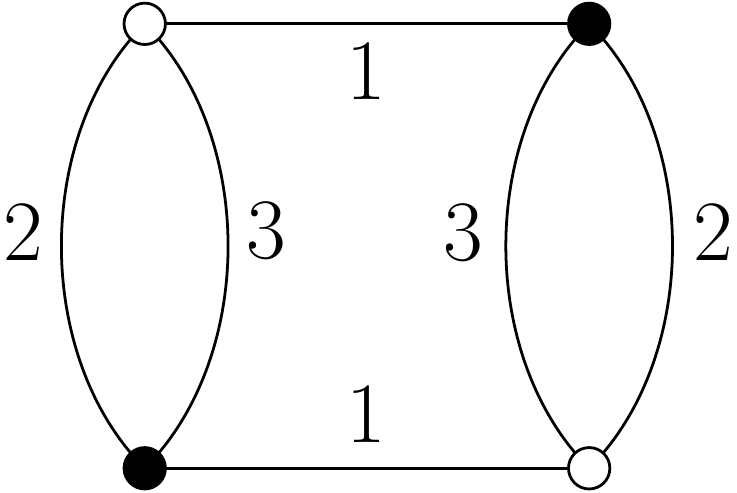}}
\hspace{2cm}
\subfloat[\label{fig:6VertexBubble2d}]{\includegraphics[scale=.35]{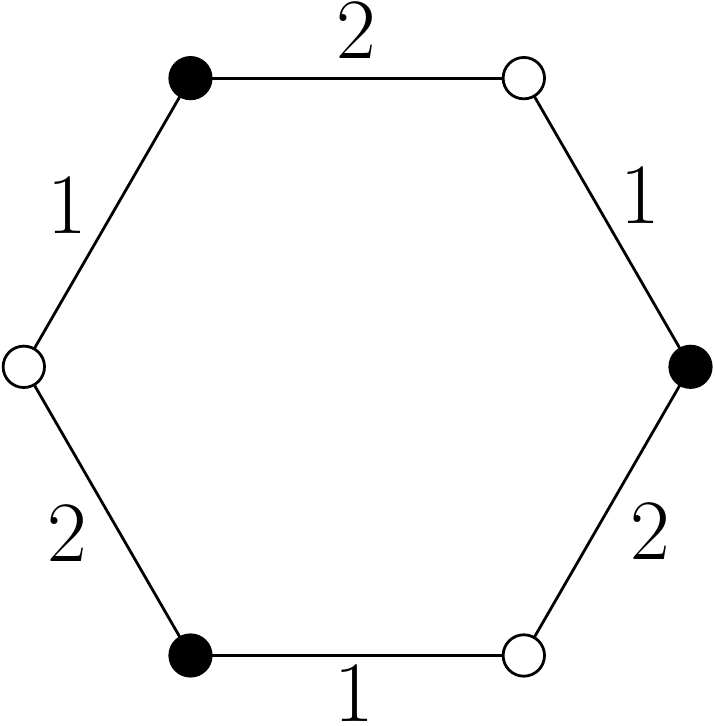}}
\caption{\label{fig:Bubbles}Examples of bubbles: \ref{fig:2VertexBubble} the 2-vertex bubble, representing the $d$-ball; \ref{fig:4VertexBubble} the 4-vertex bubbles characterized by $q\leq d/2$ and representing the $d$-ball; \ref{fig:4VertexBubble3d} only $q=1$ is possible in 3D with 4 vertices, with three possible colorings; \ref{fig:6VertexBubble2d} the 2D bubble of length 6, representing a hexagon}
\end{figure}

At $d=2$, bubbles are cycles alternating the colors 1 and 2 and are thus characterized by a single integer, the length $2p$ of the cycle (which represents a $2p$-gon, as we have seen), e.g. Figure \ref{fig:6VertexBubble2d}. However, at $d=3$, bubbles are labeled by colored boundary triangulations which cannot be characterized by just an integer anymore. In particular, there is not a single bubble at fixed number of vertices like for $d=2$. Up to color relabeling, the three-dimensional bubbles with six vertices are the following.
\begin{equation} \label{6VertexBubbles3d}
\begin{array}{c} \includegraphics[scale=.3]{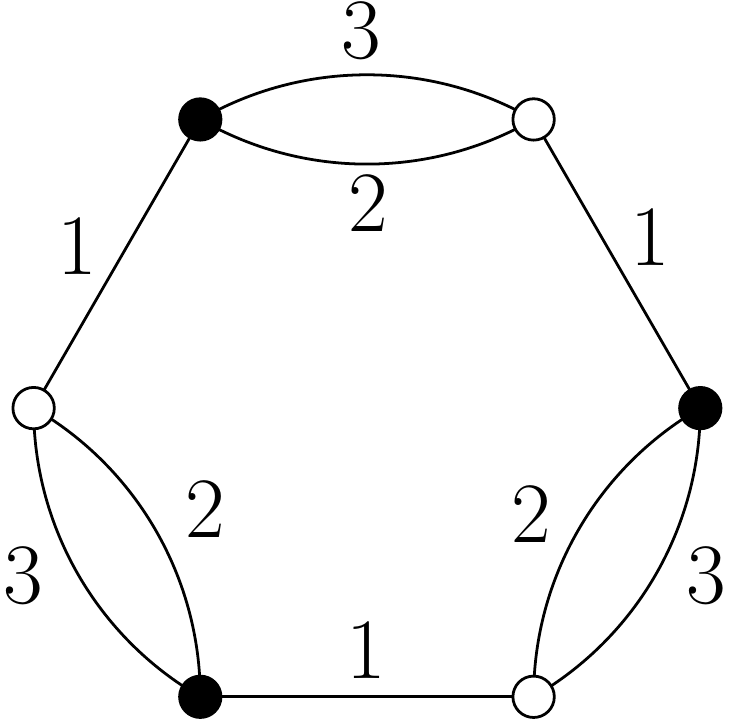} \end{array} \hspace{1.5cm}  \begin{array}{c} \includegraphics[scale=.3]{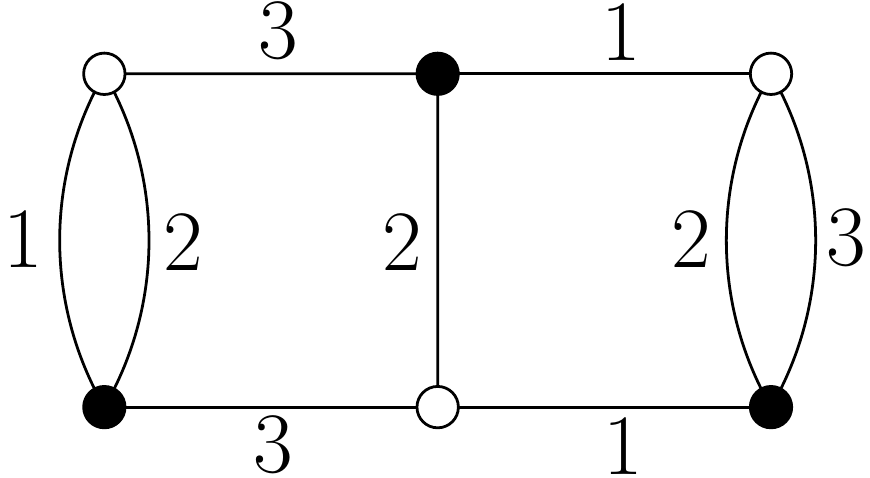} \end{array} \hspace{1.5cm} \begin{array}{c} \includegraphics[scale=.3]{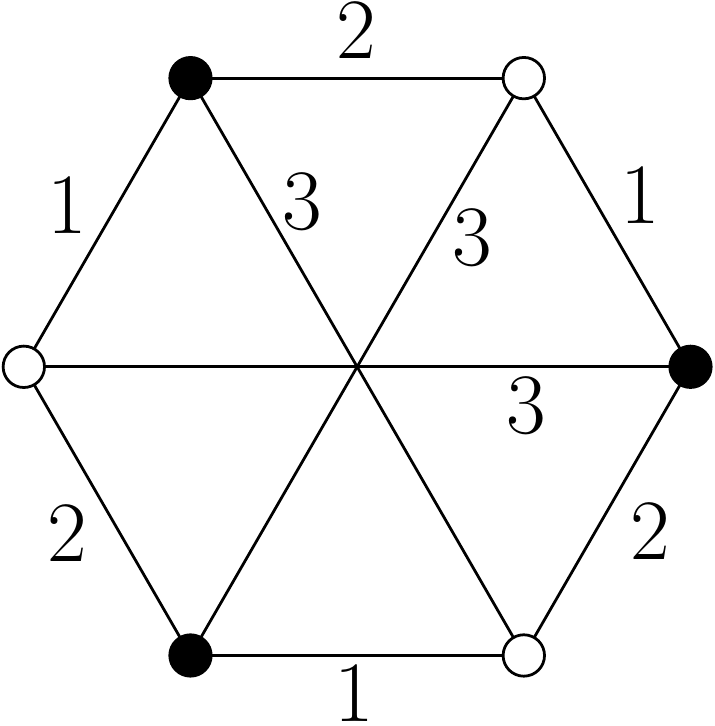} \end{array}
\end{equation}
Since we have used the cyclic order $(123)$ around white vertices and $(132)$ around black vertices, these bubbles should really be seen as maps. In particular, they are planar, except for the one with $K_{3,3}$ as underlying graph which is a torus.

\begin{corollary}[CBBs homeomorphic to 3-balls and planar bubbles] \label{cor:PlanarBubbles}
In three dimensions, a bubble $B(T)$ is a colored graph with 3 colors which is the 1-skeleton dual to a two-dimensional colored triangulation $T$. We say that the bubble $B(T)$ is {\bf planar} if the combinatorial map $M(T)$ dual to $T$ is planar. Then, a three-dimensional CBB is homeomorphic to a 3-ball if and only if its bubble is planar.
\end{corollary}

\begin{proof}
Following Proposition \ref{prop:CBB}, a three-dimensional CBB is determined by its bubble $B(T)$ which is the colored 1-skeleton of its boundary triangulation $T$. As observed in Corollary \ref{cor:2D}, the bubble $B(T)$ and the map $M(T)$ have the same 1-skeleton and the faces of $M(T)$ are the bicolored cycles of $B(T)$. Therefore, a CBB homeomorphic to a 3-ball, whose boundary is a 2-sphere, has a planar $M(T)$, and thus a planar bubble. The other way around, as shown in Proposition \ref{prop:CBB}, the CBB is the topological cone over its boundary triangulation. If the latter is a planar map, then the cone is the 3-ball.
\end{proof}

Since the bicolored cycles of a bubble $B(T)$ in three dimensions become the faces of the map $M(T)$ upon using the canonical embedding that colors are represented in the cyclic order $(123)$ around white vertices and $(132)$ around black vertices, we will use the terminology of {\bf face of colors $\{a, b\}$ and degree $p$} for a bicolored cycle with colors $\{a,b\}$ and length $p$ for planar bubbles. This is in agreement with the terminology used in the combinatorial maps literature. However, bicolored cycles with the color 0 will not be called faces.

%%%%%%%%%%%
\subsection{Bicolored cycles} \label{sec:BicoloredCycles}
%%%%%%%%%%%

Instead of considering the whole set of colored triangulations, we can focus on those which are built by gluing some CBBs out of a finite set. In two dimensions, this means studying gluings of polygons with prescribed lengths, which allows for universality checks. We will thus use this same strategy to investigate universality in higher dimensions.

%However, one immediately faces the problem of classifying colored graphs and in particular, with respect to which parameter. Gurau's theorem \ref{thm:Gurau} is a purely combinatorial extension of the genus formula which bounds the number of $(d-2)$-simplices linearly with the number of $d$-simplices. We will thus use the number of $(d-2)$-simplices. 

Let $\{K_1, \dotsc, K_N\}$ be a finite set of CBBs and $n_1, \dotsc, n_N$ some positive integers. Let $\cT_{n_1, \dotsc, n_N}(K_1, \dotsc, K_N)$ be the set of colored triangulations built from $n_i$ copies of the CBB $K_i$ for $i=1, \dotsc, N$. For $T\in \cT_{n_1, \dotsc, n_N}(K_1, \dotsc, K_N)$ we further denote $\delta_{ab}(T)$ the number of $(d-2)$ simplices of $T$ labeled with the pair of colors $\{a, b\}$ and
\begin{equation}
\cT_{n_1, \dotsc, n_N}(K_1, \dotsc, K_N | \{\delta_{ab}\}) = \left\{T \in \cT_{n_1, \dotsc, n_N}(K_1, \dotsc, K_N), \forall\ a<b \quad \delta_{ab}(T) = \delta_{ab} \right\}.
\end{equation}
Similarly, let $\{B_1, \dotsc, B_N\}$ be the set of bubbles corresponding to $\{K_1, \dotsc, K_N\}$ and $\cG_{n_1, \dotsc, n_N}(B_1, \dotsc, B_N)$ be the set of connected colored graphs with $n_i$ copies of the bubble $B_i$, for $i=1, \dotsc, N$. For $G\in \cG_{n_1, \dotsc, n_N}(B_1, \dotsc, B_N)$, we define the {\bf bicolored cycles with colors $\{a, b\}$} for $a<b$, $a, b\in\{0, \dotsc, d\}$, as the cycles of $G$ whose edges alternate the colors $a$ and $b$. The number of bicolored cycles with colors $\{a, b\}$ of $G$ is denoted $C_{ab}(G)$ and we introduce the subset of colored graphs built from the bubbles $\{B_i\}$ with a prescribed number of bicolored cycles
\begin{equation}
\cG_{n_1, \dotsc, n_N}(B_1, \dotsc, B_N | \{\delta_{ab}\}) = \left\{G \in \cG_{n_1, \dotsc, n_N}(B_1, \dotsc, B_N), \forall\ a<b \quad C_{ab}(T) = \delta_{ab} \right\}.
\end{equation}

\begin{proposition}
There is a bijection between $\cT_{n_1, \dotsc, n_N}(K_1, \dotsc, K_N | \{\delta_{ab}\})$ and $\cG_{n_1, \dotsc, n_N}(B_1, \dotsc, B_N | \{\delta_{ab}\})$ which maps each $(d-2)$-simplex of with label $\{a, b\}$ to a unique bicolored cycle with colors $\{a, b\}$ and the other way around.
\end{proposition}

\begin{proof}
This is a direct combination of Theorem \ref{thm:ColoredGraphs} with Proposition \ref{prop:CBB}. Theorem \ref{thm:ColoredGraphs} establishes the correspondence between colored triangulations and colored graphs in a way which identifies each $k$-simplex with colors $\{c_1, \dotsc, c_{d-k}\}$ with a connected component of the subgraph $G(c_1, \dotsc, c_{d-k})$ which restricts to those colors. Thus, a $(d-2)$-simplex of a colored triangulation, with a pair of colors $\{a, b\}$, corresponds in the colored graph to a connected component of $G(a, b)$. This is the subgraph with only colors $a, b$. Since it is bipartite and both colors $a$ and $b$ must be incident on each vertex exactly once, $G(a, b)$ is a disjoint union of cycles which alternate the colors $a$ and $b$. Its connected components are thus the bicolored cycles with colors $\{a, b\}$. Then Proposition \ref{prop:CBB} allows the restriction of this correspondence to colored triangulations with prescribed CBBs and colored graphs with prescribed bubbles.
\end{proof}

The proposition can obviously be extended to $k$-simplices of colors $\{c_1, \dotsc, c_{d-k}\}$ and connected components of $G(c_1, \dotsc, c_{d-k})$, but we will be exclusively interested in the number of $(d-2)$-simplices. Indeed, in two dimensions, one classifies colored triangulations using the genus. It is a bound on the number of vertices which grows linearly with the number of triangles. Gurau's theorem, i.e. Theorem \ref{thm:Gurau}, is a purely combinatorial extension of the genus formula to $d\geq 2$ which bounds the number of $(d-2)$-simplices linearly with the number of $d$-simplices. We will thus classify $\cT_{n_1, \dotsc, n_N}(K_1, \dotsc, K_N)$ according to the number of $(d-2)$-simplices at fixed numbers of CBBs $K_1, \dotsc, K_N$. This is equivalent to classifying $\cG_{n_1, \dotsc, n_N}(B_1, \dotsc, B_N)$ with respect to the number of bicolored cycles at fixed numbers of bubbles $B_1, \dotsc, B_N$.

The reason why Theorem \ref{thm:Gurau} and the Gurau-Schaeffer classification \cite{Gurau-Schaeffer} of colored graphs with respect to Gurau's degree are not sufficient to classify the colored triangulations of $\cT_{n_1, \dotsc, n_N}(K_1, \dotsc, K_N)$ is that one can only get triangulations of vanishing Gurau's degree for special CBBs called melonic CBBs. The colored triangulations of vanishing Gurau's degree are then called melonic triangulations. It is easier to describe those objects using colored graphs and bubbles.

\begin{definition} [Melonic bubbles and colored graphs]
Melonic bubbles are built recursively, starting from the bubble with two vertices connected by all the colors in $\{1, \dotsc, d\}$, Figure \ref{fig:2VertexBubble}, by inserting on any chosen edge of color $c$ two vertices connected by all the colors except $c$
\begin{equation} \label{MelonicInsertion}
\begin{array}{c}\includegraphics[scale=.4]{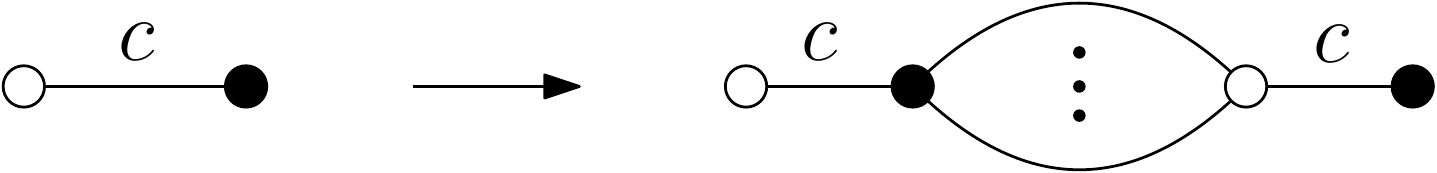} \end{array}
\end{equation}
Melonic colored graphs are defined similarly with the set of colors $\{0, 1, \dotsc, d\}$ instead of $\{1, \dotsc, d\}$.
\end{definition}

The insertion increases the number of vertices by two. In particular, the first insertion turns Figure \ref{fig:2VertexBubble} into Figure \ref{fig:4VertexBubble} with $q=1$. The two bubbles on the left of \eqref{6VertexBubbles3d} are not only planar but actually melonic. Clearly, a melonic bubble is fully encoded in the history of insertions and the latter is a tree with colored edges to remember the color on which each insertion was performed. Melonic bubbles (up to the choice of a root vertex) are just in bijection with $d$-ary trees.

\begin{theorem} \label{thm:Melons}
Colored graphs which maximize the number of bicolored cycles are those of vanishing Gurau's degree. For $d\geq 3$, they are the melonic colored graphs, in bijection with $(d+1)$-ary trees. They can only be built from melonic bubbles.
\end{theorem}

\begin{proof}
The original proof is in \cite{Melons} (see \cite{Gurau-Schaeffer} for a purely combinatorial reference) and uses special surfaces canonically embedded which are called jackets. This result was then applied in the context of graphs built from certain bubbles in \cite{Uncoloring}. In fact, the results of the present paper will include the above theorem for melonic bubbles as a special case and the proof will not rely on jackets at all. While we will restrict to $d=3$ so that our results apply to planar bubbles, all our theorems can be straightforwardly extended to arbitrary $d\geq 3$ in the case of melonic bubbles.
\end{proof}

This means that colored triangulations which are built from non-melonic CBBs cannot have vanishing Gurau's degree, because they cannot grow as many $(d-2)$-simplices. In fact, the Gurau-Schaeffer classification suggests that for non-melonic CBBs there is only a finite number of colored triangulations at fixed value of Gurau's degree. This means that no notion of large scale, continuous limit can be reached, and universality cannot be studied using Gurau's degree.

If $G\in\cG_{n_1, \dotsc, n_N}(B_1, \dotsc, B_N)$, we denote $C(G)$ the total number of bicolored cycles of colors $\{a,b\}$ for $0\leq a<b\leq d$, which is the total number of $(d-2)$-simplices of the corresponding colored triangulation
\begin{equation}
C(G) = \Delta_{d-2}(T).
\end{equation}
Fixing the bubbles $\{B_i\}$ and their numbers $\{n_i\}$ actually fixes the number of bicolored cycles with colors $\{a, b\}$ for $1\leq a<b\leq d$, i.e. those which do not have the color 0. Let us denote $C(B_i)$ the total number of bicolored cycles of $B_i$ and 
\begin{equation}
C_0(G) = \sum_{a=1} C_{0a}(G)
\end{equation}
the total number of bicolored cycles with colors $\{0, a\}$. Therefore
\begin{equation}
C(G) = \sum_{0\leq a<b\leq d} C_{ab}(G) = C_0(G) + \sum_{i=1}^N n_i\, C(B_i).
\end{equation}
Since each $C(B_i)$ is fixed in $\cG_{n_1, \dotsc, n_N}(B_1, \dotsc, B_N)$, the classification with respect to $C(G)$ is equivalent to the classification with respect to $C_0(G)$. This establishes the main question.

{\bf Main question.} We denote 
\begin{equation}
C_{n_1 \dotsc n_N}(B_1, \dotsc, B_N) = \max_{G \in \cG_{n_1, \dotsc, n_N}(B_1, \dotsc, B_N)} C_0(G)
\end{equation}
the maximal number of bicolored cycles with colors $\{0, a\}$, $a=1, \dotsc, d$ and 
\begin{equation}
\cG^{\max}_{n_1, \dotsc, n_N}(B_1, \dotsc, B_N) = \Bigl\{G \in \cG_{n_1, \dotsc, n_N}(B_1, \dotsc, B_N),\ C_0(G) = C_{n_1, \dotsc, n_N}(B_1, \dotsc, B_N)\Bigr\}
\end{equation}
the subset of graphs which have this maximal number of bicolored cycles. The main question is two-fold.
\begin{itemize}
\item Find $C_{n_1 \dotsc n_N}(B_1, \dotsc, B_N)$. This is equivalent to a sharp bound on $C_0(G)$ which, as it turns out, grows linearly with the size of the graph. From examples, we expect
\begin{equation}
C_0(G) \leq d + \alpha(B_1, \dotsc, B_N) V(B),
\end{equation}
where $V(B)$ is the number of vertices of $G$ and $\alpha(B_1, \dotsc, B_N) < d(d-1)/4$ by comparison with Gurau's value in Theorem \ref{thm:Gurau}.
\item Characterize the graphs which maximize the number of bicolored cycles, i.e. $\cG^{\max}_{n_1, \dotsc, n_N}(B_1, \dotsc, B_N)$.
\end{itemize}

Because of the special role played by the color 0, we will, from here on out, draw {\bf edges of color 0 with dashed lines and edges of colors $1, \dotsc, d$ with solid lines}.

%%%%%%%%%%%%%%%
\section{Edge flips, boundary bubbles and 4-edge-cuts} \label{sec:BdryBubblesAndFlips}
%%%%%%%%%%%%%%%

In this section, we introduce two tools: $i)$ the flips of edges which transform a colored graph into another with the same set of bubbles but a different number of bicolored cycles, $ii)$ the notion of boundary bubble which enables us to keep track of the bicolored cycles which go through two regions (typically one bubble versus the other bubbles) of a graph. This is readily applied to eliminate 4-edge-cuts made of edges of color 0 from $\cG^{\max}_{n_1, \dotsc, n_N}(B_1, \dotsc, B_N)$ for arbitrary bubbles.

%%%%%%%%%%%%%%%
\subsection{Edge flips} \label{sec:Flips}
%%%%%%%%%%%%%%%

\begin{definition} [Edge flip]
Let $G$ be a colored graph with at least two edges of color 0, $e_1$ between $v_1$ and $\bar{v}_1$, and $e_2$ between $v_2$ and $\bar{v}_2$. The flip of $e_1$ and $e_2$ is the transformation of $G$ into $G'$ where $e_1$ and $e_2$ are removed and replaced in $G'$ with two other edges of color 0, one between $v_1$ and $\bar{v}_2$ and the other between $v_2$ and $\bar{v}_1$,
\begin{equation}
G = \begin{array}{c} \includegraphics[scale=.4]{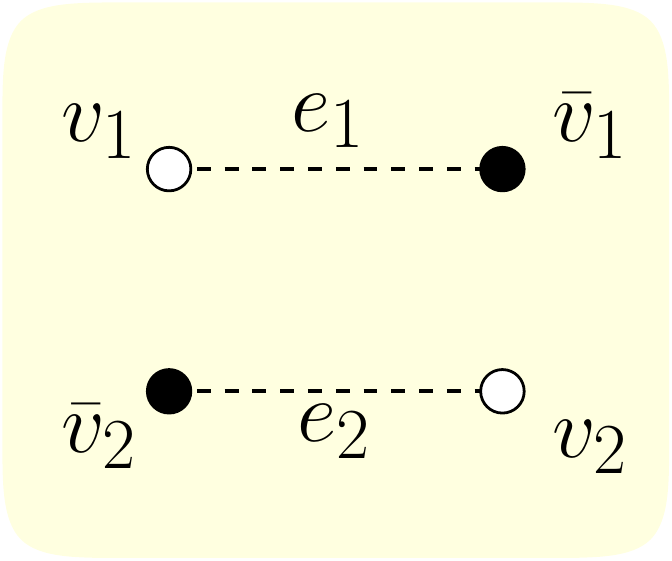} \end{array} \quad \to \quad G' = \begin{array}{c} \includegraphics[scale=.4]{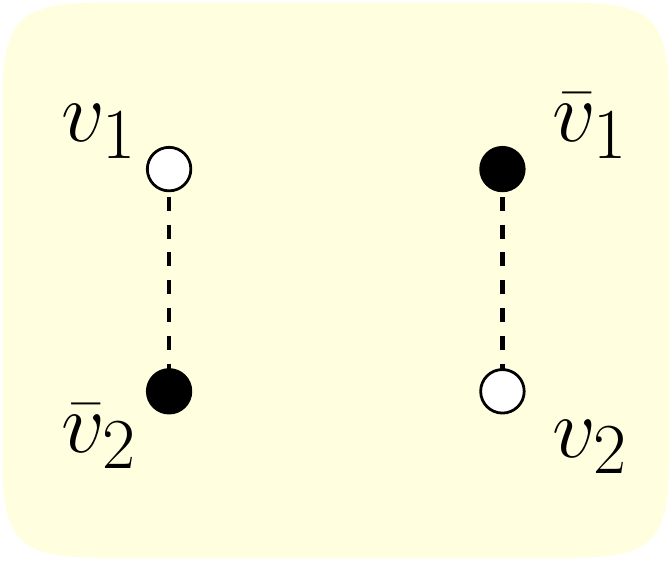} \end{array}
\end{equation}
\end{definition}
%Equivalently, the flip cuts the two edges of color into two halves and glues them back the only other way which preserves bipartiteness.
A flip may disconnect the graph, i.e. $G'$ may not be connected even when $G$ is.

To control the variation of the number of bicolored cycles through a flip, we introduce the quantity $I_G(e_1, e_2)$.

\begin{definition} \label{def:I}
If $G$ is a colored graph with colors $0, \dotsc, d$ and $e, e'$ are two edges of color 0, we denote $I_G(e, e')\subset \{1, \dotsc, d\}$ the set of colors for which the same bicolored cycle of colors $\{0, c\}$ goes along $e$ and $e'$ in $G$.
\end{definition}

For each color $c\in\{1, \dotsc, d\}$ and each edge of color 0, in particular $e$ and $e'$, there is exactly one bicolored cycle with colors $\{0,c\}$ along that edge. Therefore, for each $c\in\{1, \dotsc, d\}$, it is either the same bicolored cycle of colors $\{0,c\}$ along $e$ and $e'$, then $c\in I_G(e, e')$, or they are distinct cycles of colors $\{0,c\}$ and then $c\not\in I_G(e, e')$.

\begin{lemma} \label{lemma:flip}
Let $e, e'$ be two edges of color 0 in $G$ and assume their flip turns $G$ into the connected graph $G'$. Then 
\begin{equation}
C_0(G') = C_0(G) - d + 2|I_G(e,e')|.
\end{equation}
In particular, at $d=3$
\begin{enumerate}
\item if $e,e'$ are incident to vertices which are connected by exactly one edge of color $c\in \{1, 2, 3\}$, i.e.
\begin{equation}
\begin{array}{c} \includegraphics[scale=.4]{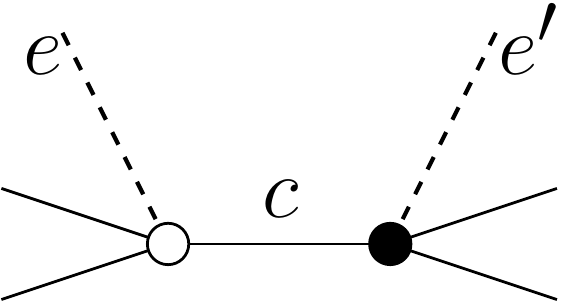} \end{array}\qquad \subset G \hspace{2cm} \text{then} \quad C_0(G') \geq C_0(G) - 1,
\end{equation}
\item if $e,e'$ are incident to vertices which are connected by exactly two edges (we say the latter form a 2-dipole), i.e. 
\begin{equation}
\begin{array}{c} \includegraphics[scale=.4]{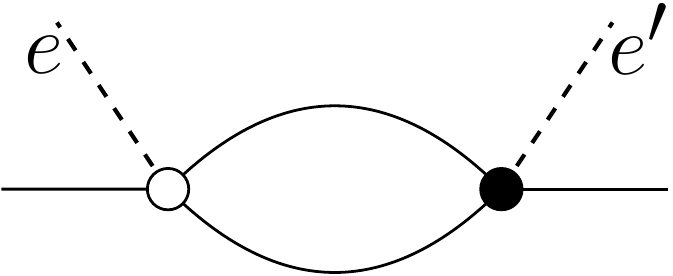} \end{array}\qquad \subset G \hspace{2cm} \text{then} \quad C_0(G') \geq C_0(G) + 1,
\end{equation}
\end{enumerate}
\end{lemma}

\begin{proof}
The bicolored cycles of $G$ which do not go along $e$ or $e'$ are not affected by the flip. We can thus focus on those which go along $e$ or $e'$. For every color $c\in\{1, \dotsc, d\}$ and every edge of color 0 in $G$, there is exactly one bicolored cycle of colors $\{0,c\}$ which goes along this edge. If $c\in I_G(e,e')$, i.e. it is the same bicolored cycle of colors $\{0,c\}$ along both $e$ and $e'$, then the flip splits it into two,
\begin{equation}
\forall c\in I_G(e, e') \qquad C_{0c}(G') = C_{0c}(G) + 1,
\end{equation}
and they are $|I_G(e,e')|$ of them. Conversely, if $c\not\in I_G(e,e')$, it means that two different bicolored cycles of colors $\{0, c\}$ go along $e$ and $e'$ and the flip merges them into one,
\begin{equation}
\forall c\not\in I_G(e, e') \qquad C_{0c}(G') = C_{0c}(G) - 1,
\end{equation}
There are $d-|I_G(e,e')|$ of them. Therefore $C_0(G') = \sum_{c=1}^d C_{0c}(G') = C_0(G) + | I_G(e,e')| - (d - |I_G(e,e')|)$.

Let us prove the special cases at $d=3$. Notice that if the graph has more than two vertices, two adjacent vertices can be connected by either one edge or two parallel edges (forming a 2-dipole). In the case where $e,e'$ are incident to vertices which are connected by exactly one edge of color $c\in \{1, 2, 3\}$, then $c\in I_G(e,e')$ and thus $|I_G(e,e')| \geq 1$. If there is a 2-dipole, say with colors $c, c'$, then they both belong to $I_G(e,e')$ and thus $|I_G(e,e')|\geq 2$.
\end{proof}

%%%%%%%%%%%%%%%
\subsection{Boundary bubbles} \label{sec:BdryBubbles}
%%%%%%%%%%%%%%%

An crucial notion is that of boundary bubbles, introduced by Gurau in \cite{Universality}. Since then, it has been used intensively in the colored graph literature: to study the Schwinger-Dyson equations (equations on the generating functions) \cite{DoubleScaling}, to find the set of graphs $\cG^{\max}_n(B)$ from the knowledge of the set $\cG^{\max}_n(B')$ for another bubble $B'$ \cite{SigmaReview, MelonoPlanar}. Here, it will be sufficient to introduce it in the context of subgraphs.

%For this, we need to consider connected colored graphs with the colors $0, 1, \dotsc, d$ where some white and black vertices have all the colors incident except 0. The typical case we will encounter consists of a connected subgraph of an element of $\cG_{n_1, \dotsc, n_N}(B_1, \dotsc, B_N)$ where all vertices of the subgraph have the colors $1, \dotsc, d$ but possibly not 0.

\begin{definition} [Colored subgraph and free vertices]
Let $G\in \cG_{n_1, \dotsc, n_N}(B_1, \dotsc, B_N)$ and $H\subset G$ a connected subgraph. We say that $H$ is a colored subgraph if each color $c=1, \dotsc, d$ is incident on all vertices of $H$. The vertices which do not have the color 0 incident are called free vertices.
\end{definition}

A bubble is a colored subgraph which only has free vertices. The subgraph $H$ can also be seen as coming from a collection of bubbles glued along the color 0 but leaving some vertices free. In terms of colored triangulations, it means that bubbles are glued together to form an object which still has a boundary: those $(d-1)$-simplices of color 0 represented by the free vertices of $H$.

The paths which alternate edges of color 0 and edges of a fixed color $c\in\{1, \dotsc, d\}$ in $H$ are thus either closed, counted as bicolored cycles of colors $\{0, c\}$, or not closed and then join two free vertices.

\begin{definition} [Boundary bubble]
Let $H$ be a colored subgraph. Its boundary bubble $\partial H$ is defined as follows. Its vertices are the free vertices of $H$. It has an edge of color $c\in\{1, \dotsc, d\}$ between two vertices if there is an open path of colors $\{0, c\}$ between the corresponding two free vertices of $H$.
\end{definition}

\begin{proposition} \label{prop:BdryBubble}
The boundary bubble $\partial H$ is either a connected bubble or a disjoint union of connected bubbles. In particular at $d=3$, all boundary bubbles with up to six vertices are either melonic, union of melonic bubbles, or a bubble whose colored graph is $K_{3, 3}$.
\end{proposition}

\begin{proof} 
A vertex of $\partial H$ is a free vertex of $H$ which by definition has exactly one incident edge of color $c$ for all $c\in\{1, \dotsc, d\}$. Moreover $\partial H$ has no edges of color 0. It thus satisfies the definition of a bubble except possibly for connectedness.

At $d=3$, all bubbles with up to six vertices are displayed in Figures \ref{fig:2VertexBubble} (melonic), \ref{fig:4VertexBubble3d} (melonic too) and in \eqref{6VertexBubbles3d}. A boundary bubble with $V$ vertices can be any disjoint union of them with $V$ total vertices. Therefore the only non-melonic connected component which can arise is the one in \eqref{6VertexBubbles3d} with $K_{3,3}$ as underlying graph, for $V=6$.
\end{proof}

The purpose of the boundary bubble is to simplify the potentially intricate bicolored paths through subgraphs $H$, of arbitrary lengths alternating colors $0$ and $c$, and to replace with single edges.

\begin{proposition} \label{prop:BdryCycles}
Consider a graph $G\in\cG_{n_1, \dotsc, n_N}(B_1, \dotsc, B_N)$ and a colored subgraph $H\subset G$ with free vertices $v_1, \dotsc, v_k$. Denote $e_1, \dotsc, e_k$ the edges of color 0 incident to them in $G$. Let $G_{/H}$ be the graph obtained from $G$ by replacing $H$ with its boundary bubble $\partial H$. We still denote $v_1, \dotsc, v_k$ the free vertices of $\partial H$ and $e_1, \dotsc, e_k$ their incident edges of color 0 in $G_{/H}$. Then for each pair of edges $\{e_i, e_j\}$
\begin{equation}
I_G(e_i, e_j) = I_{G_{/H}}(e_i, e_j)
\end{equation}
\end{proposition}

\begin{proof}
Notice that by construction of $\partial H$, the free vertices of $\partial H$ are the same as those of $H$. We recall that $c\in I_G(e_i, e_j)$ if it is the same bicolored cycle of colors $\{0, c\}$ which goes along $e_i$ and $e_j$ in $G$, and similarly in $G_{/H}$. The proposition follows directly from the definition of the boundary bubble as an encoding of the bicolored paths through $H$. The bicolored cycle of colors $\{0, c\}$ which goes along $e_i$ in $G_{/H}$ is exactly the same as in $G$ except for its parts which go through $H$ which are replaced with single edges. It thus goes along $e_j$ in $G$ if and only if it does so in $G_{/H}$.
%A bicolored cycle of color $(0c)$ going along $e_i$ hits $v_i$ then goes through $H$ until it arrives at another free vertex, say $v_m$ and then goes along $e_m$. This path from $v_1$ to $v_m$ is represented by a single edge of color $c$ in $\partial H$. If $m=j$ we are done. If not, it is because after leaving $H$ along $e_m$, the bicolored cycle comes back at a later stage to $H$ and then leaves it again using $e_j$ on the way in or out.
\end{proof}

There are three types of bicolored cycles in $G$. Those restricted to vertices and edges of $H$, $C_0(H)$ of them; those restricted to edges and vertices of $G\setminus H$, $C_0(G\setminus H)$ of them, and the others. The latter are those going through $H$ but not restricted to it. In particular, they go along the edges of color 0 of $G$ which are incident to the free vertices of $H$. Denote their number by $C_0(G, H)$. Then
\begin{equation}
C_0(G) = C_0(H) + C_0(G\setminus H) + C_0(G, H).
\end{equation}
In the graph $G_{/H}$, the complement to $H$ has not changed, hence $C_0(G\setminus H) = C_0(G_{/H} \setminus \partial H)$. Moreover
\begin{equation}
C_0(G, H) = C_0(G_{/H}, \partial H),
\end{equation}
which follows from Proposition \ref{prop:BdryCycles}.

%Let us look at the simpler boundary bubbles. 
%\begin{itemize}
%\item If $G$ is a colored graph with two free vertices, $\partial G$ is the only bubble with two vertices.
%\item If $G$ has four free vertices, its boundary bubble $\partial G$ can be either a disjoint union of two bubbles with two vertices each, or a connected bubble with four vertices. Notice that in three dimensions, i.e. with colors 1, 2, 3, (connected) bubbles with four vertices are melonic only.
%\item If $G$ has six free vertices, its boundary bubble can consist in a disjoint union of bubbles with less than six vertices or a connected bubble with six vertices. In three dimensions, the disjoint union only involves melonic bubbles, while a connected bubble with six vertices is either melonic, or has colored graph $K_{3,3}$ (toric boundary).
%\end{itemize}

%%%%%%%%%%%%%%%
\subsection{4-edge-cuts} \label{sec:4Cuts}
%%%%%%%%%%%%%%%

Edge flips and boundary bubbles will combine to allow for comparing the bicolored cycles of colored graphs and check if they are in $G^{\max}_{n_1, \dotsc, n_N}(B_1, \dotsc, B_N)$. Our first application is the following.

\begin{proposition} \label{prop:4EdgeCuts}
For $d$ odd, for any bubbles $B_1, \dotsc, B_N$, graphs in $\cG^{\max}_{n_1, \dotsc, n_N}(B_1, \dotsc, B_N)$ have no 4-edge-cuts with four edges of color 0.
\end{proposition}

We recall that a $k$-edge-cut is a set of $k$ edges whose removal disconnects the graph but removing only a subset of them does not.

\begin{proof}
Consider a graph $G\in\cG_{n_1, \dotsc, n_N}(B_1, \dotsc, B_N)$ in odd dimension with a 4-edge-cut on four edges of color 0. It is thus made of of two colored subgraphs $H_L, H_R$ connected together by these four edges of color 0, $e_1$, $e_2$, $e_3$, $e_4$, as follows
\begin{equation}
G = \begin{array}{c} \includegraphics[scale=.4]{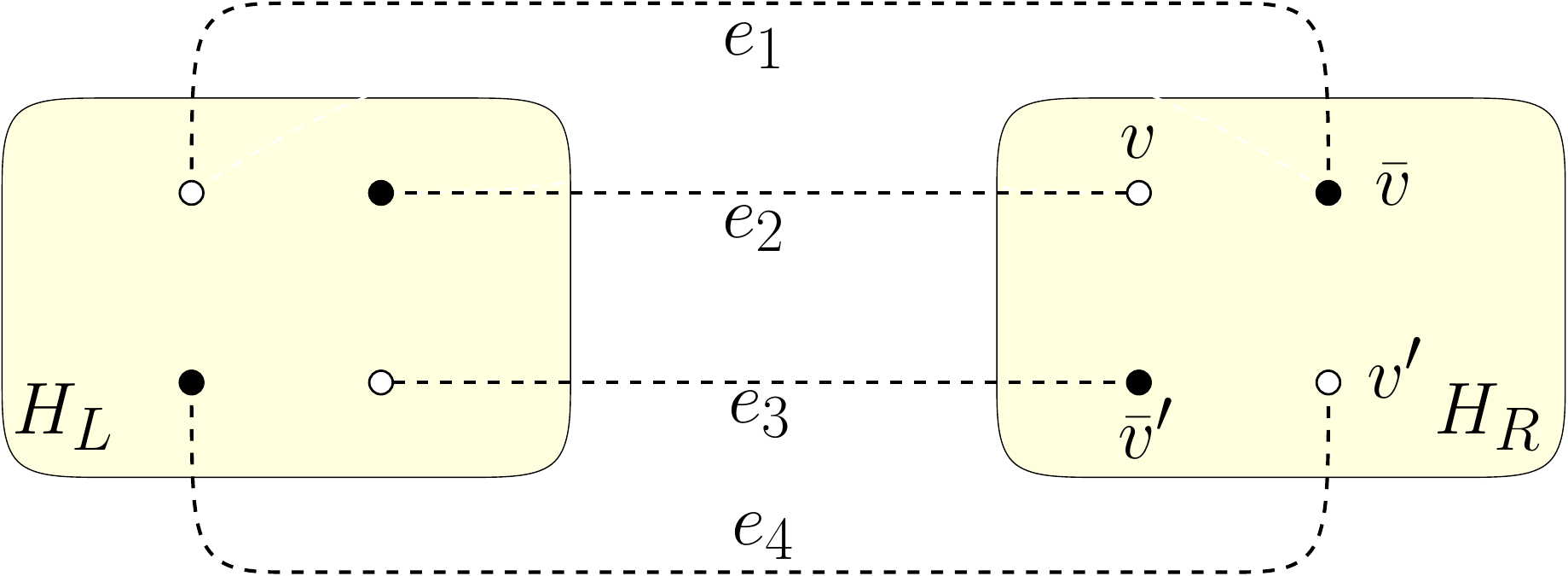} \end{array}
\end{equation}
where $v, v', \bar{v}, \bar{v}'$ are the free vertices of $H_R$.

We would like to see if there is a flip which would increase the number of bicolored cycles so that $G$ could not be in $\cG^{\max}_{n_1, \dotsc, n_N}(B_1, \dotsc, B_N)$. To see if flipping $e_1$ with $e_2$ or with $e_4$ increases the number of bicolored cycles, we need to know the sizes of $I_G(e_1, e_2)$ and $I_G(e_1, e_4)$, according to Lemma \ref{lemma:flip}. Proposition \ref{prop:BdryCycles} states that
\begin{equation}
I_G(e_1, e_j) = I_{G_{/H_R}}(e_1, e_j)
\end{equation}
where the graph $G_{/H_R}$ is obtained from $G$ by replacing $H_R$ with its boundary bubble $\partial H_R$.

The (possibly not connected) bubble $\partial H_R$ has only four vertices, $v, v', \bar{v}, \bar{v}'$. It can only be either like in Figure \ref{fig:4VertexBubble}, i.e. $1\leq q \leq d-1$ edges of colors $c_1, \dotsc, c_q$ connecting $v$ to $\bar{v}$ and $v'$ to $\bar{v}'$, and $d-q$ edges with the complementary colors connecting $v$ to $\bar{v}'$ and $v'$ to $\bar{v}$, or the cases $q=0, d$, i.e. two copies of Figure \ref{fig:2VertexBubble}. Since $d$ is odd, $q=d/2$ cannot happen. Up to exchanging the roles of $e_2$ and $e_4$, we can assume that $q>d/2$. At $d=3$ and $q=2$ for instance
\begin{equation}
G_{/H_R} = \begin{array}{c} \includegraphics[scale=.4]{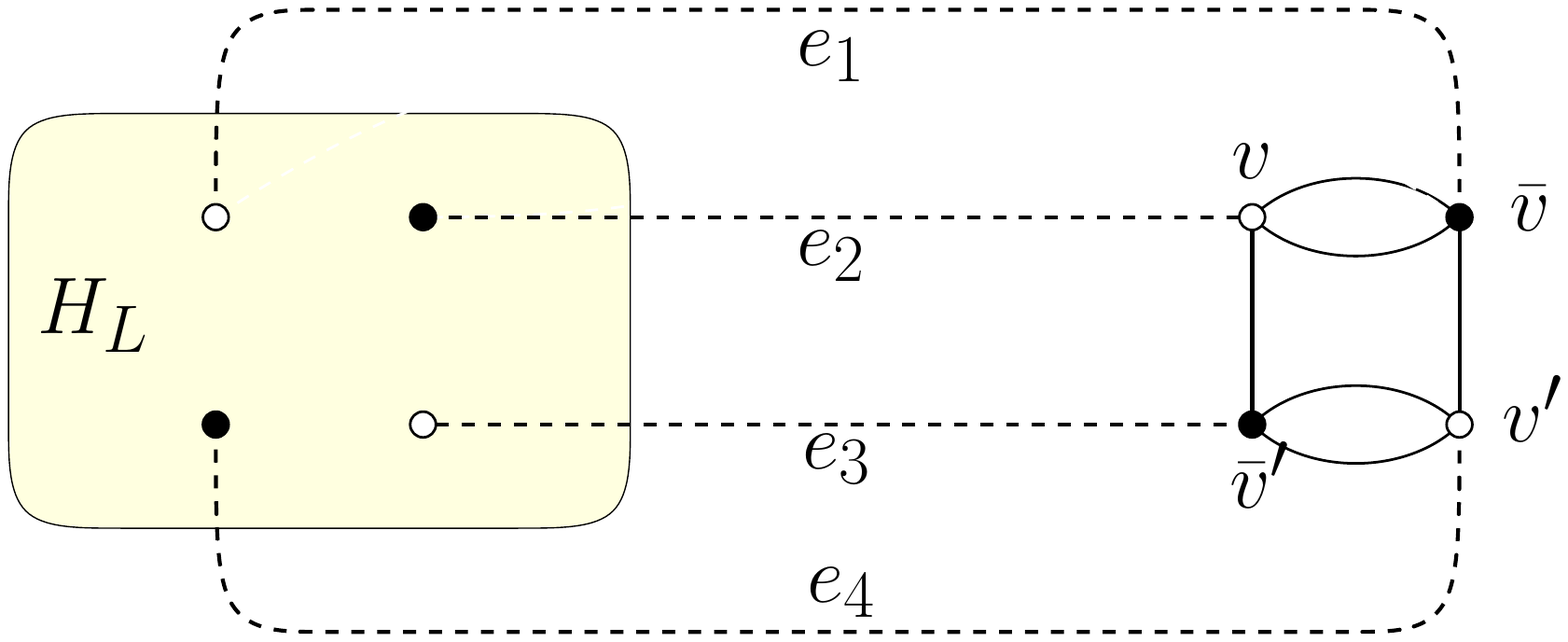} \end{array}
\end{equation}
This implies that $I_{G_{/H_R}}(e_1, e_2)$ contains at least $q$ colors and so does $I_G(e_1, e_2)$, yielding
\begin{equation}
|I_G(e_1, e_2)| > d/2.
\end{equation}
Let $G'$ be the graph $G$ with $e_1$ and $e_2$ flipped,
\begin{equation}
\begin{array}{c} \includegraphics[scale=.4]{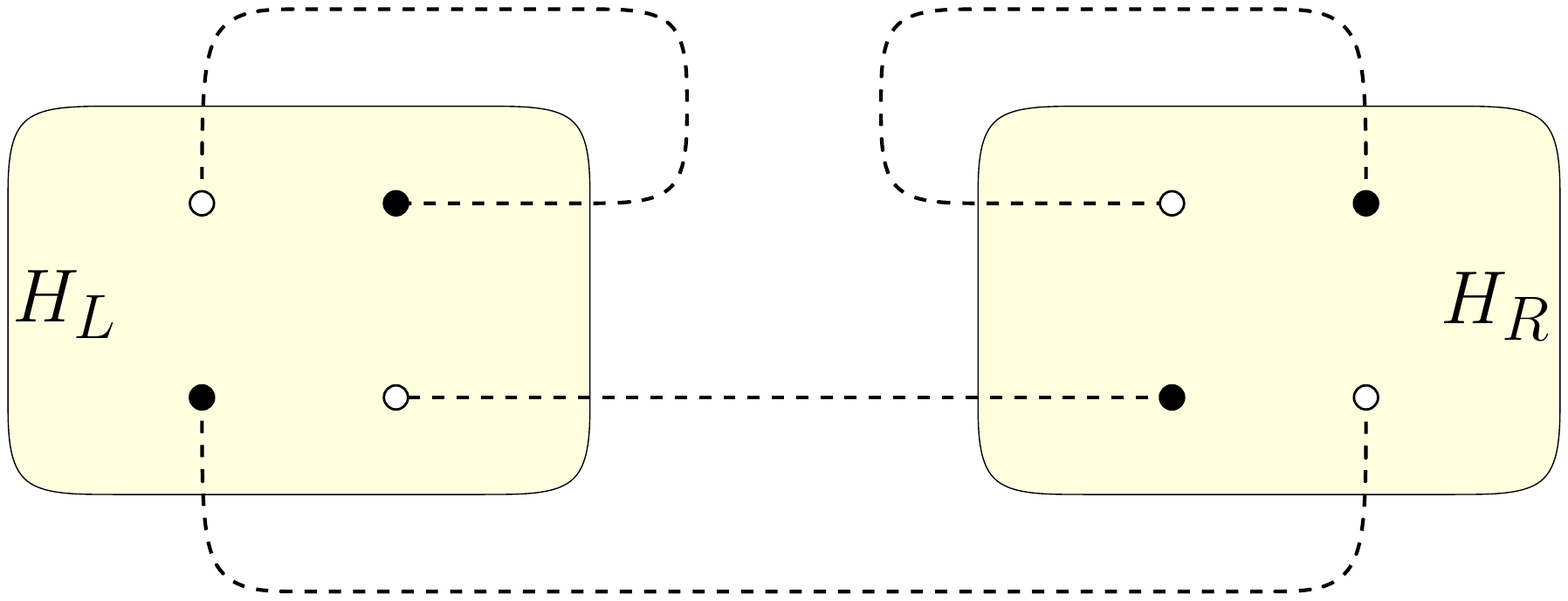} \end{array}
\end{equation}
From Lemma \ref{lemma:flip}, the number of bicolored cycles is
\begin{equation}
C_0(G') = C_0(G) -d +2|I_G(e_1, e_2)| > C_0(G),
\end{equation}
meaning that $G\not\in \cG^{\max}_{n_1, \dotsc, n_N}(B_1, \dotsc, B_N)$.
\end{proof}

%%%%%%%%%%%%%%%
\subsection{Contractions} \label{sec:Contractions}
%%%%%%%%%%%%%%%

\begin{definition} [Contractions] \label{def:Contraction} 
Let $G$ be a colored graph and $e$ an edge of color 0 in $G$, incident to $v$ and $v'$. Denote $e_c$ the edge of color $c$ incident to $v$ and $e'_c$ incident to $v'$, for $c=1, \dotsc, d$ (they may be the same). The contraction of $e$ is the graph $G/e$ obtained by removing $e, v, v'$ and joining $e_c$ with $e'_c$. Edges parallel to $e$ are removed, i.e. if $e_c=e'_c$,
\begin{equation}
\begin{array}{c} \includegraphics[scale=.4]{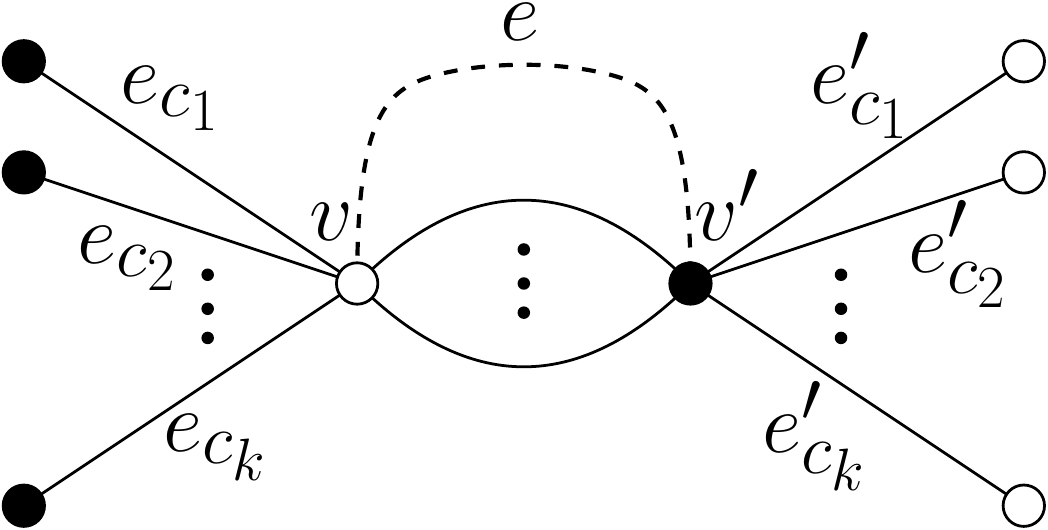} \end{array} \qquad \subset G \hspace{1cm} \underset{\text{contraction}}{\to} \hspace{1cm} \begin{array}{c} \includegraphics[scale=.4]{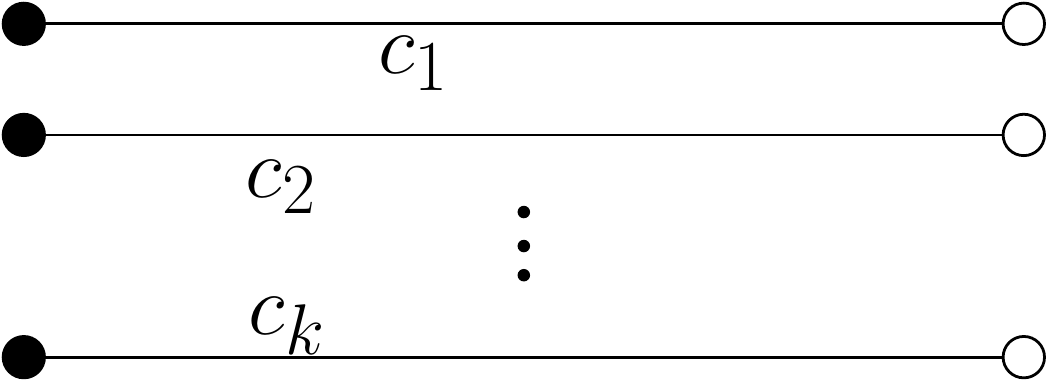} \end{array} \qquad \subset G/e
\end{equation}
\end{definition}

We will use contractions repeatedly in the proof of Theorem \ref{thm:1Planar}. Since they decrease the total number of vertices, they seem designed for inductions. However, a contraction changes some bubbles of the graph and it is thus necessary to make sure that the induction hypotheses can still apply. This requires the simple proposition below.

\begin{proposition} \label{prop:Contraction}
Let $G\in \cG_{n_1, \dotsc, n_N}(B_1, \dotsc, B_N)$ at $d=3$. The following contractions preserve the topologies of the bubbles,
\begin{itemize}
\item when $e$ is parallel to two edges of color $c_1, c_2\in\{1, 2, 3\}$,
\begin{equation} \label{2DipoleMove}
\begin{array}{c} \includegraphics[scale=.4]{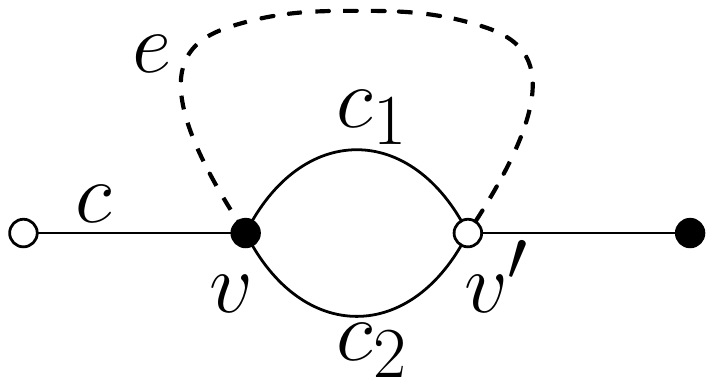} \end{array} \quad \to \quad \begin{array}{c} \includegraphics[scale=.4]{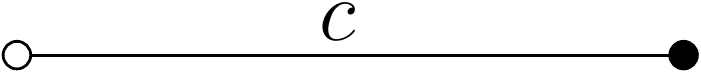} \end{array}
\end{equation}
\item when $e$, incident to the vertices $v, v'$, is parallel to an edge of color $c\in\{1, 2, 3\}$ and the bicolored cycles of colors $\{c_1,  c_2\} = \{1, 2, 3\}\setminus \{c\}$ incident to $v$ and $v'$ are distinct, i.e. $f_{c_1 c_2}\neq f'_{c_1 c_2}$,
\begin{equation} \label{1DipoleMove}
\begin{array}{c} \includegraphics[scale=.4]{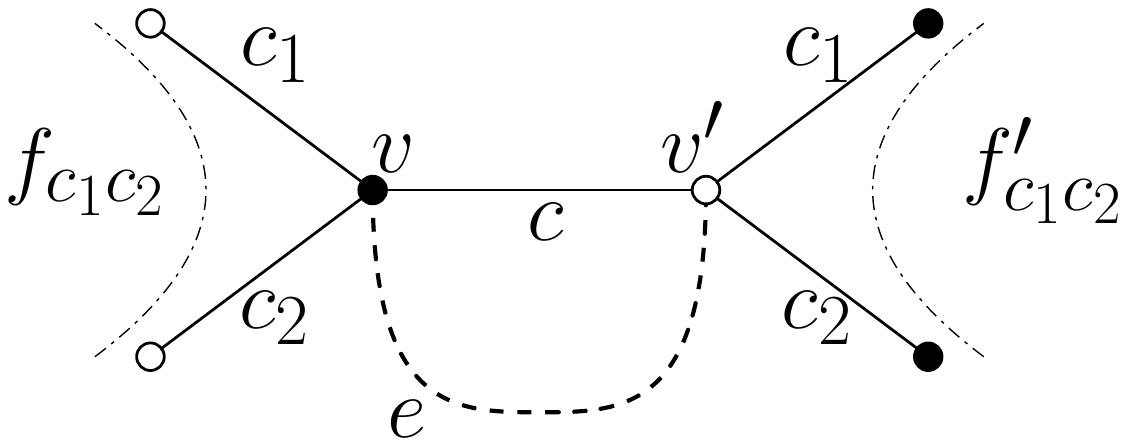} \end{array} \quad \to \quad \begin{array}{c} \includegraphics[scale=.4]{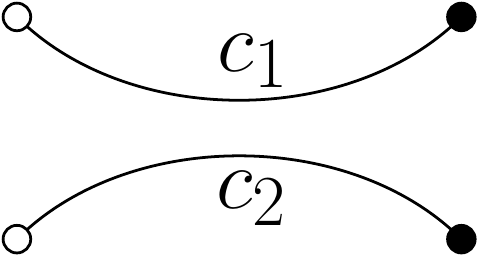} \end{array}
\end{equation}
\end{itemize}
\end{proposition}

\begin{proof}
This is part of a general theory of moves, known as dipole moves, which are topology preserving. The main theorem is stated in Section \ref{sec:Topology} without proof. The combinatorial proof of Theorem \ref{thm:1Planar} however only need the topology of the two-dimensional bubbles to be preserved, in which case a simple proof can be given.

We recall from Corollaries \ref{cor:2D} and \ref{cor:PlanarBubbles} that a bubble $B$ has a canonical embedding as a combinatorial map (cyclic order $(123)$ around white vertices and $(132)$ around black vertices) where the faces (bicolored cycles) of colors $\{a,b\}$ of the bubble are the faces of the map and the topology of the bubble is that of the map.
%The fact that these two moves do not change the topology of $G$ is part of the general theory of colored triangulations and their colored graph representation. They are a $3$-dipole move and a $2$-dipole move respectively, where a $k$-dipole consists in $k$ parallel edges between $v$ and $v'$ such that the connected components which have the complementary colors and incident to $v$ and $v'$ are distinct. Contracting the dipole does not change the topology. For the 2-dipole, the connected components with the complementary colors are bicolored cycles.

%The same theorem can be applied to the bubble $B_i$ which contains the parallel edges. This shows that the topology of the bubble is preserved. In the case $d=3$, this can be seen directly using the canonical embedding of the bubble as a combinatorial map, Corollary \ref{cor:2D}: use the cyclic order $(123)$ for the colors incident to white vertices and the cyclic order $(132)$ for the colors incident to black vertices. Then the bicolored cycles of the bubbles are the faces of the map.
The first move thus removes a face of degree two of the map which does not change its topology obviously. As for the second move, the bicolored cycles of colors $\{c_1,  c_2\}$ incident to $v$ and $v'$ being distinct means that the corresponding faces are distinct in the map. The move thus merges them into a single face. It also removes two vertices and three edges (one of each color). The Euler's characteristic of the map is thus unchanged.
\end{proof}

Contractions will play a major role and will combine with inductions thanks to the following lemma, which will be used in the proof of Theorem \ref{thm:1Planar} repeatedly.

\begin{lemma} \label{lemma:NotMax}
Set $d=3$ again. If $G\in\cG_{n_1, \dotsc, n_N}(B_1, \dotsc, B_N)$ has an edge $e$ of color 0 involved in a bicolored cycle of length two and colors $\{0, c\}$ for some $c\in\{1, 2, 3\}$, and such that its contraction $G/e$ does not maximize the number of bicolored cycles for its set of bubbles, then $G\not\in \cG^{\max}_{n_1, \dotsc, n_N}(B_1, \dotsc, B_N)$.
\end{lemma}

A word of caution is in order. Contractions can disconnect bubbles in some cases. Then the maximizing of the number of bicolored cycles of $G/e$ has to be done for graphs with possibly several connected components, as shown in the proof below.

\begin{proof}
In all cases described below, we assume that the edge $e$ is incident to a bubble of type $B_1$.

The first case corresponds to the situation of \eqref{2DipoleMove}. The contraction turns $B_1$ into a bubble $B$ (still connected and which may be or not of one of the other types $B_i$). In fact, we will need to retain more information: in order to undo the contraction, we need to remember the location of $v$ and $v'$. We denote $B^{(e_c)}$ the bubble $B$ with the edge $e_c$ marked,
\begin{equation} \label{2DipoleMoveMarked}
B_1 = \begin{array}{c} \includegraphics[scale=.4]{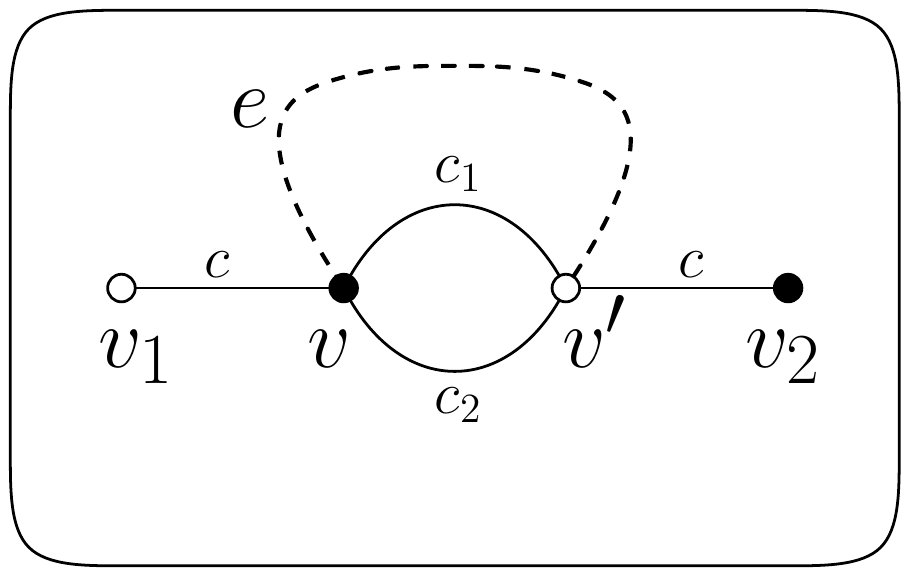} \end{array} \quad \leftrightarrow \quad B^{(e_c)} = \begin{array}{c} \includegraphics[scale=.4]{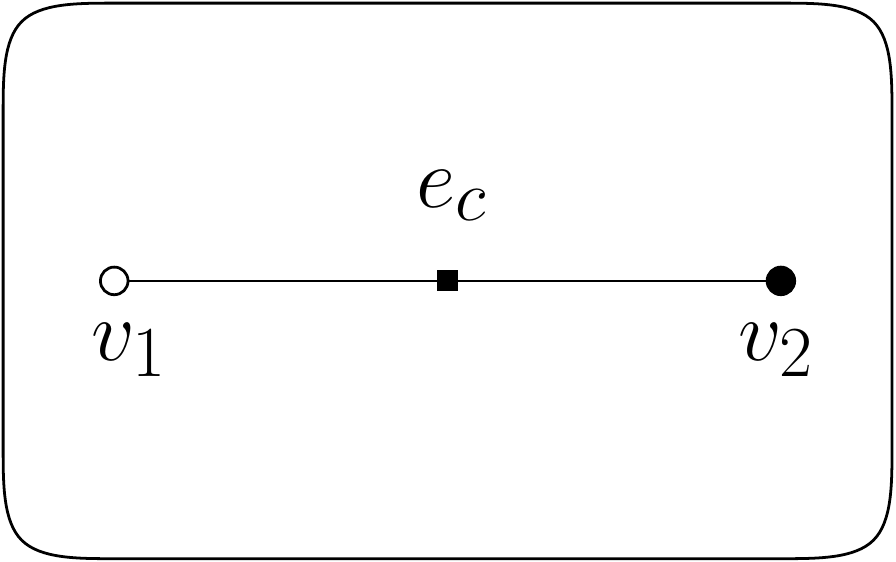} \end{array}
\end{equation}
Thus $G/e$ is connected and belongs to
\begin{equation}
\cG_{\text{case 1}} = \cG_{1, n_1-1, \dotsc, n_N}(B^{(e_c)}, B_1, \dotsc, B_N).
\end{equation}
Moreover, the contraction removes exactly one bicolored cycle with colors $\{0, c_1\}$ and one with colors $\{0, c_2\}$. The number of bicolored cycles with colors $\{0, c\}$ is unaffected, hence
\begin{equation} \label{BicoloredCyclesContractionCase1}
C_0(G/e) = C_0(G) - 2 \qquad \text{in case 1}.
\end{equation}

The second case corresponds to $e$ parallel to a single edge $e_c$ of color $c$ in $B_1$ and such that the contraction turns $B_1$ into a (connected) bubble $B$,
\begin{equation} \label{1DipoleMoveMarkedCase2}
B_1 = \begin{array}{c} \includegraphics[scale=.4]{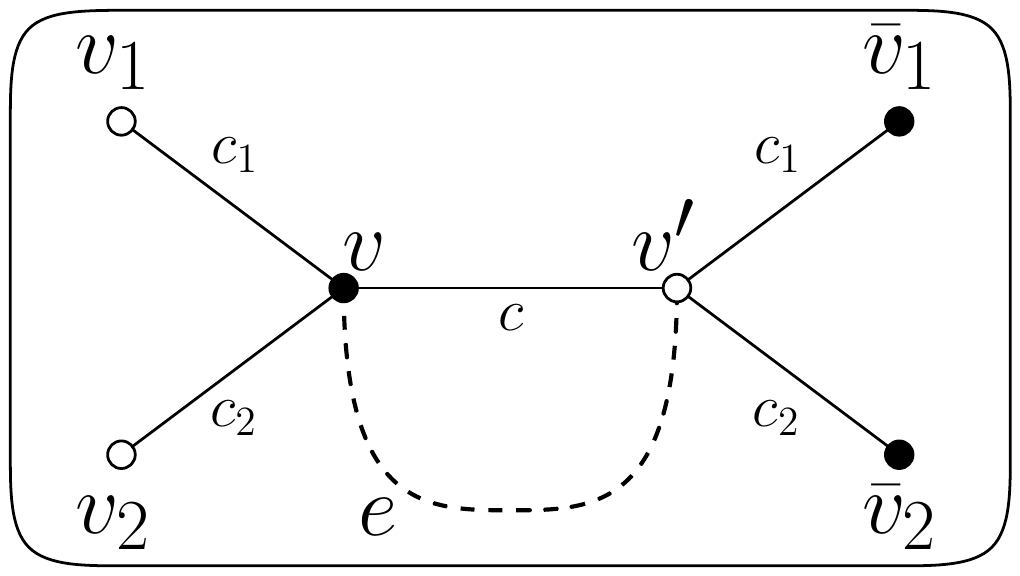} \end{array} \quad \leftrightarrow \quad B^{(e_{c_1}, e_{c_2})} = \begin{array}{c} \includegraphics[scale=.4]{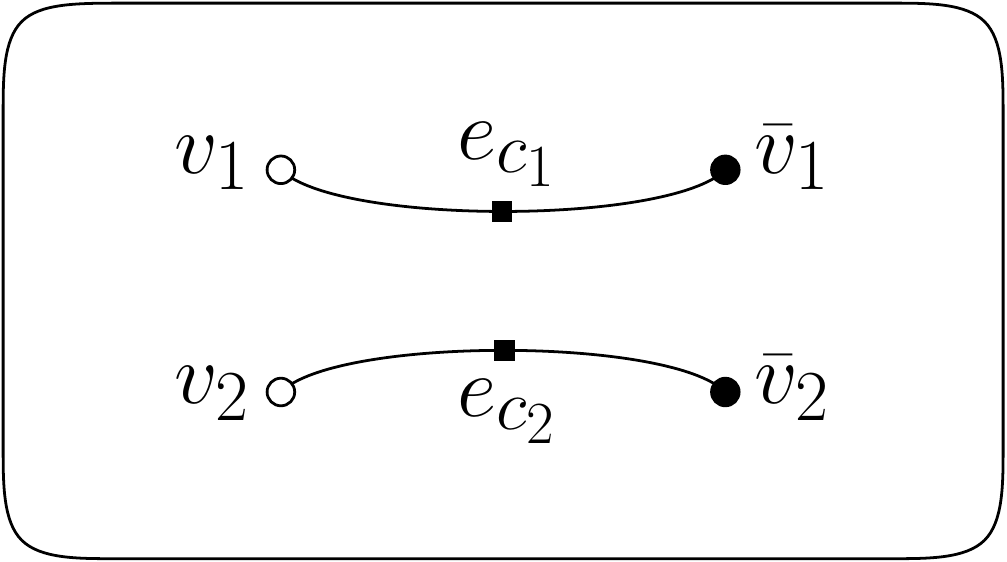} \end{array}
\end{equation}
This situation is for instance that of \eqref{1DipoleMove}. It can also be obtained if $B_1$ is non-planar and it is the same bicolored cycle of colors $\{c_1, c_2\}$ which is incident on both $v$ and $v'$, i.e. $f_{c_1 c_2} = f'_{c_1 c_2}$. The bubble $B_1$ is turned into a bubble $B^{(e_{c_1}, e_{c_2})}$ with the edges $e_{c_1}, e_{c_2}$ marked, the latter being the remaining edges in the right hand side of \eqref{1DipoleMove}. Thus $G/e$ is connected and belongs to
\begin{equation}
\cG_{\text{case 2}} = \cG_{1, n_1-1, \dotsc, n_N}(B^{(e_{c_1}, e_{c_2})}, B_1, \dotsc, B_N).
\end{equation}
Moreover, the contraction removes exactly one bicolored cycle with colors $\{0, c\}$ and that is all, hence
\begin{equation} \label{BicoloredCyclesContractionCase2}
C_0(G/e) = C_0(G) - 1  \qquad \text{in case 2}.
\end{equation}

The third case is similar except that the contraction turns $B_1$ into a union of (connected) bubbles. It is the case if $B_1$ is planar and $f_{c_1 c_2} = f'_{c_1 c_2}$. Let us show that there can only be two connected components,
\begin{equation}
B_1 = \begin{array}{c} \includegraphics[scale=.4]{1DipoleCase2.pdf} \end{array} \quad \leftrightarrow \quad B_L^{(e_{c_1})} = \begin{array}{c} \includegraphics[scale=.4]{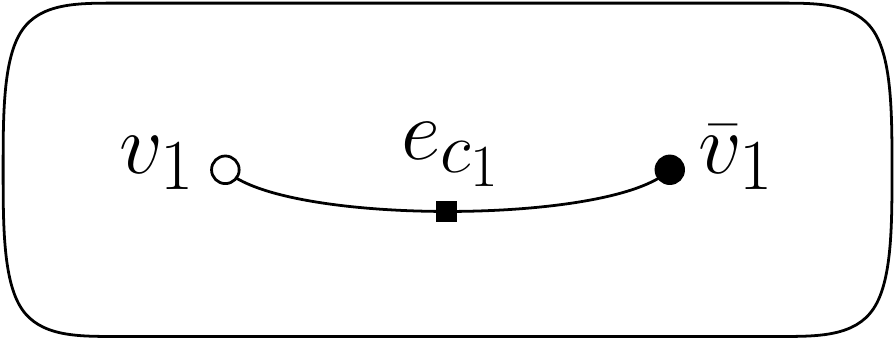} \end{array} \qquad B_R^{(e_{c_2})} = \begin{array}{c} \includegraphics[scale=.4]{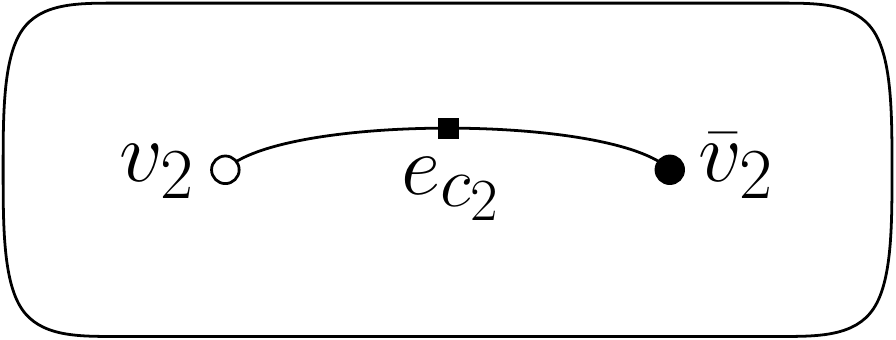} \end{array}
\end{equation}
Consider a vertex $\tilde{v} \neq v, v'$ in $B_1$. If there is a path from $\tilde{v}$ to $v_2$ (or $\bar{v}_2$) which does not go through $v$ and $v'$, then all the edges of this path are unaffected by the contraction and $\tilde{v}$ will be in the same component as $v_2$ denoted $B_R$. If however all paths from $\tilde{v}$ to $v_2$ go through $v$ or $v'$, it means that there is a path from $\tilde{v}$ to $v_1$ which does not go through $v$ and $v'$. Then $\tilde{v}$ is in the same component as $v_1$ denoted $B_L$.

The same argument applies to $G$ itself and shows that $G/e$ is either connected or has two connected components, $G_L$ which contains $B_L$ and $G_R$ which contains $B_R$. Both situations in fact occur when $e$ is fixed but all the other edges of color 0 of $G$ are changed so that $G$ visits all of $\cG_{n_1, \dotsc, n_N}(B_1, \dotsc, B_N)$ except for $e$ being fixed. When $G/e$ is connected, it belongs to 
\begin{equation}
\cG_{\text{connected}} = \cG_{1,1,n_1-1, \dotsc, n_N}(B_L^{(e_{c_1})}, B_R^{(e_{c_2})}, B_1, \dotsc, B_N).
\end{equation}
When it is not, the connected components $G_L$ and $G_R$ can contain any other bubble from the set of allowed bubbles $B_1, \dotsc, B_N$, as long as the total number of bubbles of type $B_i$ is $n_i$, for $i=2, \dotsc, N$ and $n_1-1$ for $i=1$. Therefore we introduce $n_{1}^L, \dotsc, n_{N}^L$ and $n_1^R, \dotsc, n_N^R$, such that $0 \leq n_i^L, n_i^R \leq n_i$ and $n_i^L + n_i^R = n_i$ for $i=2, \dotsc, N$ and $n_1^L + n_1^R = n_1-1$. Then $G/e$ belongs to
\begin{equation}
\cG_{\text{disconnected}} = \bigcup_{\substack{n_1^L, \dotsc, n_N^L\\ n_1^R, \dotsc, n_N^R}} \cG_{1, n_1^L, \dotsc, n_N^L}(B_L^{(e_{c_1})}, B_1, \dotsc, B_N) \times \cG_{1, n_1^R, \dotsc, n_N^R}(B_R^{(e_{c_2})}, B_1, \dotsc, B_N).
\end{equation}
Thus the space for $G/e$ is
\begin{equation} \label{ContractionSpaceCase3}
\cG_{\text{case 3}} = \cG_{\text{connected}} \cup \cG_{\text{disconnected}}.
\end{equation}
In the disconnected case, the total number of bicolored cycles is $C_0(G/e) = C_0(G_L) + C_0(G_R)$. In both cases, exactly one bicolored cycle is lost from $G$ to $G/e$, with colors $\{0, c\}$, hence
\begin{equation} \label{BicoloredCyclesContractionCase3}
C_0(G/e) = C_0(G) - 1  \qquad \text{in case 3}.
\end{equation}

For each case $i=1, 2, 3$, assume that $G/e$ does not maximize the number of bicolored cycles in $\cG_{\text{case $i$}}$. Then there exists $\tilde{G}' \in \cG_{\text{case $i$}}$ such that $C_0(\tilde{G}') > C_0(G/e)$. Notice that in case 3, $\tilde{G}'$ may be connected or disconnected independently of $G/e$. The graph $\tilde{G}'$ contains either the bubble $B^{(e_c)}$, or $B^{(e_{c_1}, e_{c_2})}$, or the bubbles $B_L^{(e_{c_1})}, B_R^{(e_{c_2})}$, which in all three cases have marked edges. These edges can be used to ``undo'' the contraction, that is reinsert the vertices $v, v'$ and the appropriate edges between them together with the edge $e$ of color 0. This gives a graph $\tilde{G}$ such that
\begin{equation}
\tilde{G}/e = \tilde{G}'.
\end{equation}
The bubbles with marked edges now reform the original bubble $B_1\subset \tilde{G}$ so $\tilde{G}\in \cG_{n_1, \dotsc, n_N}(B_1, \dotsc, B_N)$.

The last key point is that the relation between the numbers of bicolored cycles of $\tilde{G}'$ and $\tilde{G}$ is the same as between $G$ and $G/e$ for each case, namely Equations \eqref{BicoloredCyclesContractionCase1}, \eqref{BicoloredCyclesContractionCase2}, \eqref{BicoloredCyclesContractionCase3}. For instance in case 1, $C_0(\tilde{G}) = C_0(\tilde{G}') + 2 > C_0(G/e) + 2 = C_0(G)$. In all cases, one gets $C_0(\tilde{G}) > C_0(G)$.
\end{proof}

%%%%%%%%%%%%%%%
\section{Colored graphs with planar bubbles} \label{sec:PlanarBubbles}
%%%%%%%%%%%%%%%

%%%%%%%%%%%%%%%
\subsection{One-CBB triangulations} \label{sec:1CBB}
%%%%%%%%%%%%%%%

A particular set of triangulations consists of those with a single CBB. They generalize unicellular maps to higher dimensions. In terms of simplices, they are formed by a perfect matching, or a pairing, of its boundary $(d-1)$-simplices which identifies them two by two to form a closed space. In terms of colored graphs, they have a single bubble $B$, and are obtained by adding edges of color 0 between its black and white vertices following a perfect matching.

\begin{definition} [Pairing of a bubble]
A pairing (or perfect matching) $\pi$ of a CBB is a partition of its boundary simplices in bipartite pairs, which defines a unique closed colored triangulation. In terms of dual colored graphs, a pairing is an element of $\cG_1(B)$, where the pairs are the black and white vertices connected by edges of color 0. We will denote those pairs $\{v, \pi(v)\}$ for vertices $v$ of $B$ (and of course $\pi^2(v) = v$).
\end{definition}

If $B$ is a bubble with $2V$ vertices, then there are $V!$ pairings. $\cG^{\max}_1(B)$ is the set of pairings of $B$ which maximize the number of bicolored cycles. In two dimensions, a CBB is a $2p$-gon and pairings correspond to identifications of the boundary edges two by two. The Harer-Zagier polynomial gives the number of pairings which form a surface of genus $g$. For $g=0$, this is the Catalan number of order $p$ obviously.

\begin{lemma} \label{lemma:qEdges}
If $B$ contains two vertices $v, \bar{v}$ connected by $q>d/2$ edges, then all graphs in $\cG^{\max}_1(B)$ have an edge of color $0$ between $v$ and $\bar{v}$, i.e. $\pi(v) = \bar{v}$ for $\pi\in\cG^{\max}_1(B)$ or graphically
\begin{equation}
\text{$B$ contains}\qquad \begin{array}{c} \includegraphics[scale=.4]{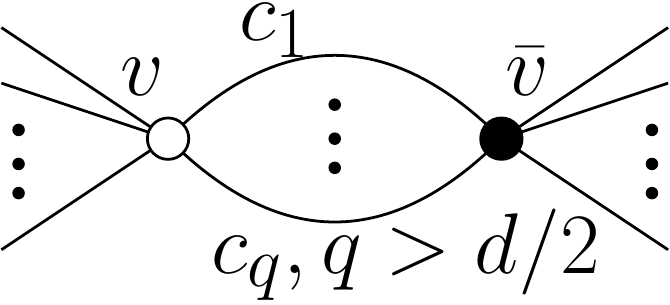} \end{array} \qquad \Rightarrow \qquad \text{$G\in\cG^{\max}_1(B)$ contains}\qquad \begin{array}{c} \includegraphics[scale=.4]{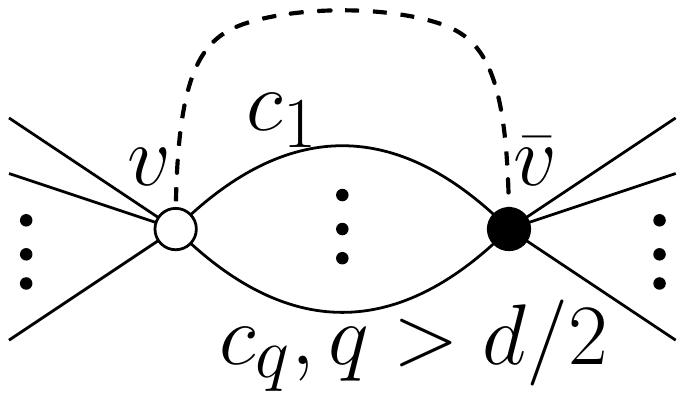} \end{array}
\end{equation}
\end{lemma}

\begin{proof}
Consider $G\in\cG_1(B)$ such that there are two distinct edges $e, e'$ of color 0 incident on $v$ and $\bar{v}$, and consider the graph $\tilde{G}$ with $e, e'$ flipped. The number of colors such that it is the same bicolored cycle which goes along $e$ and $e'$ in $G$ is $|I_G(e, e')| \geq q >d/2$. Therefore, Lemma \ref{lemma:flip} gives the variation of the number of bicolored cycles as $C_0(\tilde{G}) = C_0(G) - d + 2|I_G(e, e')| > C_0(G)$.
\end{proof}

This can be used to prove the following simple (and well-known) fact.

\begin{corollary}
If $B$ is a melonic bubble then $\cG^{\max}_1(B)$ has a single pairing and 
\begin{equation}
C_1(B) = \frac{(d-1)}{2} V(B) + 1
\end{equation}
where $V(B)$ is the total number of vertices of $B$.
\end{corollary}

\begin{proof}
We recall that a melonic bubble $B$ is made from recursive melonic insertions, i.e. insertions of vertices connected by $d-1$ edges as shown in \eqref{MelonicInsertion}, and starting from the 2-vertex bubble (Figure \ref{fig:2VertexBubble}). There exists a sequence of bubbles $B^{(1)}, \dotsc, B^{(V(B)/2-1)}, B^{(V(B)/2)}$ such that $B^{(V(B)/2)} = B$ and $B^{(1)}$ is the 2-vertex bubble, $B^{(k)}$ has $2k$ vertices and $B^{(k+1)}$ is obtained from $B^{(k)}$ by a melonic insertion.

Let us consider the two vertices $v$ and $\bar{v}$ and the $d-1$ parallel edges between them which have been added to $B^{(V(B)/2 - 1)}$ to get $B$. Lemma \ref{lemma:qEdges} applies directly with $q=d-1$. This fixes $\pi(v) = \bar{v}$ and thus the restriction of a pairing $\pi\in\cG^{\max}_1(B)$ to $v$ and $\bar{v}$. We can thus ``undo'' the melonic insertion \eqref{MelonicInsertion} and consider $B^{(V(B)/2-1)}$. One obviously has
\begin{equation}
C_1(B^{(V(B)/2-1)}) = C_1(B) - (d-1).
\end{equation}
This can be continued as an induction, with $C_1(B^{(k)}) = C_1(B^{(k+1)}) - (d-1)$, down to $k=1$. One then arrives at $B^{(1)}$ which has a single pairing, with $d$ bicolored cycles, $C_1(B^{(1)}) = d$.
\end{proof}

Identifying the subset $\cG^{\max}_1(B) \subset \cG_1(B)$ for an arbitrary bubble $B$ is a tremendously difficult matter and success has been obtained, sometimes only partially, only in limited cases. They include: 
\begin{itemize}
\item ``almost melonic'' bubbles, where instead of inserting $(d-1)$ parallel edges in \eqref{MelonicInsertion}, one inserts two vertices connected by $d-1>q>d/2$ parallel edges recursively, so that Lemma \ref{lemma:qEdges} applies at each step.
\item in even dimensions the case of ``necklaces'', i.e. bubbles made of a single cycle whose vertices are connected by exactly $d/2$ parallel edges. It is easy to see that the counting of bicolored cycles $C_0(G)$ is then equivalent to $d/2$ independent copies of the two-dimensional case. 
\item some trees of necklaces, that is a mix of the two cases above, \cite{UnitaryIntegrals, MelonoPlanar},
\item four-dimensional bubbles with exactly one cycle with colors $\{1,2\}$ and one of colors $\{3,4\}$. The numbers $C_1(B)$ of pairings maximizing the numbers of bicolored cycles are meander numbers \cite{MeandricBubbles}.
\end{itemize}
Importantly, none of these cases exist for $d=3$, and only the case of melonic bubble is known.

%%%%%%%%%%%%%%%
\subsection{Planar bubbles and the maximal 2-cut property} \label{sec:Max2Cut}
%%%%%%%%%%%%%%%

The most important property of planar bubbles in three dimensions that we will use is an elementary property about the lengths of their bicolored cycles.

\begin{lemma} \label{lemma:FacesPlanarBubble}
A planar bubble without bicolored cycles of length 2 has at least six bicolored cycles of length 4.
\end{lemma}

\begin{proof}
A bubble in three dimensions is a colored graph dual to the boundary surface of a CBB. We can thus use Corollaries \ref{cor:2D} and \ref{cor:PlanarBubbles} to identify a planar bubble with its canonical embedding which is a planar map, such that the bicolored cycles of the bubble are the faces of the map. The lemma is then a classical reasoning based on Euler's formula for planar maps $F(B) - E(B) + V(B) = 2$, where we use $F(B)$ to denote the total number of faces of the map for the three pairs of colors $\{a, b\}$ with $1\leq a<b\leq 3$. First, since $B$ is a colored graph whose vertices have degree three, $E(B) = 3 V(B)/2$, hence
\begin{equation}
V(B) = 2 F(B) - 4.
\end{equation}
Then, we count the edges using the faces. Since the latter are bicolored, each edge, say of color $c$, lies on the boundary of exactly two faces, one of colors $\{c, c_1\}$ and one of color $\{c, c_2\}$, such that $\{c, c_1, c_2\} = \{1, 2, 3\}$. Therefore
\begin{equation}
\sum_{i \geq 1} 2i F^{(2i)}(B) = 2E(B) = 3V(B).
\end{equation}
Here $F^{(2i)}(B)$ is the number of faces of degree $2i$. We can combine the two previous equations, while writing $F(B) = \sum_{i\geq 1} F^{(2i)}(B)$, to get $\sum_{i\geq 1} (6 - 2i) F^{(2i)}(B) = 12$. The coefficients change sign for $i\geq 3$, therefore
\begin{equation}
2 F^{(2)}(B) + F^{(4)}(B) = 6 + \sum_{i\geq 3} (3-i) F^{(2i)}(B) \geq 6.
\end{equation}
In particular if there are no faces of degree 2, then $F^{(4)}(B) \geq 6$.
\end{proof}

%%%%%%%%%%%%%%%
\subsubsection{The maximal 2-cut property}
%%%%%%%%%%%%%%%

The idea of our main theorems is that to maximize the number bicolored cycles, one needs to glue bubbles using 2-edge-cuts. To make this more precise, we need the following definitions.

\begin{definition} [$k$-edge-cut incident on a bubble]
An $k$-edge-cut incident on a bubble $B$ or on some vertices of a bubble in a colored graph $G$ is a $k$-edge-cut formed by edges of color 0 which all have one end in $B$ and the other not in $B$.
\end{definition}

Let $\mathcal{E}(G;B)$ be the set of edges of color 0 of $G$ which have one end in $B$ and the other not in $B$.
There is a unique partition of $\mathcal{E}(G; B)$ into edge-cuts incident on $B$. Indeed, removing all edges of $\mathcal{E}(G; B)$ from $G$ disconnects $G$, since only the edges of color 0 which connect two vertices of $B$ are left incident on $B$. This turns $G$ into $B$, decorated with edges of color 0 between some of its vertices, together with $L$ connected components $G_1, \dotsc, G_L$. The set of edges of color 0 which connects $B$ to $G_l$ in $G$ is $\mathcal{E}_l(G;B) \subset \mathcal{E}(G; B)$ and
\begin{equation}
\mathcal{E}(G; B) = \bigcup_{l=1}^L \mathcal{E}_l(G; B),
\end{equation}
which is a disjoint union. If $k_l$ is the number of edges in $\mathcal{E}_l(G; B)$, then those edges form a $k_l$-edge-cut incident on $B$.

Our main theorem is that in order to maximize the number of bicolored cycles, a planar bubble can only be incident to 2-edge-cuts positioned in a particular way. We call this the {\bf maximal 2-cut property}.

\begin{definition} [Maximal 2-cut property]
Let $G\in \cG_{n_1, \dotsc, n_N}(B_1, \dotsc, B_N)$ and a bubble $B\subset G$. We say that $B\subset G$ satisfies the maximal 2-cut property if there exists a pairing $\pi \in \cG^{\max}_1(B)$ such that for any pair of vertices $\{v, \pi(v)\}$, $v\in B$, there is 
\begin{itemize}
\item either an edge of color 0 between them,
\item or two edges of color 0 forming a 2-edge-cut incident on $B$ (i.e. all $\mathcal{E}_l(G; B)$ have size two).
\end{itemize}
This is illustrated here,
\begin{equation}
\begin{array}{c} \includegraphics[scale=.5]{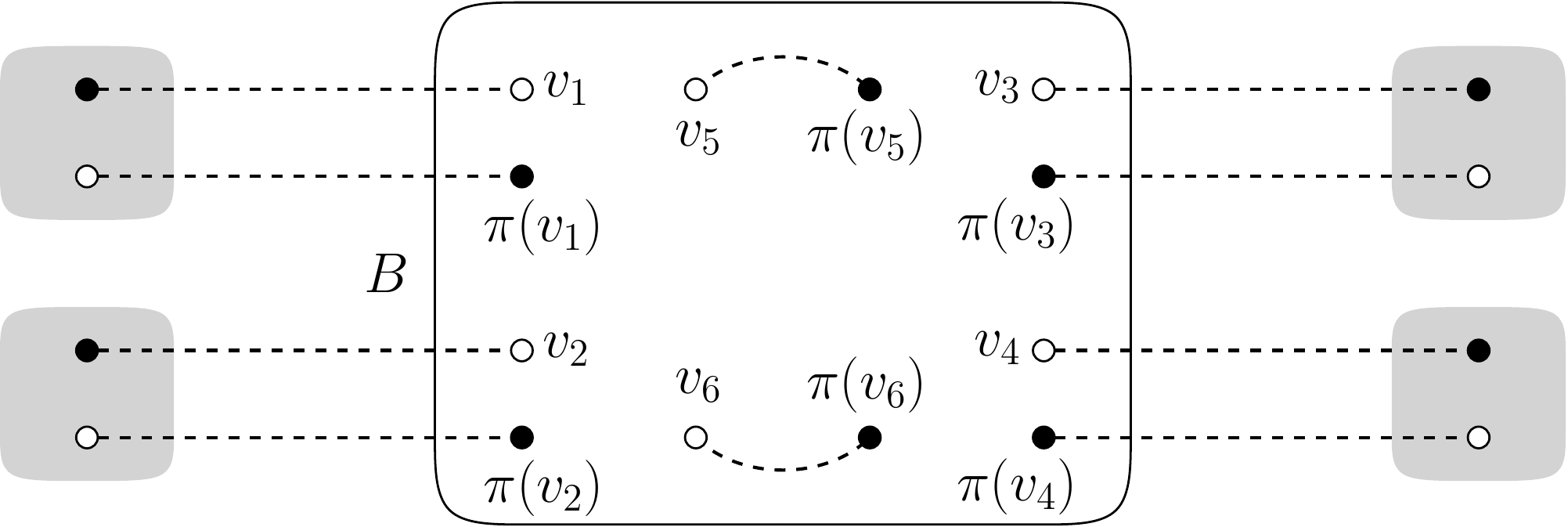} \end{array}
\end{equation}
\end{definition}

%%%%%%%%%%%%%%%
\subsubsection{Edge-cuts incident on a planar bubble}
%%%%%%%%%%%%%%%

\begin{theorem} \label{thm:1Planar}
If $B_i$ is planar for some $i\in\{1, \dotsc, N\}$ and $G\in\cG^{\max}_{n_1, \dotsc, n_N}(B_1, \dotsc, B_N)$, then the $n_i$ copies of $B_i\subset G$ satisfy the maximal 2-cut property.
\end{theorem}

This theorem was previously known to hold when all bubbles $B_1, \dotsc ,B_N$ are melonic \cite{Uncoloring}, and when there is a single type of bubble which is the octahedron \cite{Octahedra}. A method detailed in \cite{SigmaReview} can be used to extend it to more bubbles of the form $\partial H$ when $H$ is a subgraph of some $G\in \cG^{\max}_{n_1, \dotsc, n_N}(B_1, \dotsc, B_N)$. This however does not bring genuinely new cases because the graphs maximizing the number of edges with $\partial H$ form a subset of $\cG^{\max}_{n_1, \dotsc, n_N}(B_1, \dotsc, B_N)$.

Theorem \ref{thm:1Planar} is a thus a far reaching generalization of the existing results since it imposes the single constraint on one bubble to be homeomorphic to the 3-ball. Among all existing results, it only leaves out the case where all bubbles are $K_{3,3}$, \eqref{6VertexBubbles3d}, which has the topology of the torus, so no bubble homeomorphic to the ball. This case was studied in \cite{StuffedWalshMaps} with a conclusion slightly less restrictive than the maximal 2-cut property of Theorem \ref{thm:1Planar}. In fact, from the result for $K_{3,3}$ and the way the Theorem works in the proof below, we can conjecture that a weaker version holds for non-spherical bubbles (notice that their gluings would not be manifolds).

\begin{proof}
We proceed by induction on the total number of vertices of $G$. First notice that the theorem is true for all sets of 1-bubble graphs $\cG^{\max}_1(B)$ (even for non-planar bubbles) by definition of the maximal 2-cut property. There is a single colored graph with two vertices and it has a single bubble (Figure \ref{fig:2VertexBubble}). With four vertices, $G$ can have a single bubble, Figure \ref{fig:4VertexBubble3d}, or two 2-vertex bubbles which can only be connected in one way (where the theorem is true). So the theorem is true for graphs with up to four vertices.

Assume the theorem holds for all sets $\cG_{n'_1, \dotsc, n'_{N'}}(B_1', \dots, B'_{N'})$ with $\sum_{j=1}^{N'} n'_j V(B'_j) < V$ vertices. Consider now $G\in \cG_{n_1, \dotsc, n_N}(B_1, \dotsc, B_N)$, with $\sum_{j=1}^N n_j V(B_j) = V$ vertices, and $B\subset G$ a planar bubble not satisfying the maximal 2-cut property. It means that 
\begin{enumerate}[label=(\alph*)]
\item\label{enum:edge-cut} either there is a $k$-edge-cut incident on $B$ with $k\geq 4$, 
\item\label{enum:pairing} or a pairing $\pi \in \cG_1(B)\setminus \cG_1^{\max}(B)$ such that the vertices of pairs $\{v, \pi(v)\}$ are either connected by edges of color 0 or incident to 2-edge-cuts.
\end{enumerate}

We consider the following cases.
\begin{description}[wide]
\item[2-edge-cut] Assume there is a 2-edge-cut incident on $B$ with edges of color 0 $e, e'$,
\begin{equation} \label{2CutGraph}
G = \begin{array}{c} \includegraphics[scale=.45]{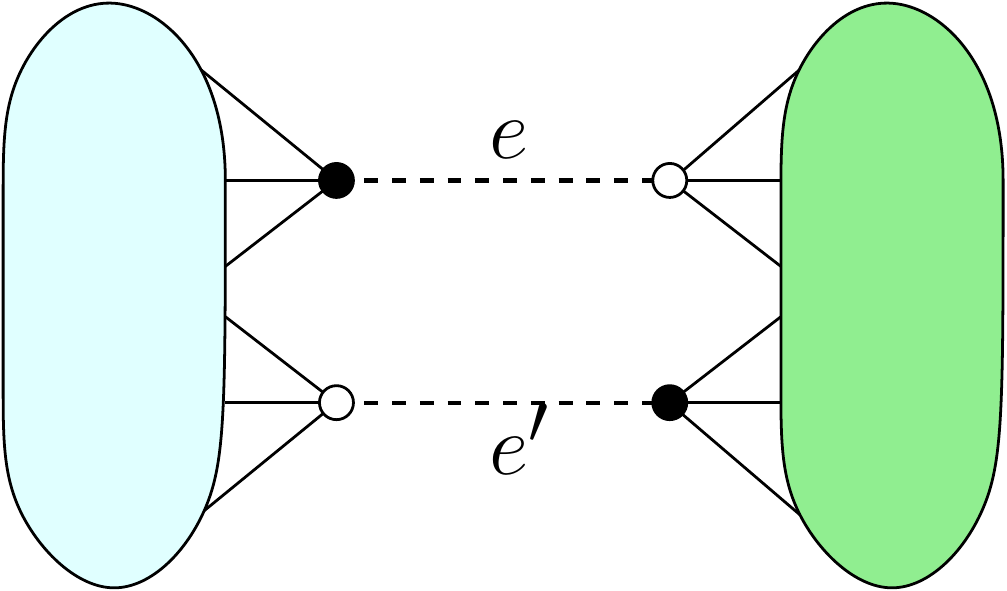} \end{array}
\end{equation}
Then the three bicolored cycles of colors $\{0,1\}, \{0, 2\}, \{0, 3\}$ which go along $e$ also go along $e'$. Therefore, flipping $e$ and $e'$ turns $G$ into two connected graphs $G_1, G_2$ with two new edges $e_1$ in $G_1$, $e_2$ in $G_2$
\begin{equation} \label{2CutFlipped}
G_1 = \begin{array}{c} \includegraphics[scale=.45]{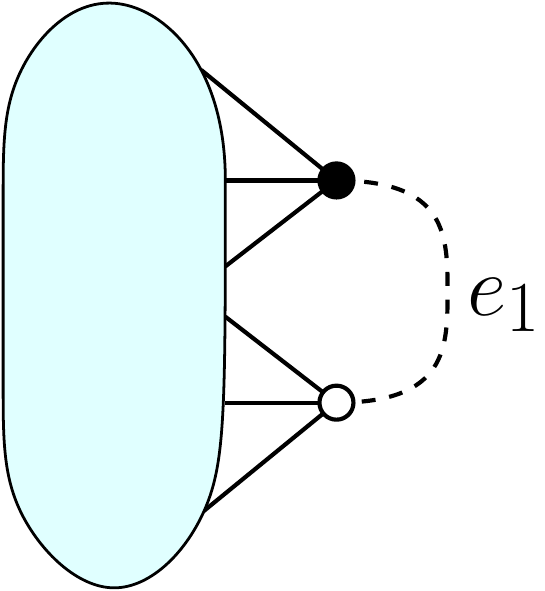} \end{array} \qquad G_2 = \begin{array}{c} \includegraphics[scale=.45]{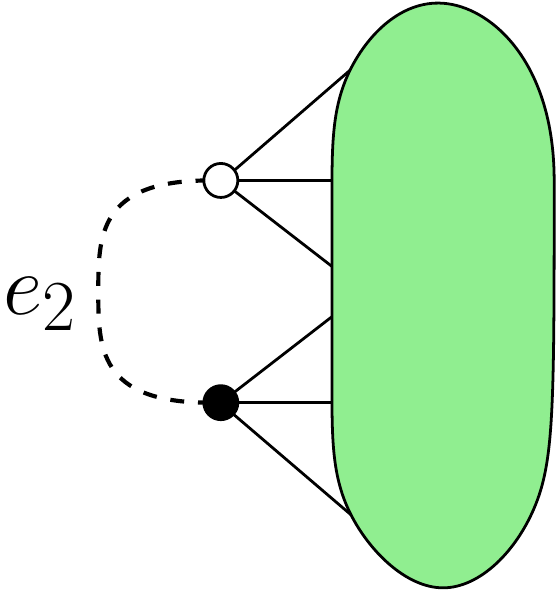} \end{array}
\end{equation}
and satisfying 
\begin{equation}
C_0(G) = C_0(G_1) + C_0(G_2) - 3.
\end{equation}
Assume $B$ is in $G_1$. It still does not satisfy the maximal 2-cut property.
\begin{itemize}
\item Indeed, if it was incident to a $(k\geq 4)$-edge-cut in $G$, it still is in $G_1$.
\item If there were only 2-edge-cuts and edges of color 0 connecting its vertices in $G$ which formed a pairing $\pi\in\cG_1(B)$ not in $\cG^{\max}_1(B)$, then the edges $e, e'$ are incident to vertices $v, v'$ such that $v'=\pi(v)$. Therefore, since the flip connects directly $v$ to $\pi(v)$, the pairing $\pi$ is unchanged in $G_1$.
\end{itemize}
Since $G_1$ has fewer than $V$ vertices and $B\subset G_1$ does not satisfy the maximal 2-cut property, we know from the induction hypothesis that there is a graph $\tilde{G}_1$ with the same bubbles as $G_1$ but where $B$ satisfies the maximal 2-cut property and $C_0(\tilde{G}_1)>C_0(G_1)$. A flip can then be performed between $e_2$ in $G_2$ and any edge of $\tilde{G}_1$ to form a connected graph $\tilde{G}$ with
\begin{equation}
C_0(\tilde{G}) = C_0(\tilde{G}_1) + C_0(G_2) - 3 > C(G).
\end{equation}
This implies that $G\not\in \cG^{\max}_{n_1, \dotsc, n_N}(B_1, \dotsc ,B_N)$.
%If $G$ has a 2-edge-cut, with edges $e, e'$, then all 3 bicolored cycles of colors $(0c)$ which go along $e$ also go along $e'$. Therefore, the flip transforms $G$ into two connected components $G_1, G_2$ and $C(G) = C(G_1) + C(G_2) - 3$. Both $G_1$ and $G_2$ have less than $b$ bubbles and at least one of them does not satisfy the maximal 2-cut property. Assume it is in $G_1$. Then from the induction hypothesis, there exists $\tilde{G}_1$ with the same bubbles as $G_1$ and $C(\tilde{G}_1)>C(G_1)$. A flip can be performed on any pair of edges of $\tilde{G}_1$ and $G_2$ to form 

\item[2-dipole] Assume that $B$ has a 2-dipole, i.e. two parallel edges, say of color 1 and 2 between $v$ and $\bar{v}$. From the 2-edge-cut case above, we know the edges of color 0 incident to $v$ and $\bar{v}$ in $G$ do not form a 2-edge-cut. We are left with two cases.

\begin{itemize}
\item First, suppose there is an edge $e$ of color 0 between $v$ and $v'$
\begin{equation}
G = \begin{array}{c} \includegraphics[scale=.4]{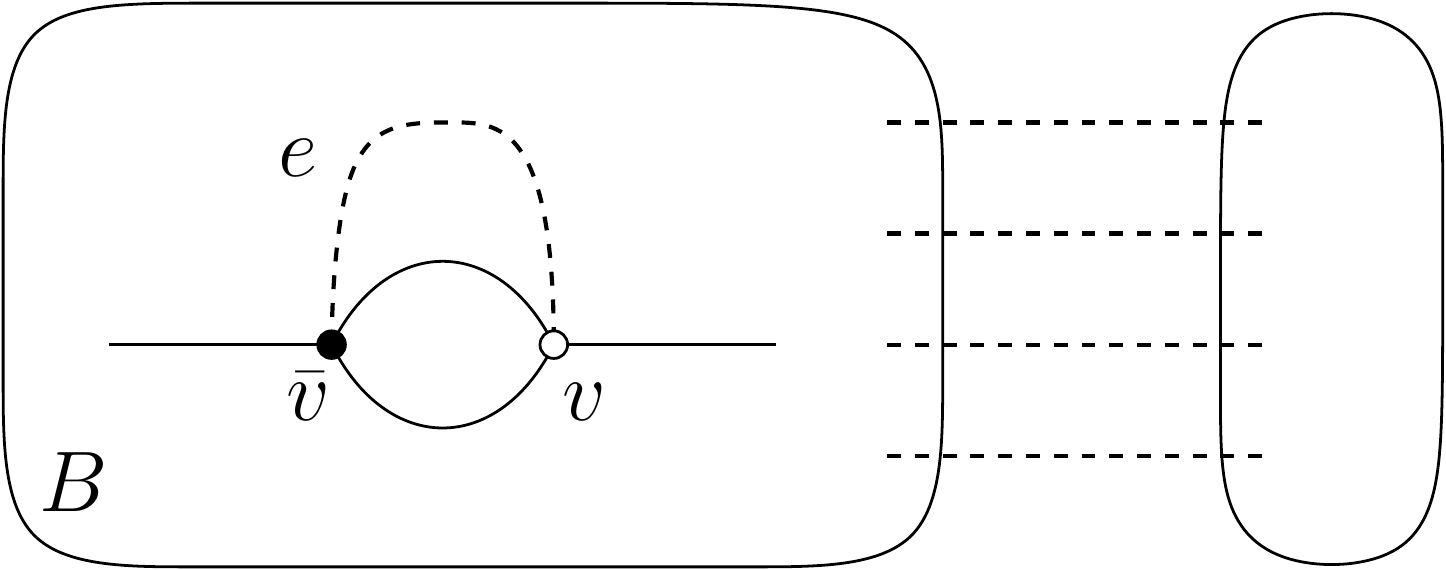} \end{array}
\end{equation}
We are either in situation \ref{enum:edge-cut}, i.e. a $k$-edge-cut for $k\geq 4$ incident on other vertices of $B$ in $G$, or in a refinement of situation \ref{enum:pairing} with a pairing $\pi\in\cG_1(B)$ not in $\cG^{\max}_1(B)$ due to other edges of color 0 and/or 2-edge-cuts incident on $B$. Indeed, $e$ is part of any pairing of $\cG^{\max}_1(B)$ as shown in Lemma \ref{lemma:qEdges}, so that the fact that $\pi\not\in \cG^{\max}_1(B)$ is due to the edges of color 0 incident on the other vertices of $B$.

Consider the contraction $G/e$, in which $B$ is turned into a bubble $B'$. In situation \ref{enum:edge-cut}, $B'\subset G/e$ still has the same incident $k$-edge-cut. In situation \ref{enum:pairing}, $B'$ now has the pairing $\pi'=\pi/e$ which can also be seen as the restriction of $\pi$ to the vertices of $B'$, hence $\pi'\not\in \cG^{\max}_1(B')$ as discussed above. Since $G/e$ has two vertices less than $G$ and $B'\subset G/e$, still planar, does not satisfy the maximal 2-cut property, we can apply the induction hypothesis. It implies that $G/e$ does not maximize the number of bicolored cycles and it can be concluded from Lemma \ref{lemma:NotMax} that $G\not\in\cG^{\max}_{n_1, \dotsc, n_N}(B_1, \dotsc, B_N)$.

\item Suppose instead that there are two different edges $e, e'$ incident to the vertices $v, v'$ of the dipole. Clearly $\{1, 2\} \subset I_G(e, e')$ so that $|I_G(e, e')| \geq 2$. Since $e, e'$ do not form a 2-edge-cut, the flip produces a connected graph $\tilde{G}$. According to Lemma \ref{lemma:flip}, $C_0(\tilde{G})> C_0(G)$, hence $G\not\in \cG^{\max}_{n_1, \dotsc, n_N}(B_1, \dotsc, B_N)$.
\end{itemize}

\item[4-edge-cut] If there is a 4-edge-cut incident on $B$, Proposition \ref{prop:4EdgeCuts} directly applies and shows that $G$ does not have the maximal number of bicolored cycles, $G\not\in\cG^{\max}_{n_1, \dotsc, n_N}(B_1, \dotsc, B_N)$.

\item[$k$-edge-cut, $k\geq 6$] Since $B$ does not have a 2-dipole, it cannot be a melonic bubble. And since it is planar and all planar bubbles with six or less vertices are melonic, $B$ has at least eight vertices.

$B$, being planar and without 2-dipoles, has a face of degree 4, as shown in Lemma \ref{lemma:FacesPlanarBubble}. Denote its vertices $\bar{v}_1, v_2, \bar{v}_3, v_4$. We set the colors of the face to be $\{1,2\}$ with edges of color 1 $e_{12}^{(1)}$ between $\bar{v}_1$, $v_2$ and $e^{(1)}_{34}$ between $\bar{v}_3$, $v_4$, and edges of color 2 $e^{(2)}_{23}$ between $v_2$, $\bar{v}_3$ and $e^{(2)}_{14}$ between $v_4$, $\bar{v}_1$:
\begin{equation} \label{FaceDegree4}
\begin{array}{c} \includegraphics[scale=.5]{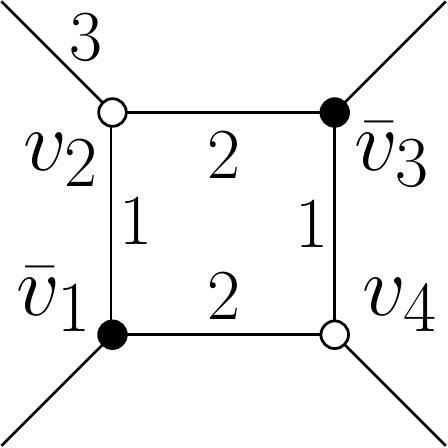} \end{array} \quad \subset B.
\end{equation}

Moreover, as $B$ has no incident 2-edge-cuts or 4-edge-cuts, it is connected to the rest of $G$ by some $k_\alpha$-edge-cuts with $k_\alpha\geq 6$ for all $\alpha$. We will thus study all the cases which depend on whether the edges of color 0 incident to $\bar{v}_1, v_2, \bar{v}_3, v_4$ belong to some of those edge-cuts.

\begin{itemize}
\item If there is an edge of color 0 connecting any two of $\bar{v}_1, v_2, \bar{v}_3, v_4$ while the two other vertices have different edges of color 0, then we can flip these two edges and gain bicolored cycles. Indeed, say $\bar{v}_1$ is connected to $v_2$ by an edge of color 0 and $e_3, e_4$ are two distinct edges incident to $\bar{v}_3, v_4$,
\begin{equation} \label{OneEdgeFaceDegree4}
G = \begin{array}{c} \includegraphics[scale=.4]{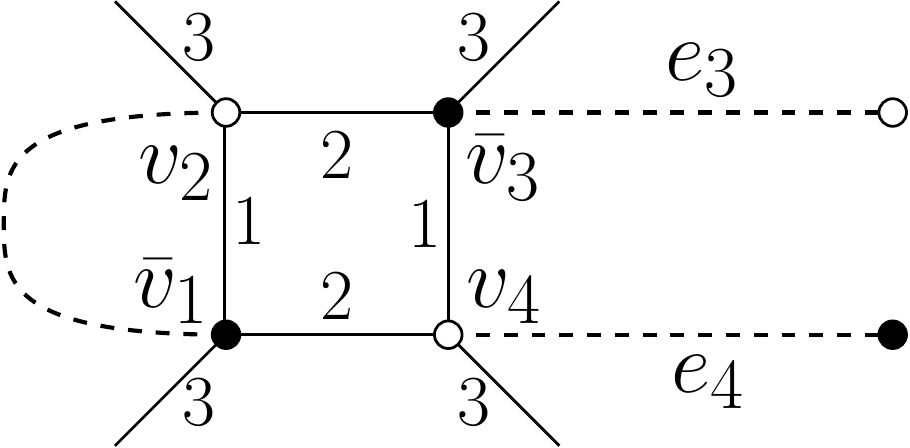} \end{array} \quad \to \quad G' = \begin{array}{c} \includegraphics[scale=.4]{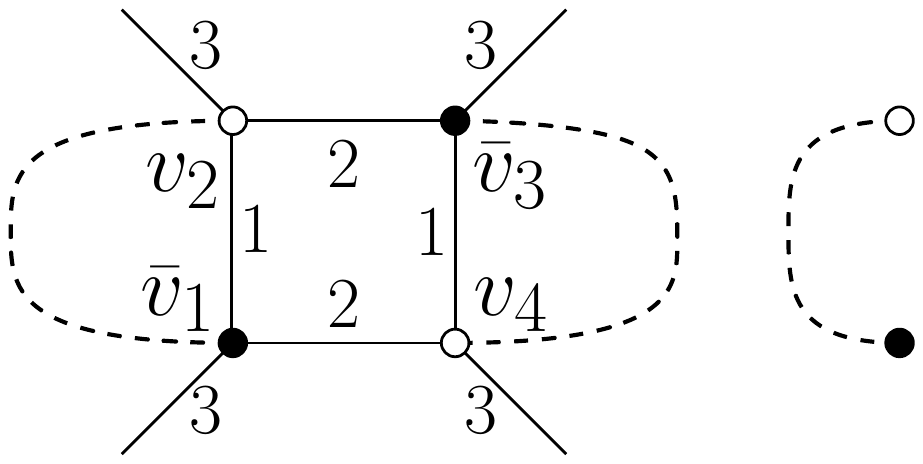} \end{array}
\end{equation}
It is the same bicolored cycle of colors $\{0, c\}$ going along $e_3$ and $e_4$ for $c=1, 2$, so $\{1, 2\} \subset I_G(e_3, e_4)$. According to Lemma \ref{lemma:flip} with $|I_G(e_3, e_4)| \geq 2$, $G'$ after flipping $e_3$ with $e_4$ has
\begin{equation}
G_0(G') \geq C_0(G) + 1
\end{equation}
hence has more bicolored cycles than $G$.

\item {\bf Reference case.} The next case we study will be encountered several times in the remaining of the proof. We will refer to this situation as the reference case. If there is an edge $e_{12}$ of color 0 between $\bar{v}_1$ and $v_2$ and another one $e_{34}$ between $\bar{v}_3$ and $v_4$, then we know there are at least six more vertices and a $(k\geq 6)$-edge-cut incident on $B$ connecting it to a subgraph $H$. There is a face $f_{23}$ of colors $\{2, 3\}$ along the edge of color 2 $e^{(2)}_{23}$ which connects $v_2$ to $\bar{v}_3$, and also a face $f'_{23}$ of colors $\{2, 3\}$ along the edge $e^{(2)}_{14}$ of color 2 which connects $v_4$ to $\bar{v}_1$. We have to distinguish the cases where those two faces of colors $\{2,3\}$ are actually a single face or two different ones. This follows a similar structure to the cases 2 and 3 of the proof of Lemma \ref{lemma:NotMax}.

Let us first assume that they are two different faces, i.e. $f_{23} \neq f'_{23}$, 
\begin{equation} \label{FaceDegree4Case2}
G = \begin{array}{c} \includegraphics[scale=.4]{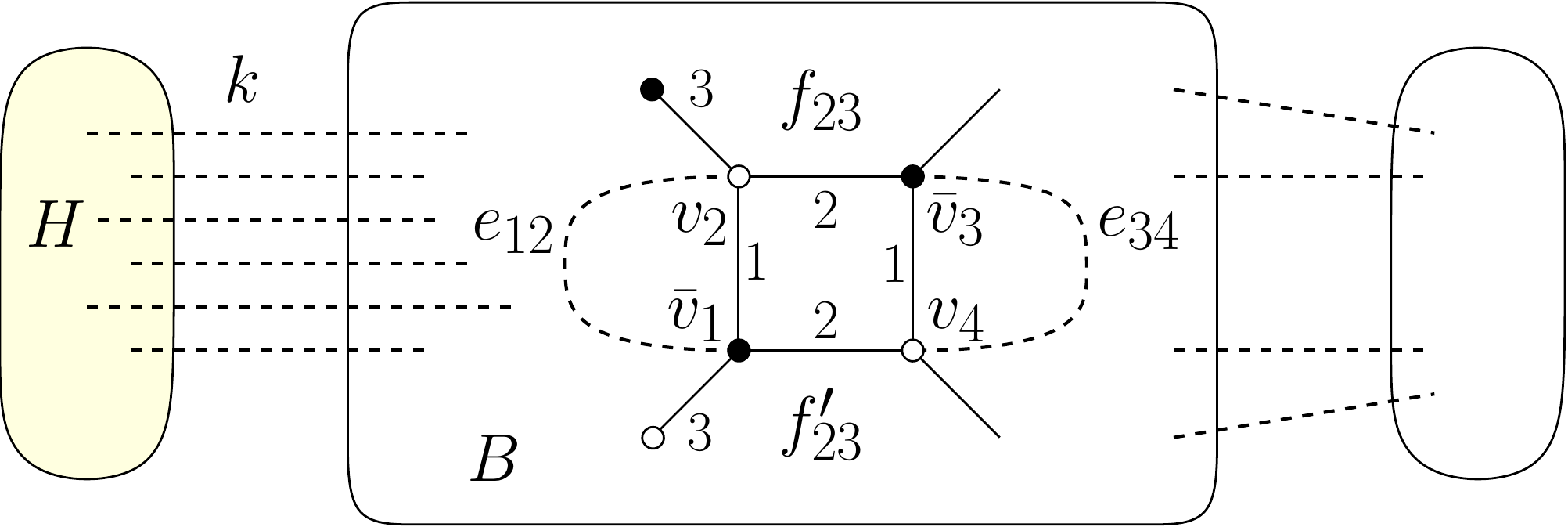} \end{array}
\end{equation}
As the case 2 of the proof of Lemma \ref{lemma:NotMax}, contracting $e_{12}$ does not disconnect $B$, which becomes a bubble $B'\subset G/e_{12}$ and we also know from Proposition \ref{prop:Contraction} that $B'$ is planar. The $(k\geq 6)$-edge-cut incident on $B$ is unaffected by the contraction and is thus still incident to $B'$ in $G/e_{12}$,
\begin{equation} \label{FaceDegree4Case2Contraction}
G/e_{12} = \begin{array}{c} \includegraphics[scale=.4]{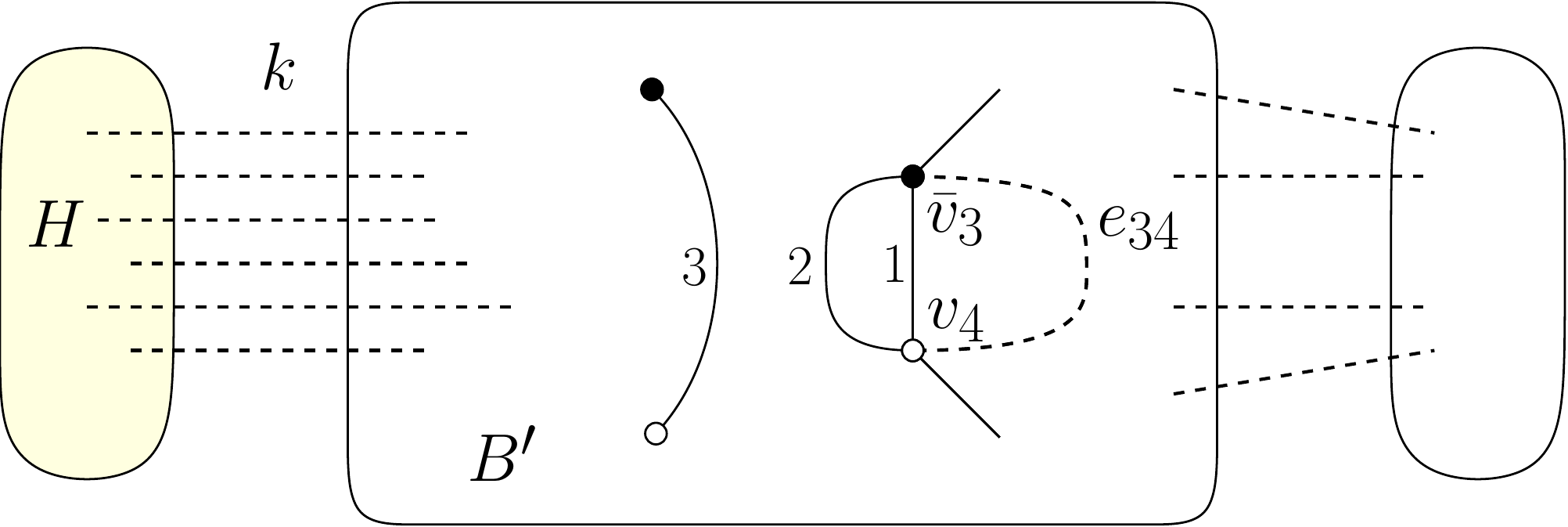} \end{array}
\end{equation}
$G/e_{12}$ has two vertices less than $G$, and $B'$ planar in $G/e_{12}$ is incident on a $(k\geq 6)$-edge-cut and therefore does not satisfy the maximal 2-cut property. From the induction hypothesis, $G/e_{12}$ does not maximize the number of bicolored cycles for its set of bubbles. We can thus conclude from Lemma \ref{lemma:NotMax} that $G\not\in \cG^{\max}_{n_1, \dotsc, n_N}(B_1, \dotsc, B_N)$.

We now consider the same situation with the major difference that the faces $f_{23}$ and $f'_{23}$ are a single one, i.e. $f_{23} = f'_{23}$. This implies that the contraction of $e_{12}$ splits $B$ into two planar bubbles $B_L, B_R$ as in the case 3 of the proof of Lemma \ref{lemma:NotMax}. We can assume that $B$ is of type $B_1$ to have the same notation as in the proof of Lemma \ref{lemma:NotMax}. The vertices $\bar{v}_1, v_2$ play the roles of $v, v'$ and $e_{12}$ of $e$. The graph $G/e_{12}$ then lives in the space $\cG_{\text{case 3}} = \cG_{\text{connected}} \cup \cG_{\text{disconnected}}$ described in \eqref{ContractionSpaceCase3} and it may be disconnected.

Assume first that it is disconnected. It has a connected component $G_L$ which contains $B_L$ and another one $G_R$ which contains $B_R$. If some edges of a $(k\geq 6)$-edge-cut incident on $B$ have vertices in $B_L$ while others have vertices in $B_R$, then $G/e_{12}$ would be connected. Therefore, $G/e_{12}$ being disconnected implies that there is a $(k\geq 6)$-edge-cut incident to $B_L$ (or $B_R$). Then the induction hypothesis establishes that $G/e_{12}$ cannot maximize the number of bicolored cycles in $\cG_{\text{case 3}}$ since we can find a graph $\tilde{G}_L$ with the same bubbles as $G_L$ and more bicolored cycles. We conclude from Lemma \ref{lemma:NotMax} that $G\not\in\cG^{\max}_{n_1, \dotsc, n_N}(B_1, \dotsc, B_N)$.

We now consider the case where $G/e_{12}$ is connected. Assume it maximizes the number of bicolored cycles in the connected part of $\cG_{\text{case 3}}$, i.e. $\cG^{\max}_{1, 1, n_1 -1, \dotsc, n_N}(B_L, B_R, B_1, \dotsc, B_N)$. From our induction hypothesis, we know that $B_L$ and $B_R$ satisfy the maximal 2-cut property. It implies that there is a pair of edges $\{e, e'\}$ forming a 2-edge-cut incident on $B_L$ such that all paths from $B_L$ to $B_R$ go along $e$ or $e'$. Flipping $e$ with $e'$ produces a disconnected graph $\tilde{G}'\in\cG_{\text{disconnected}}$ whose number of bicolored cycles is
\begin{equation}
C_0(\tilde{G}') = C_0(G/e_{12}) +3,
\end{equation}
since $I_{G/e_{12}}(e, e') = \{1, 2, 3\}$ and using Lemma \ref{lemma:flip}. This shows that maximizing the number of bicolored cycles in $\cG_{\text{case 3}}$, i.e. when contracting $e_{12}$ splits $B$ into two connected components, implies that the graph also must have two connected components. We conclude again from Lemma \ref{lemma:NotMax} that $G\not\in\cG^{\max}_{n_1, \dotsc, n_N}(B_1, \dotsc, B_N)$.
\end{itemize}

Only remain to be explored the cases where the edges of color 0 $e_1, e_2, e_3, e_4$ incident on $\bar{v}_1, v_2, \bar{v}_3, v_4$ are all different,
\begin{equation}
G = \begin{array}{c} \includegraphics[scale=.4]{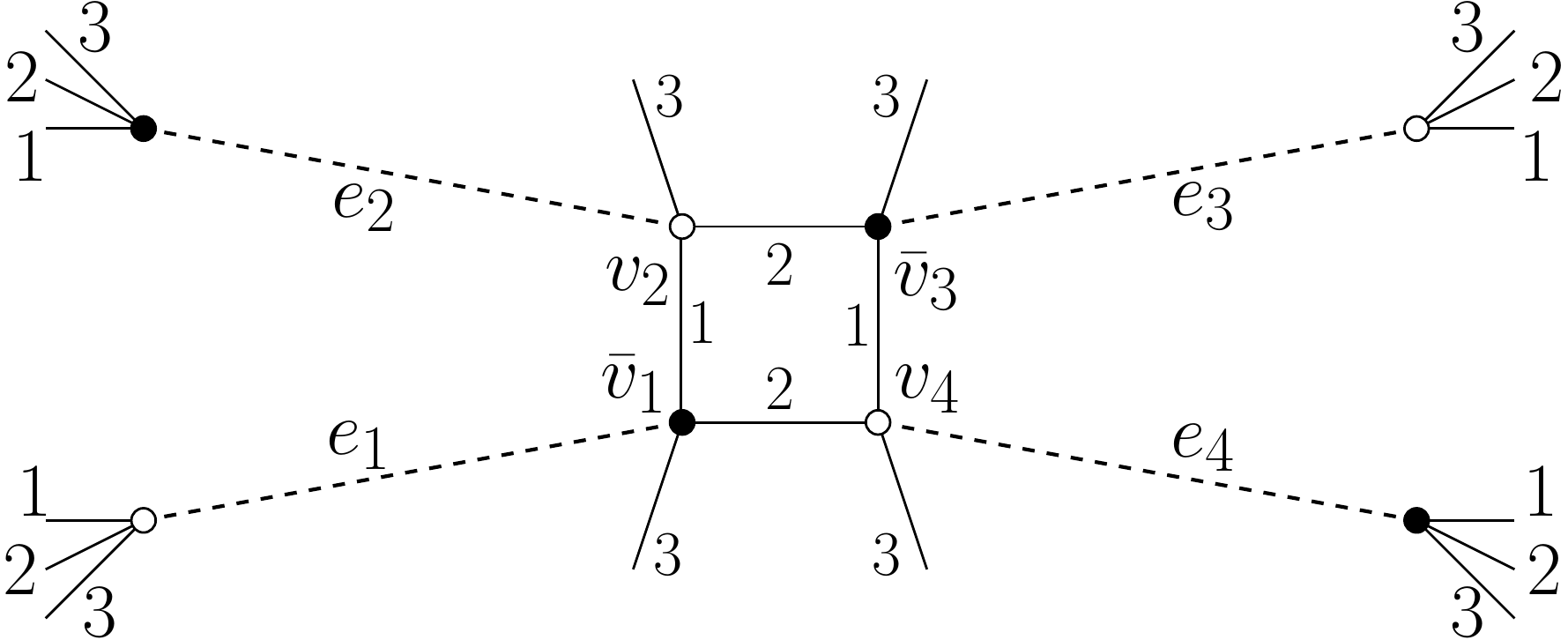} \end{array}
\end{equation}
For each $i=1, 2, 3, 4$, the edge $e_i$ either connects to another vertex of $B$, different of $\bar{v}_1, v_2, \bar{v}_3, v_4$, or is part of a $(k\geq 6)$-edge-cut incident on $B$. This gives sixteen possibilities to study, many of which are equivalent. In most cases, the strategy is to find some flips which bring us back to the reference case studied above. This is easily done since we can for instance flip $e_1$ with $e_2$ and $e_3$ with $e_4$, or else $e_1$ with $e_4$ and $e_2$ with $e_3$. 
%If we flip $e_i$ with $e_j$, we use the notation $G_{ij}$ for the resulting graph, and $G_{ij, kl}$ if we further flip $e_k$ with $e_l$. In all cases, those graphs are connected since $e_1, e_2, e_3, e_4$ do not form any 2-edge-cuts or 4-edge-cuts. Therefore
%\begin{equation}
%G_{ij}, G_{ij, kl} \in\cG_{n_1, \dotsc, n_N}(B_1, \dotsc, B_N).
%\end{equation}
For instance, if $e_3$ and $e_4$ are flipped first 
\begin{equation} \label{OneFlip}
G_{\mid} = \begin{array}{c} \includegraphics[scale=.4]{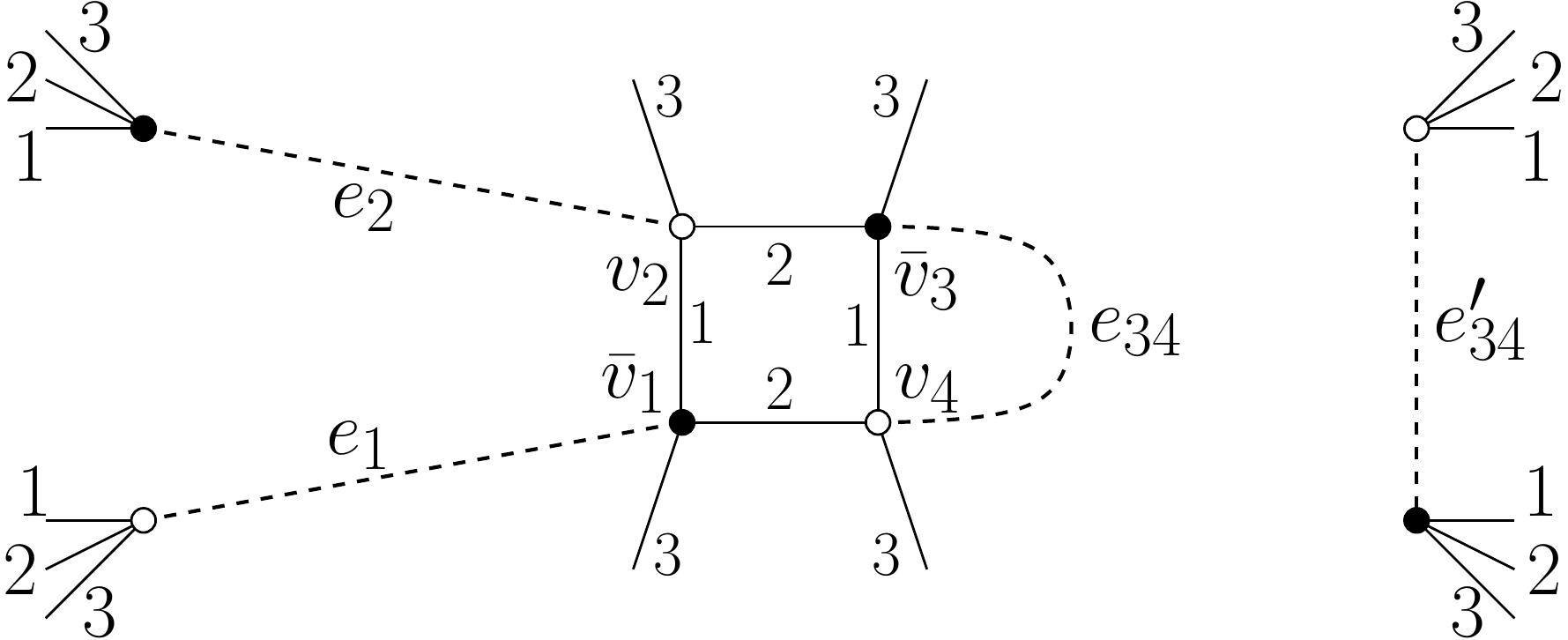} \end{array} 
\end{equation}
and then $e_1$ with $e_2$,
\begin{equation} \label{TwoFlips}
G_{\parallel} = \begin{array}{c} \includegraphics[scale=.4]{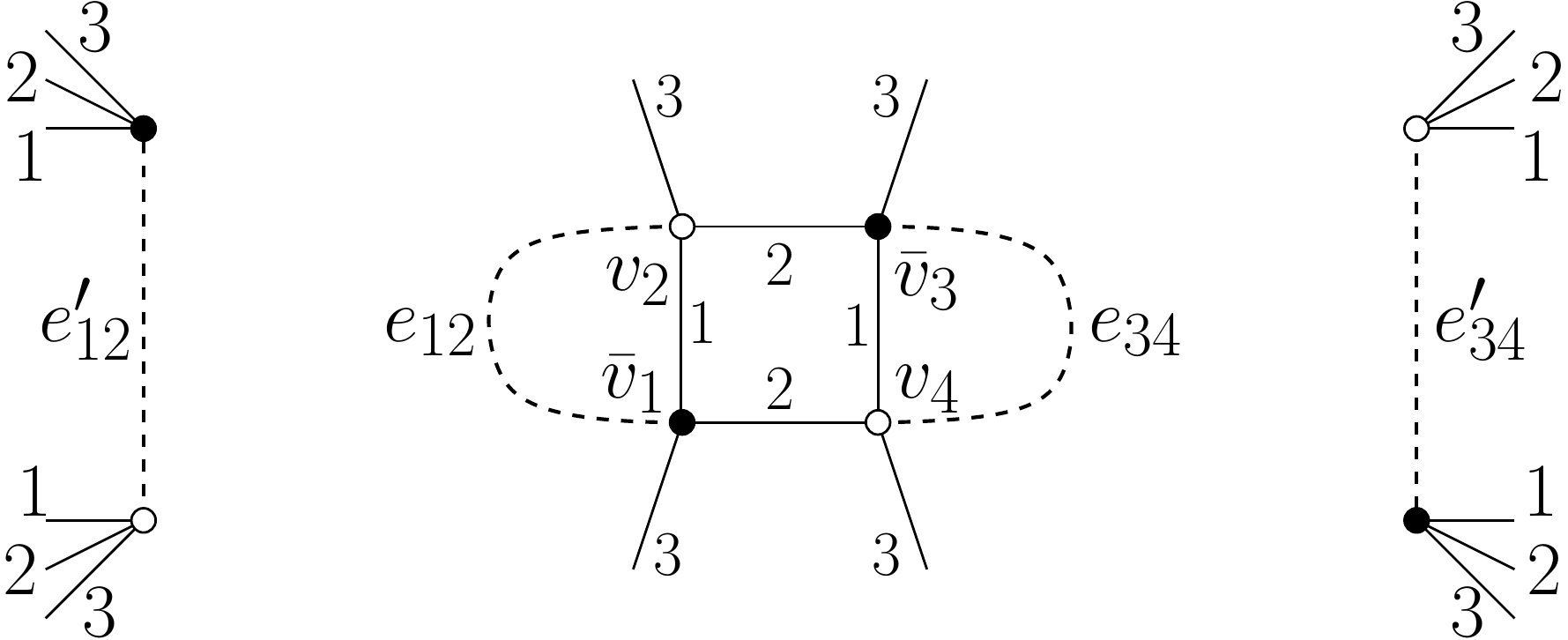} \end{array}
\end{equation}
Both $G_{\mid}$ and $G_{\parallel}$ are connected since $e_1, e_2, e_3, e_4$ do not form any 2-edge-cuts or 4-edge-cuts. Let us study the variation of the number of bicolored cycles through the flip sequence and show that overall it never decreases the number of bicolored cycles. Indeed, $\{1\}\subset I_G(e_3, e_4)$ so that after flipping $e_3$ with $e_4$, $G_{\mid}$ gains one bicolored cycle with colors $\{0,1\}$ but it may lose one bicolored cycle with colors $\{0,2\}$ and one with colors $\{0,3\}$. Therefore
\begin{equation}
C_0(G_{\mid}) \geq C_0(G) -1,
\end{equation}
see Lemma \ref{lemma:flip}. So this flip alone can decrease the number of bicolored cycles. Then however in $G_{\mid}$, we have $\{1, 2\}\subset I_{G_{\mid}}(e_1, e_2)$ so that after flipping $e_1$ with $e_2$, one gets $G_{\parallel}$ with one more bicolored cycle with colors $\{0,1\}$ and one more with colors $\{0,2\}$. Therefore $C_0(G_{\parallel}) \geq C_0(G_{\mid}) +1$, see Lemma \ref{lemma:flip}, hence
\begin{equation} \label{TwoFlipsBound}
C_0(G_{\parallel}) \geq C_0(G).
\end{equation}

A similar sequence of flips leads to
\begin{equation} 
G_{=} = \begin{array}{c} \includegraphics[scale=.4]{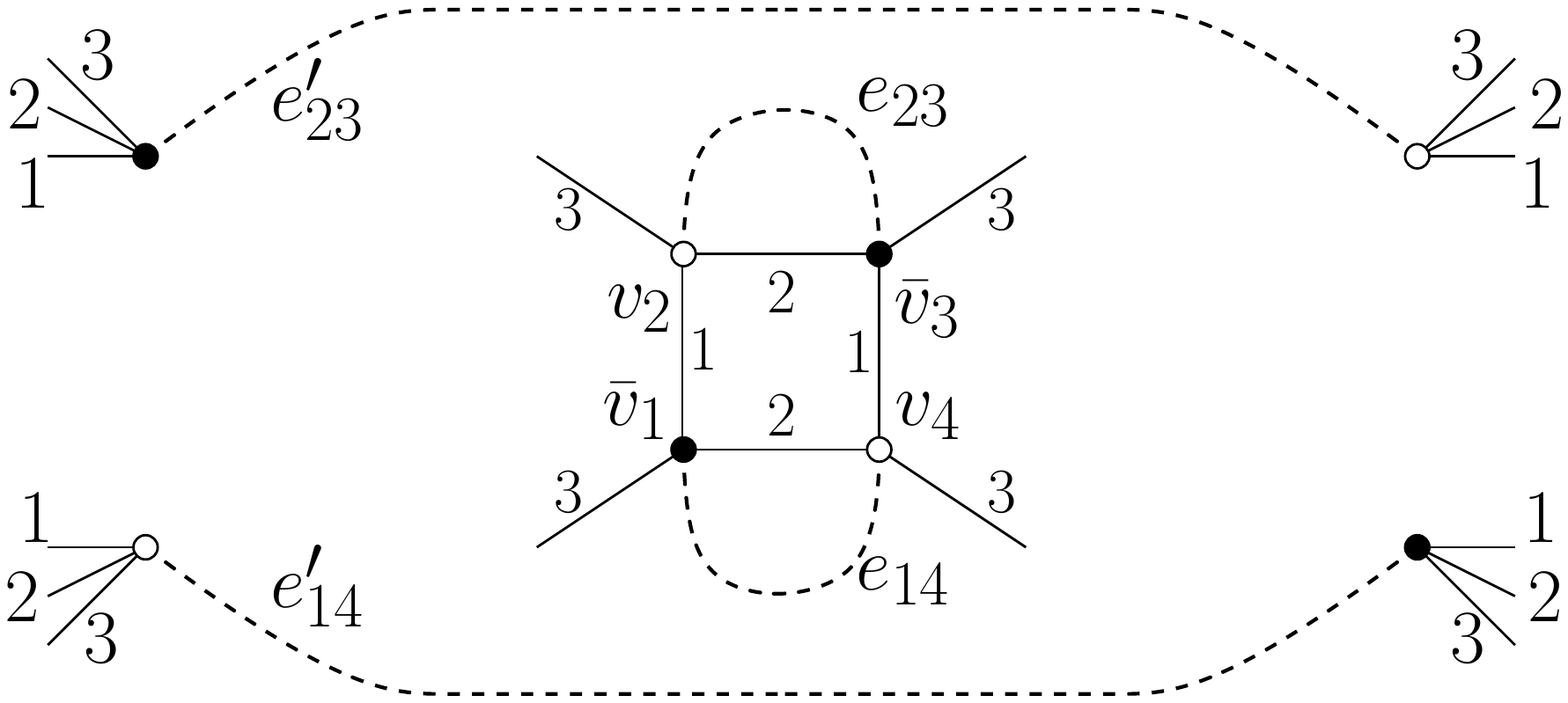} \end{array}
\end{equation}
which satisfies the same bound, $C_0(G_{=}) \geq C_0(G)$.

Obviously, when the inequality is strict for $G_{=}$ or $G_{\parallel}$, this is over. But we also need to account for the possibility of the equality.
\begin{itemize}
\item If $e_1, e_2, e_3, e_4$ are all connected to other vertices of $B$, then they are not part of any of the existing $(k\geq 6)$-edge-cuts incident on $B$,
\begin{equation} \label{FaceDegree4Case3}
G = \begin{array}{c} \includegraphics[scale=.4]{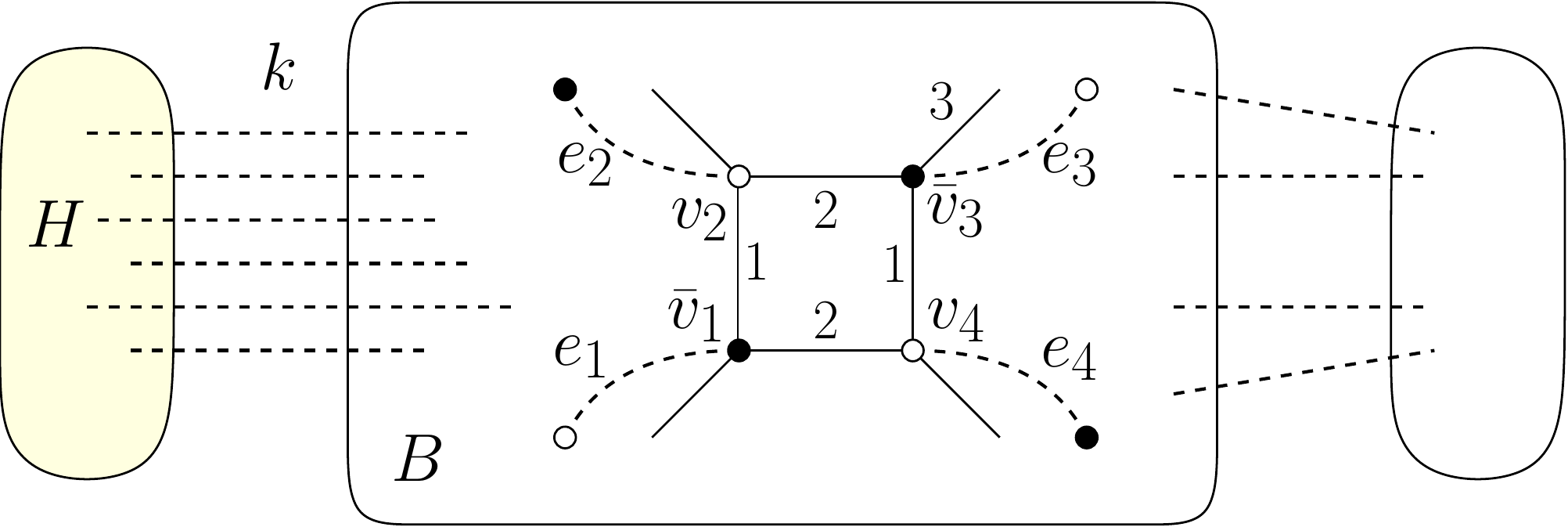} \end{array}
\end{equation}
Therefore, after both flips, $G_{\parallel}$ is exactly of the type described in the reference case.

\item All remaining situations have at least one edge among $e_1, e_2, e_3, e_4$ part of an edge-cut. Consider the case where the edge which belongs in an edge-cut incident on $B$, say $e_1$, has as neighbor in the face of degree 4 a vertex, $v_2$ or $v_4$, which is connected to a vertex in $B$ (via $e_2$ or $e_4$). The other two edges incident on the face of degree 4 can be connected to any vertex of $G\setminus H$. For instance,
%by an edge of color 1 or 2 to another edge $e_j$ which is itself connected to another vertex of $B$, while the two other edges $e_k, e_l$ are either connected to other vertices of $B$ or part of other edge-cuts. For instance, with $e_i = e_1$ connected to $e_j=e_2$ by an edge of color 1 and $e_k=e_3, e_l=e_4$,
\begin{equation} \label{FaceDegree4OneEdgeCut}
G = \begin{array}{c} \includegraphics[scale=.4]{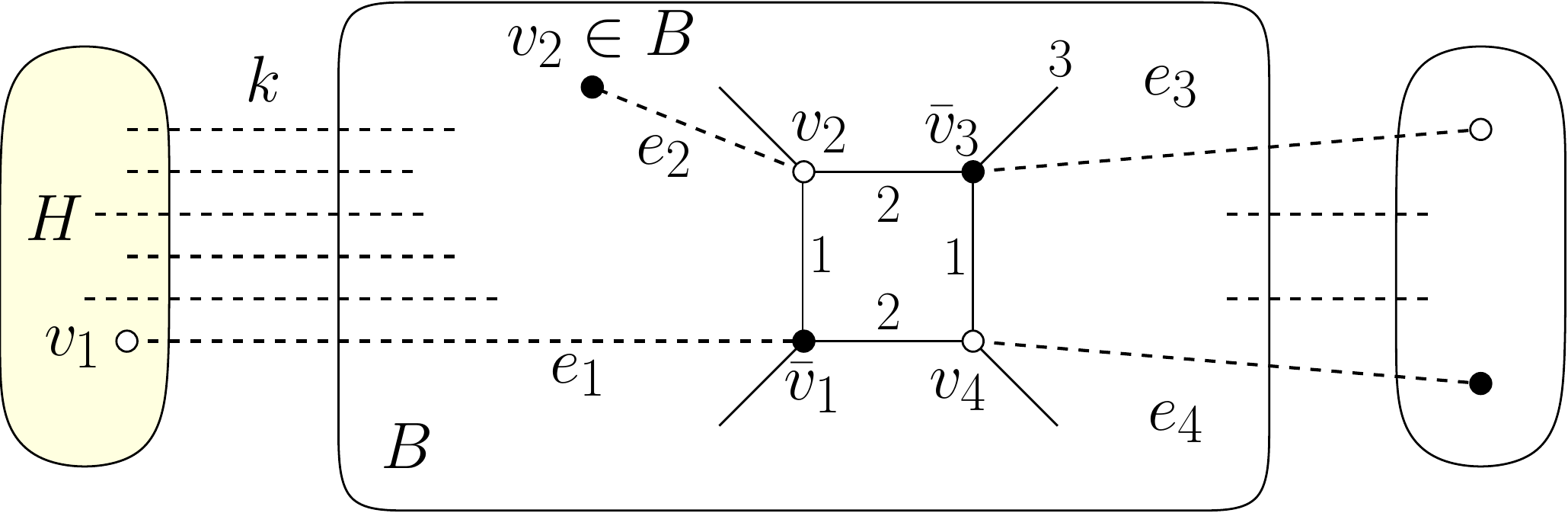} \end{array}
\end{equation}
The flip between $e_1$ and $e_2$ destroys $e_1$. Instead, one now has a new edge $e'_{12}$ connecting $B$ to $H$. In the above example, the graph $G_{\parallel}$ thus looks like
\begin{equation} \label{FaceDegree4OneEdgeCutFlipped}
G_{\parallel} = \begin{array}{c} \includegraphics[scale=.4]{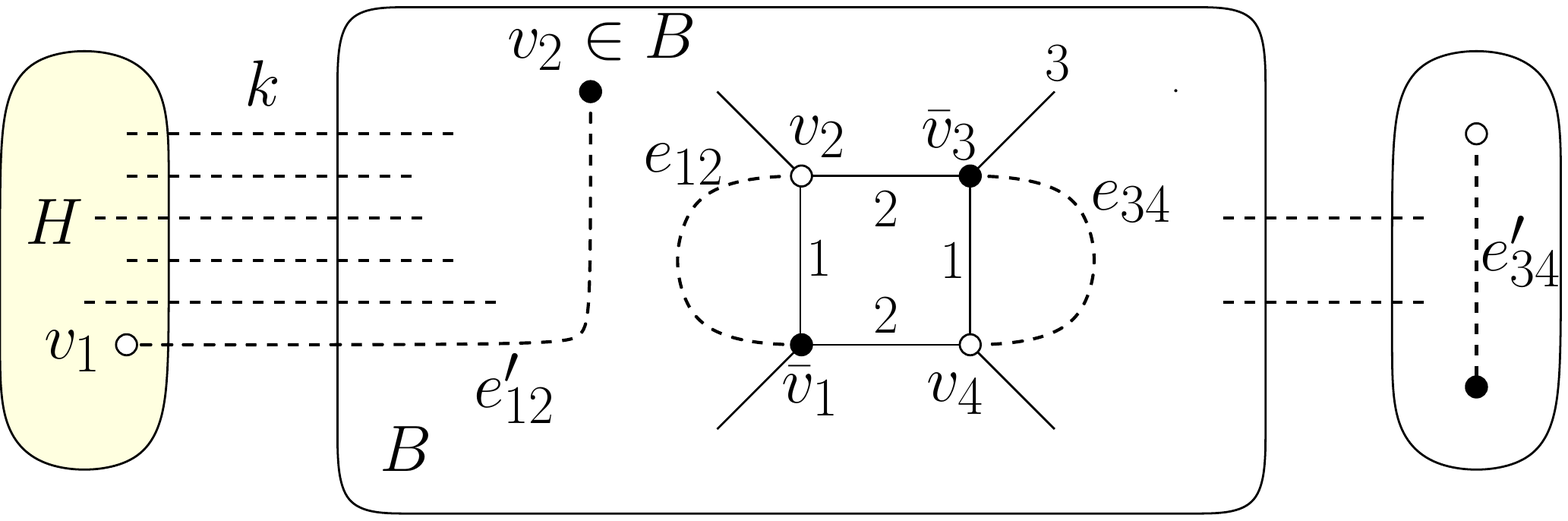} \end{array}
\end{equation}
which is exactly the situation described in the reference case.% In particular, this directly applies if only $e_i$ is connected to $H$ while $e_j, e_k, e_l$ are connected to other vertices of $B$.

\item Consider the cases where two edges among $e_1, e_2, e_3, e_4$ are part of edge-cuts incident on $B$ while the other two are connected to $B$. If they are not part of the same edge-cut, then this fits into the case studied just above, in \eqref{FaceDegree4OneEdgeCut}, \eqref{FaceDegree4OneEdgeCutFlipped}. Indeed, say, without loss of generality, that $e_1$ is part of such a cut, then either $e_2$ or $e_4$ is connected to $B$ while the two other edges are not part of the same edge-cut. We can use $G_{\parallel}$ as in \eqref{FaceDegree4OneEdgeCutFlipped} if this is $e_2$ or $G_=$ if this is $e_4$ which is connected to $B$.

However, if two edges are part of the same edge-cut, it can be one of two cases,
\begin{equation}
G = \begin{array}{c} \includegraphics[scale=.4]{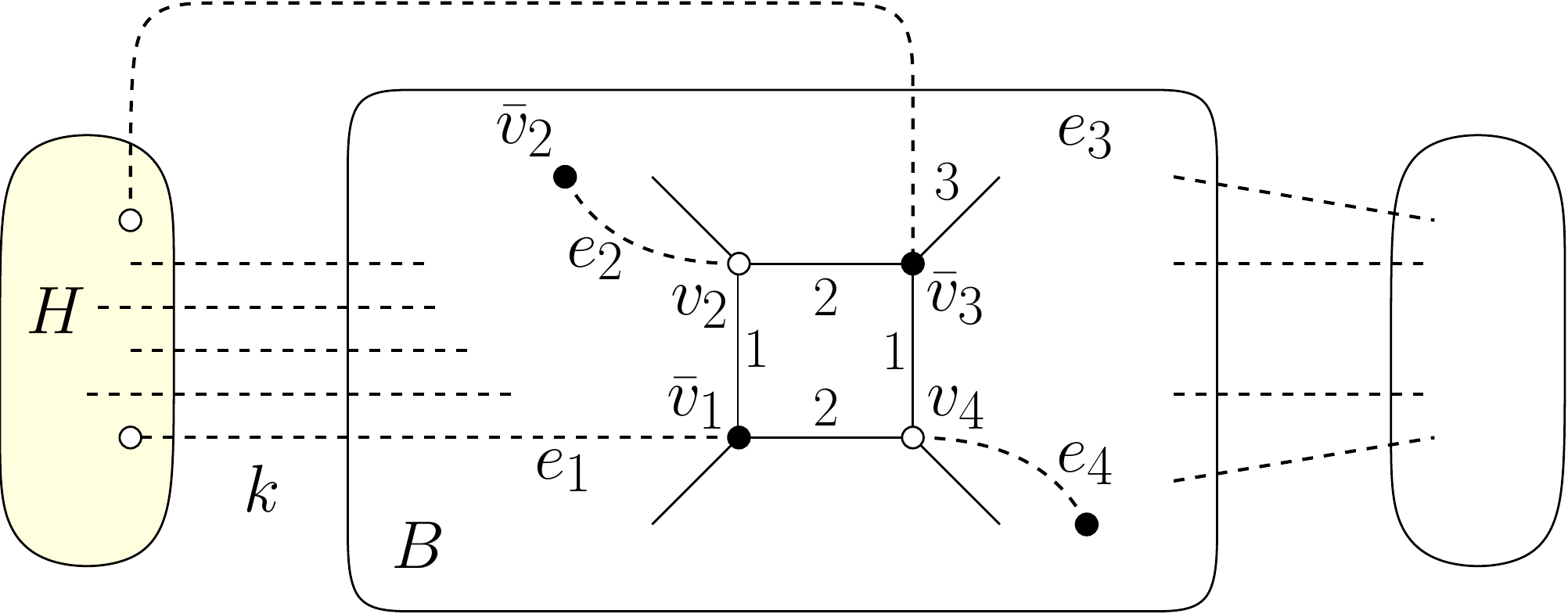} \end{array}
\end{equation}
which is pretty much like \eqref{FaceDegree4OneEdgeCut}, \eqref{FaceDegree4OneEdgeCutFlipped} for both flips between $e_1$ and $e_2$ and between $e_3$ and $e_4$, and $G_{\parallel}$ is exactly like in the reference case, or
\begin{equation} \label{FaceDegree4TwoEdgesOneCut}
G = \begin{array}{c} \includegraphics[scale=.4]{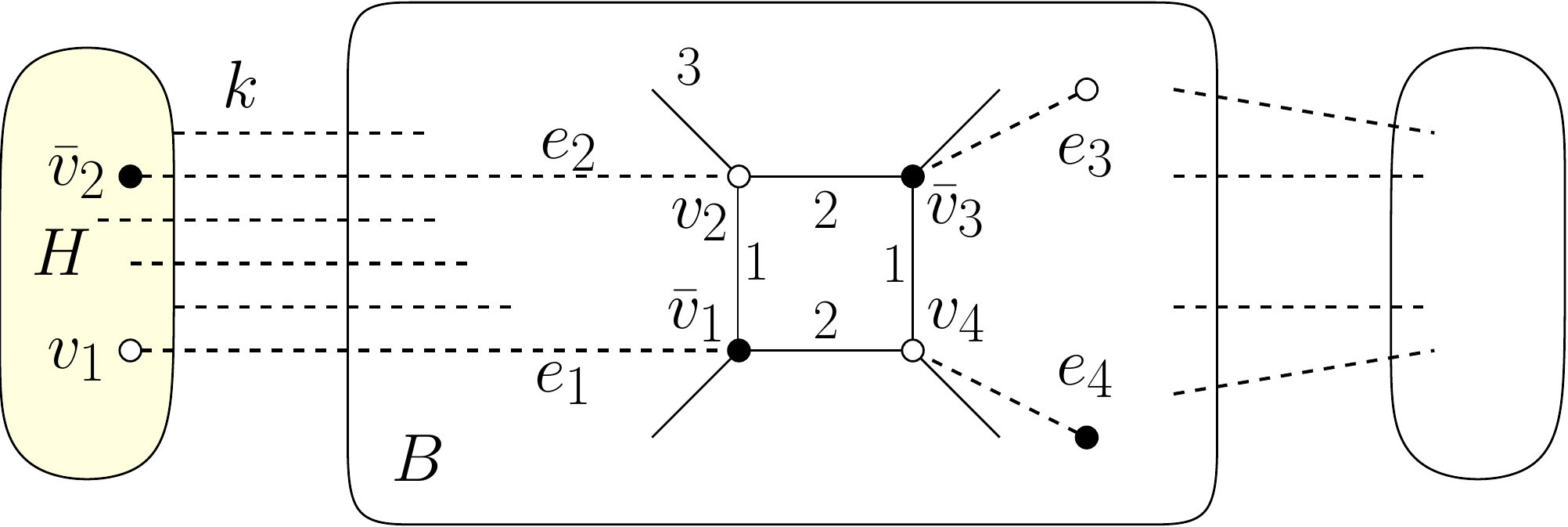} \end{array}
\end{equation}
After the flip of $e_1$ with $e_2$, $H$ is only connected to $B$ via $k-2$ edges. If $k\geq 8$, the graph $G_{\parallel}$ is once again described by the reference case. If $k=6$, $B$ is incident to a 4-edge-cut, and thus cannot maximize the number of bicolored cycles according to Proposition \ref{prop:4EdgeCuts}.

\item The next cases are as follows: among $e_1, e_2, e_3, e_4$, only one edge, say $e_2$ is connected to another vertex of $B$ while the three others $e_1, e_3, e_4$ are part of edge-cuts. They can belong in a single edge-cut incident to $B$, or two different edge-cuts, or three different edge-cuts. If they belong to three different edge-cuts, this is a situation already described in \eqref{FaceDegree4OneEdgeCut}. If they belong to two different edge-cuts, the only situation not described by \eqref{FaceDegree4OneEdgeCut} is when the two edges which are part of the same cut are incident to opposite vertices of the face, like
\begin{equation} \label{FaceDegree4ThreeEdgesTwoCuts}
G = \begin{array}{c} \includegraphics[scale=.4]{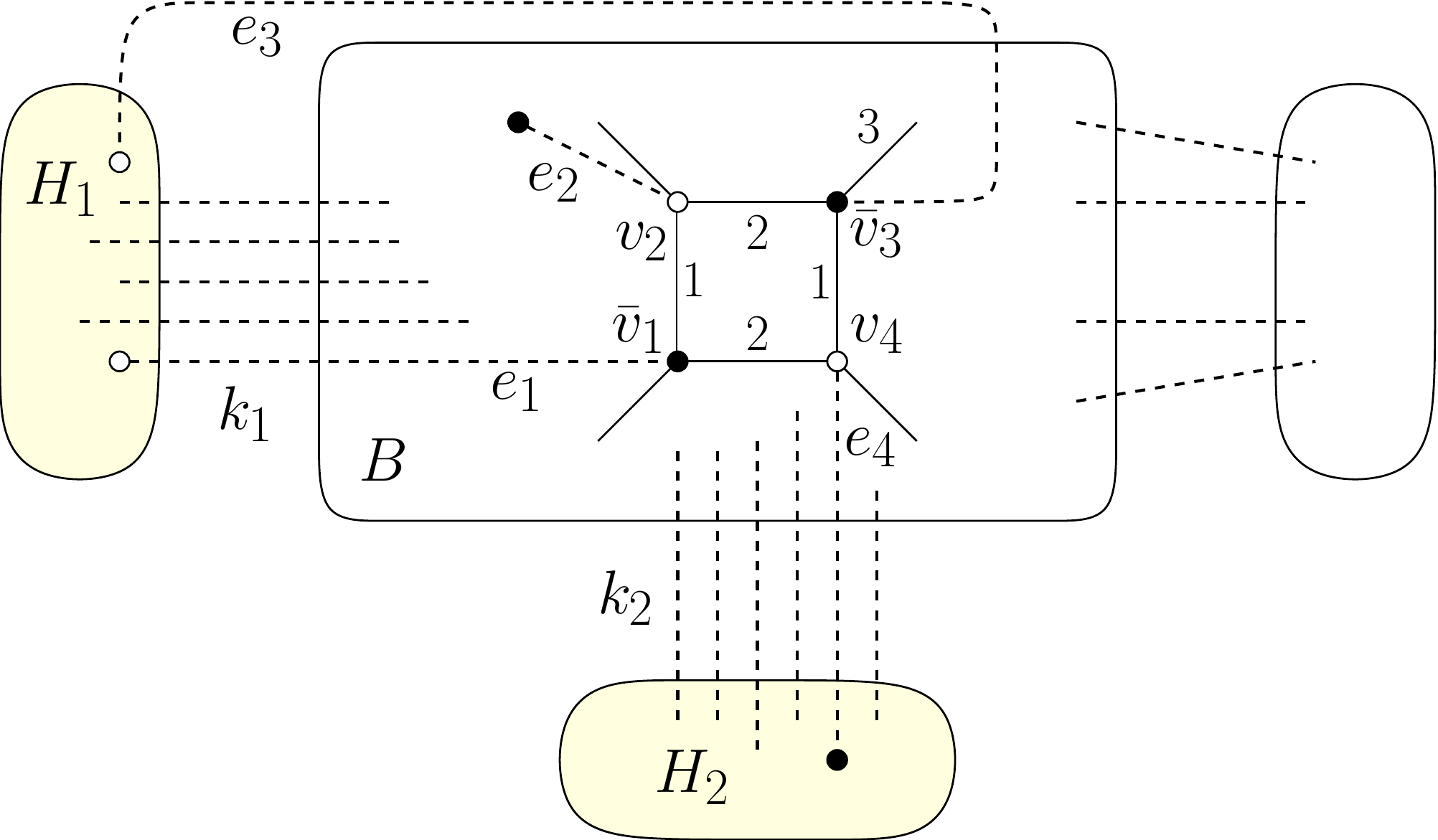} \end{array}
\end{equation}
Indeed, if they are incident to adjacent vertices of the face, like $e_3$ and $e_4$ in the same edge-cut, then $e_1$ and $e_2$ satisfy the hypothesis used in \eqref{FaceDegree4OneEdgeCut} (up to exchanging the roles of $e_1$ and $e_2$). Flipping the edges in \eqref{FaceDegree4ThreeEdgesTwoCuts} produces the same situation for both $G_{\parallel}$ or $G_{=}$. For instance
\begin{equation}
G_{\parallel} = \begin{array}{c} \includegraphics[scale=.4]{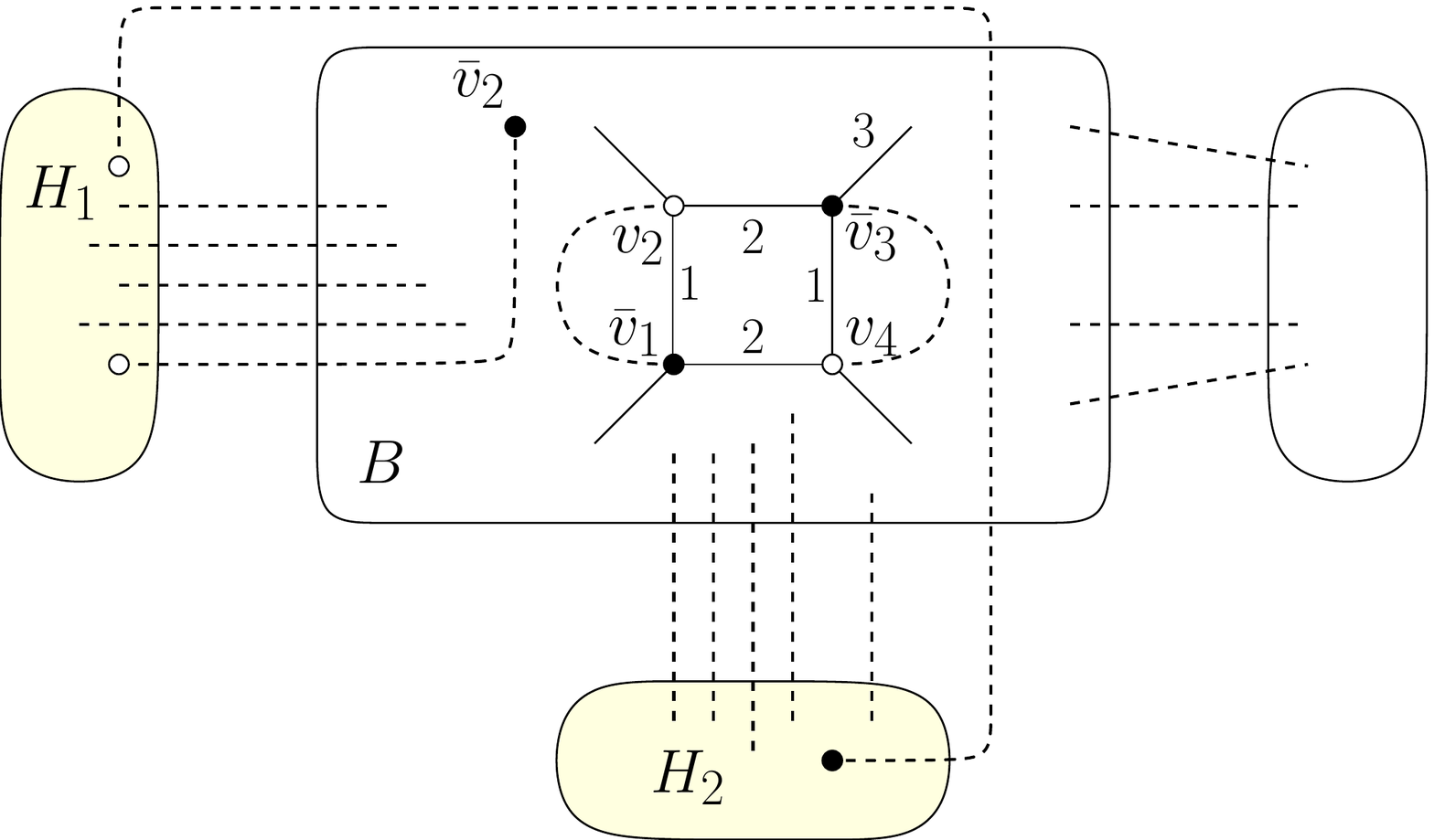} \end{array}
\end{equation}
Clearly $B$ is incident to a $(k_1+k_2-2 \geq 10)$-edge-cut in $G_{\parallel}$.

If $e_1, e_3, e_4$ belong to a single $k$-edge-cut incident to $B$, it is easy to see that $G_{\parallel}$ has a $(k-2)$-edge-cut incident to $B$. For $k\geq 8$, we can apply the induction hypothesis and for $k=6$, we can conclude using Proposition \ref{prop:4EdgeCuts}.

\item The last set of situations is when $e_1, e_2, e_3, e_4$ are all part of edge-cuts incident to $B$. Again, we have to distinguish several cases, depending on which edge belong to which edge-cut with which other edges. If they belong to four different edge-cuts with $k_1, k_2, k_3, k_4$ edges, it is easy to see that in $G_{\parallel}$, $B$ is still incident to a $(k_1+k_2-2\geq 10)$-edge-cut, and a $(k_3+k_4-2\geq 10)$-edge-cut, so we can conclude using the induction hypothesis and Lemma \ref{lemma:NotMax}.

If $e_1, e_2, e_3, e_4$ belong to three different edge-cuts incident on $B$, then two edges must belong to the same edge-cut. They have vertices either on adjacent vertices of the face, like
\begin{equation}
G = \begin{array}{c} \includegraphics[scale=.4]{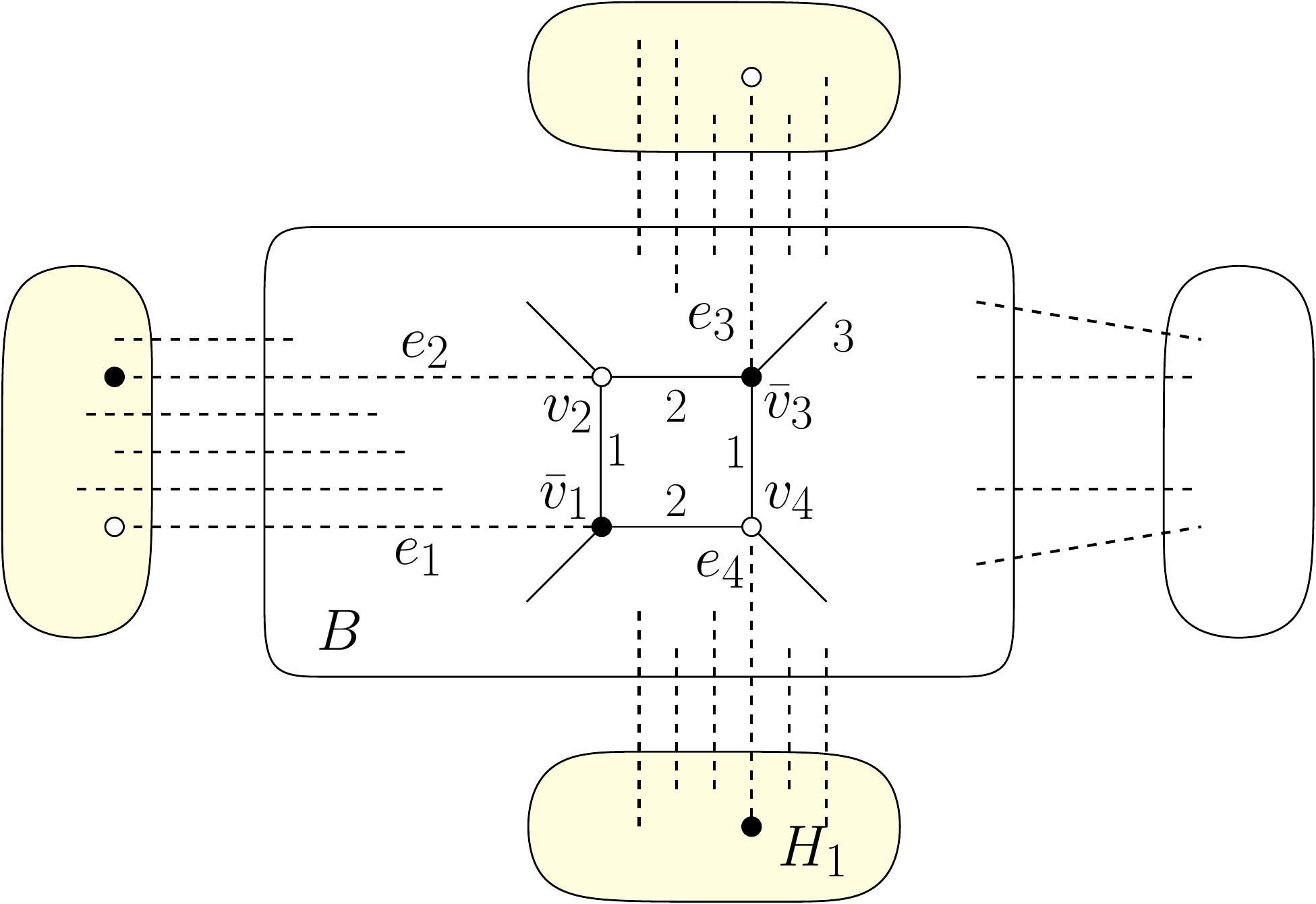} \end{array}
\end{equation}
or on opposite vertices of the face, like
\begin{equation}
G = \begin{array}{c} \includegraphics[scale=.4]{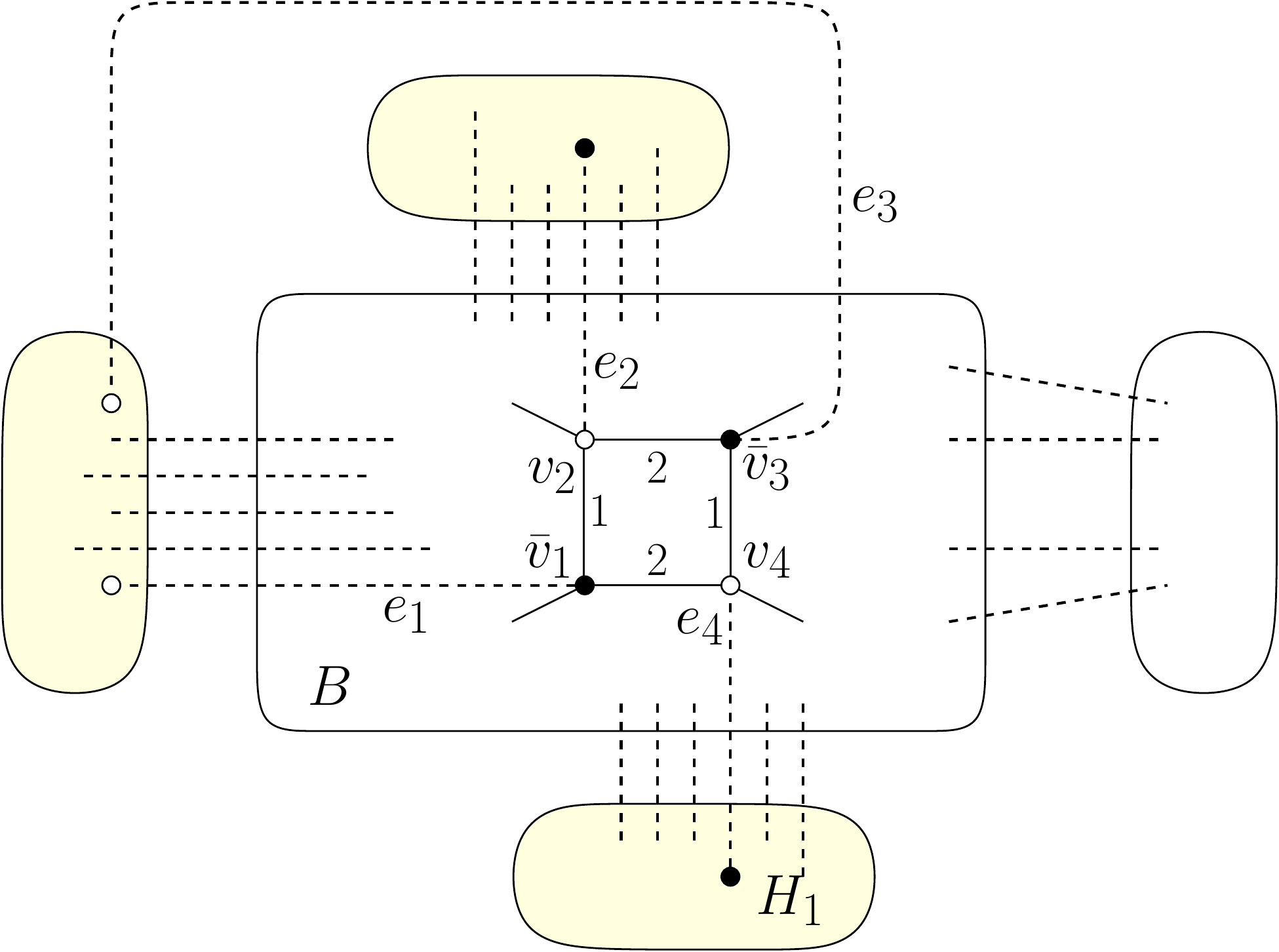} \end{array}
\end{equation}
The first case is similar to some previously seen, like \eqref{FaceDegree4TwoEdgesOneCut}. In the second case, flipping the edges gives
\begin{equation}
G_{=} = \begin{array}{c} \includegraphics[scale=.4]{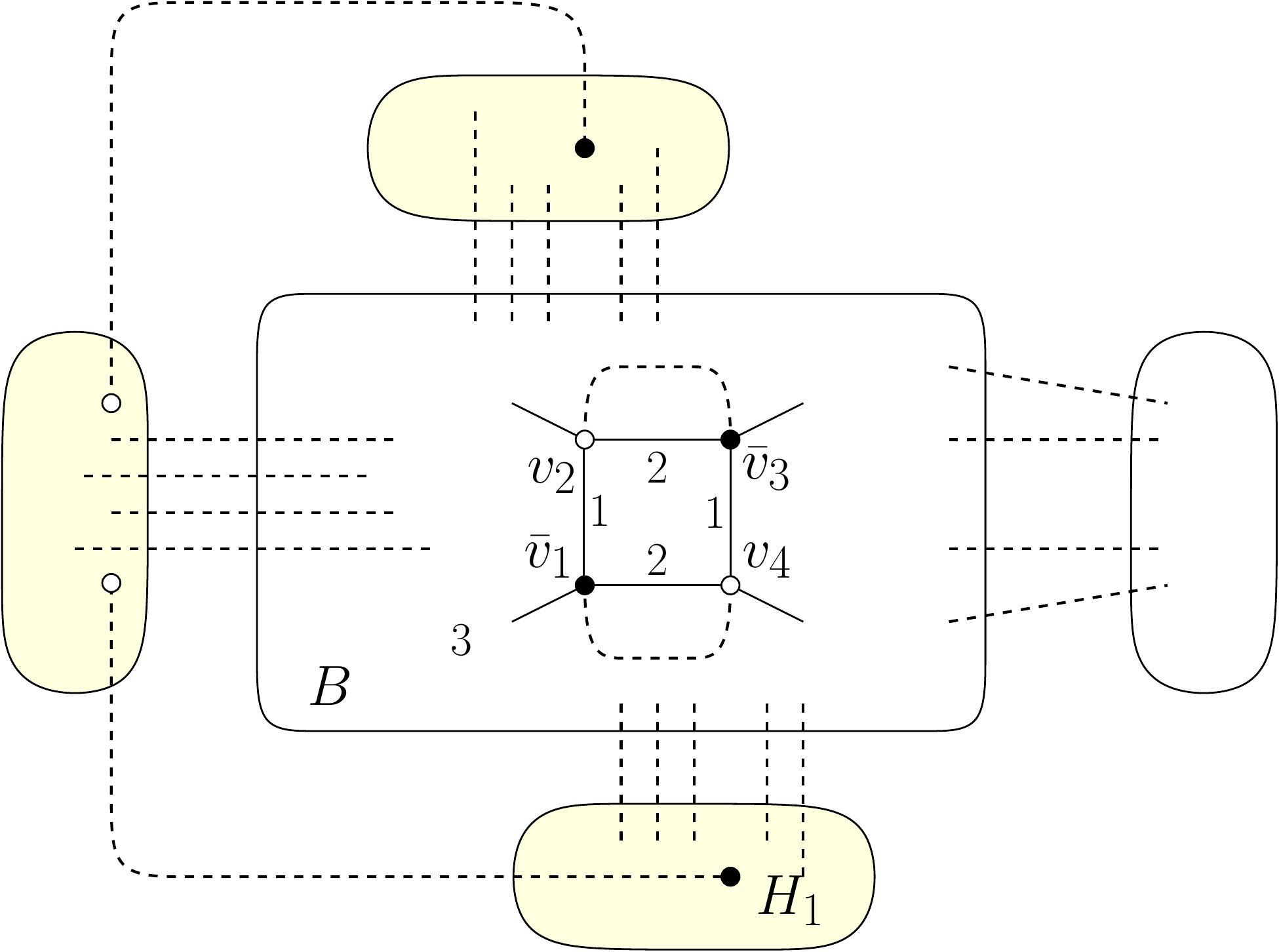} \end{array}
\end{equation}
to which the reference case applies once again.

If $e_1, e_2, e_3, e_4$ belong to two different edge-cuts incident on $B$, they are three possibilities but only two inequivalent ones. The edges which belong to a common edge-cut can be incident on vertices which are adjacent or opposite. The adjacent case is straightforward, similar to ``twice'' \eqref{FaceDegree4TwoEdgesOneCut}. The opposite case looks like
\begin{equation}
G = \begin{array}{c} \includegraphics[scale=.4]{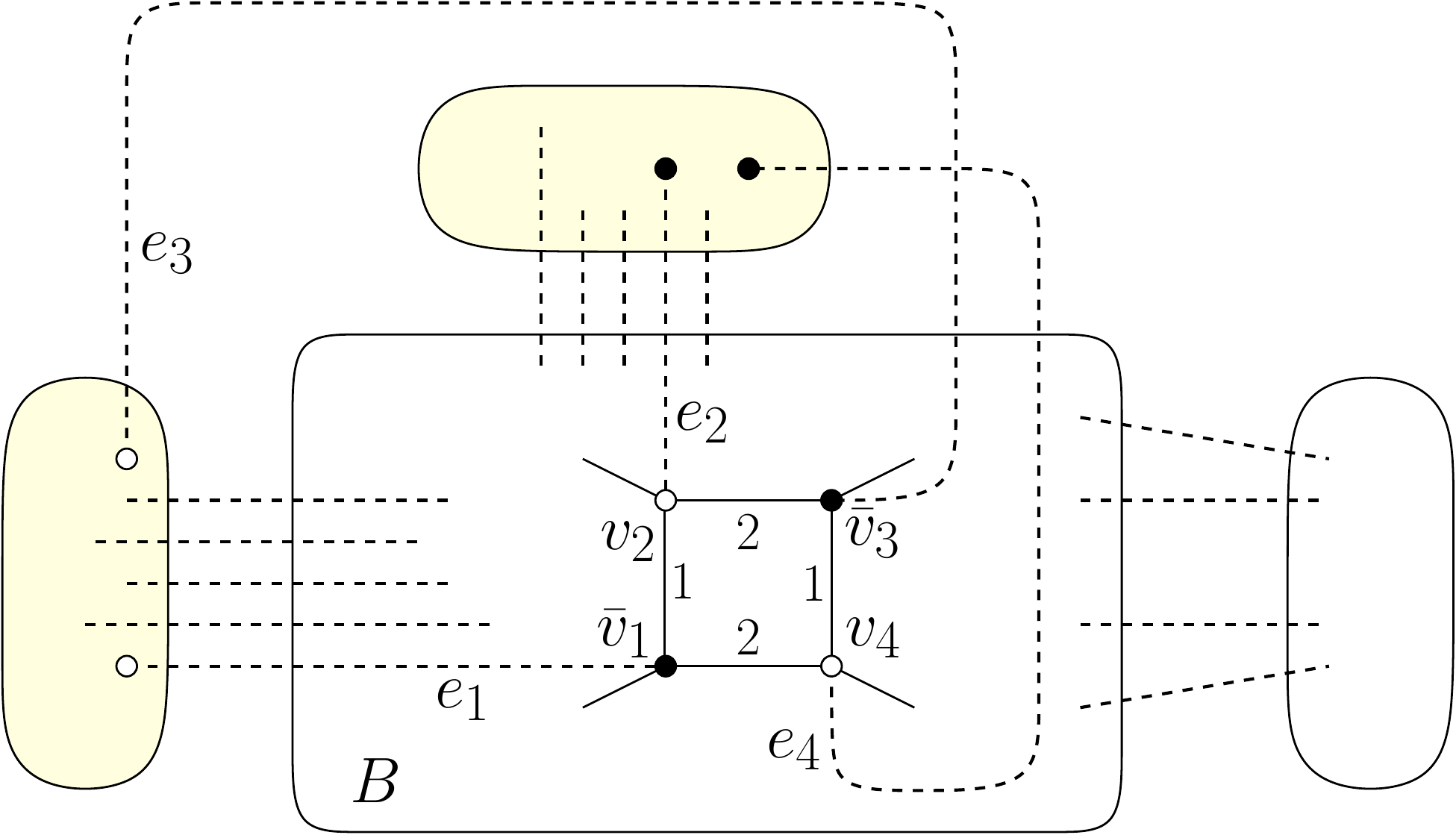} \end{array}
\end{equation}
and after the flips
\begin{equation}
G_= = \begin{array}{c} \includegraphics[scale=.4]{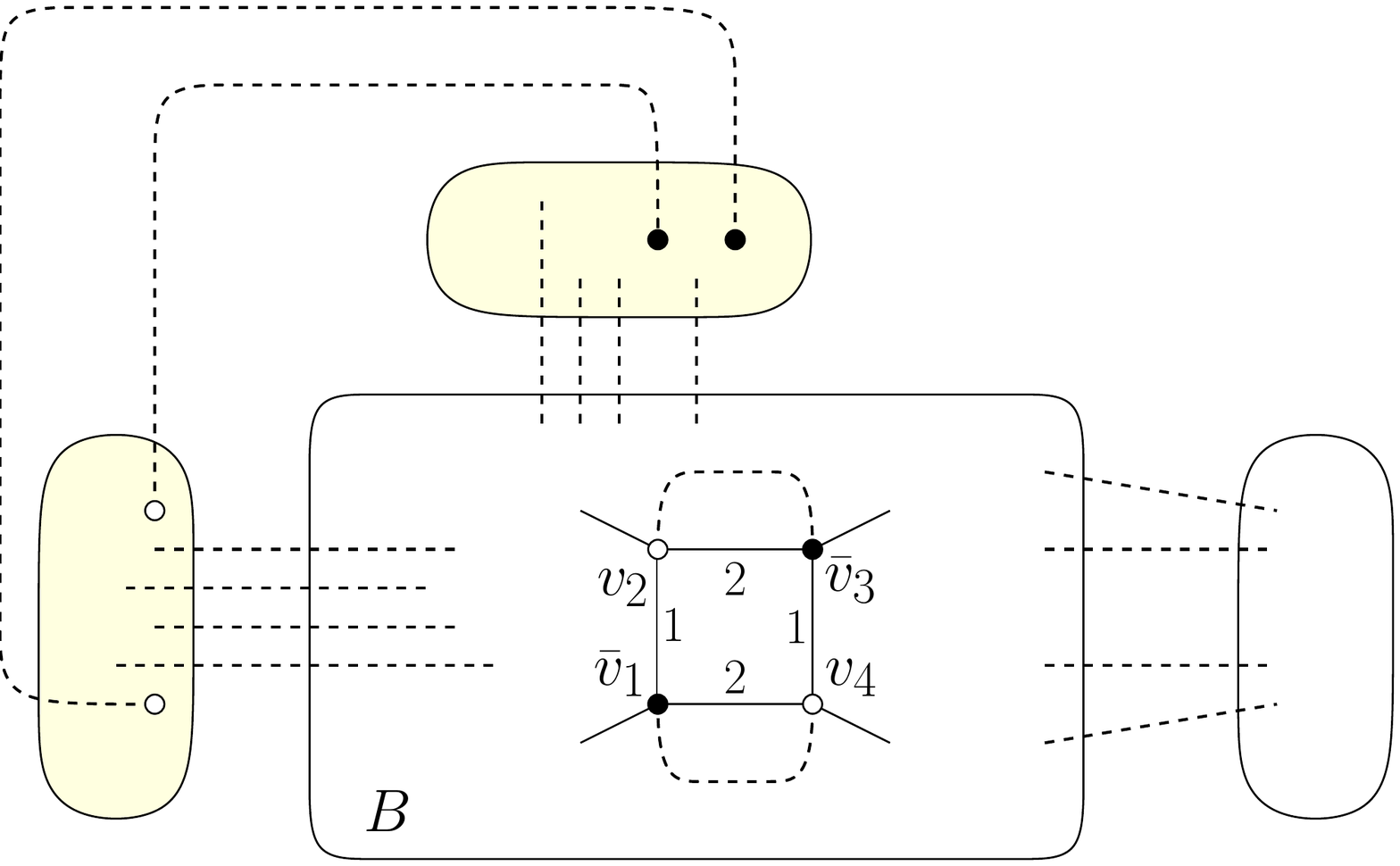} \end{array}
\end{equation}
which can also be treated with the reference case.

If $e_1, e_2, e_3, e_4$ are part of a single $(k\geq 8)$-edge-cut incident on $B$, 
\begin{equation}
\begin{aligned}
&G = \begin{array}{c} \includegraphics[scale=.4]{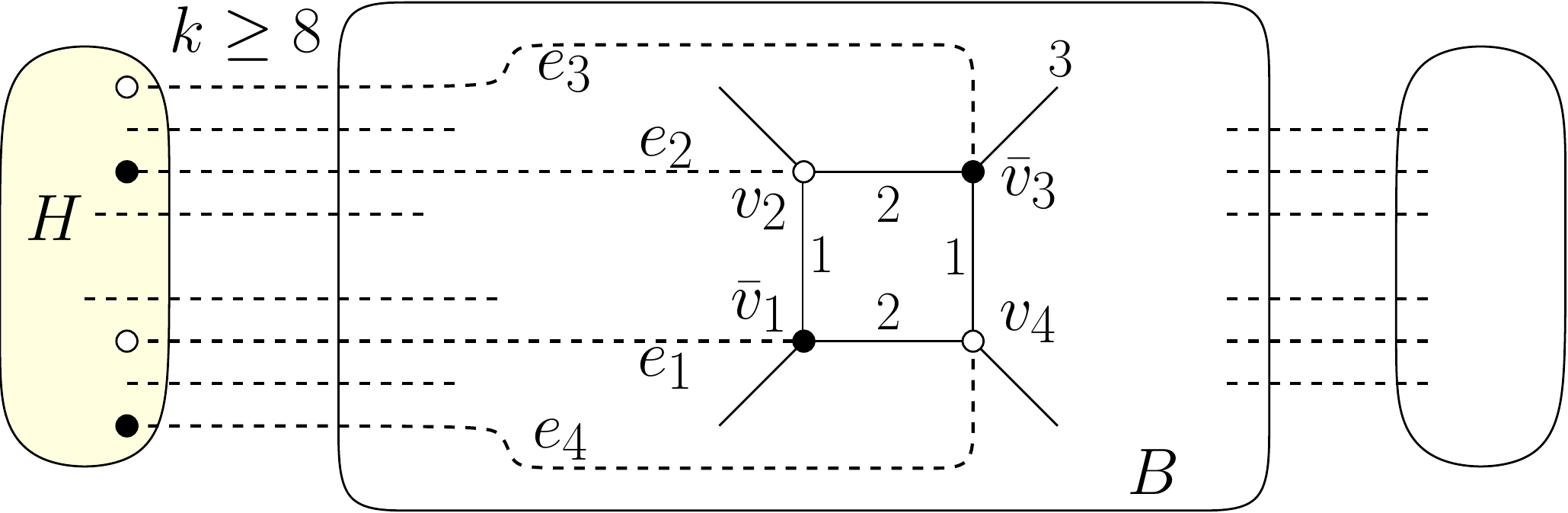} \end{array}\\
\Rightarrow\qquad &G_= = \begin{array}{c} \includegraphics[scale=.4]{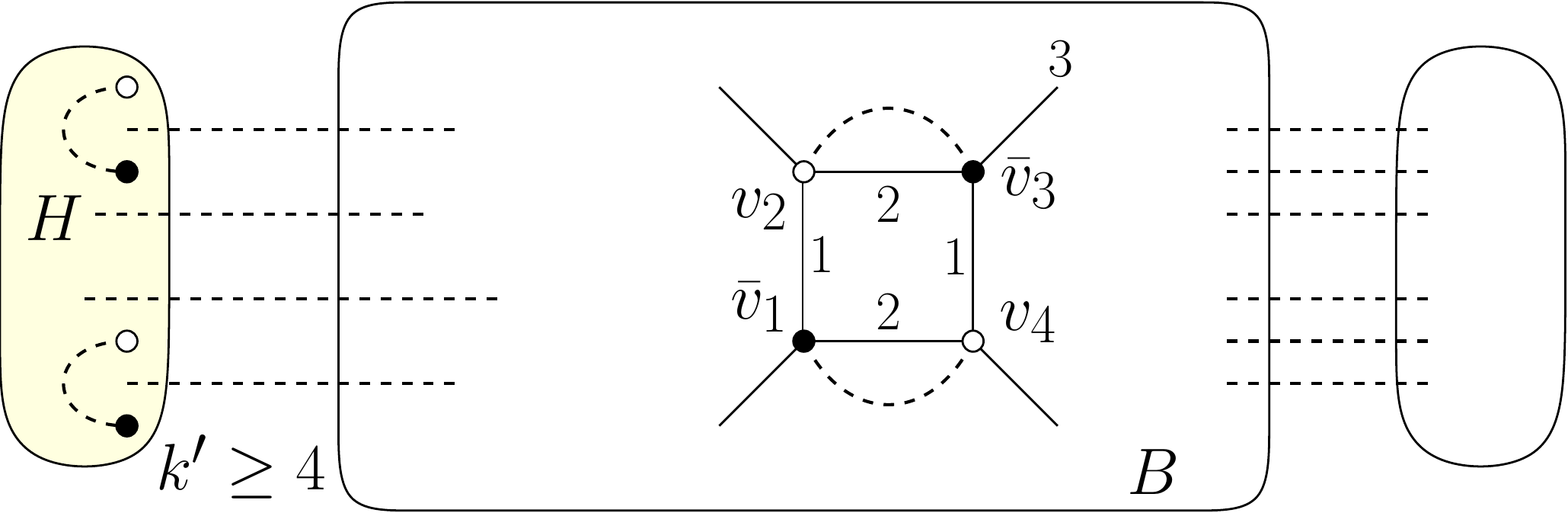} \end{array}
\end{aligned}
\end{equation}
it is easy to check that $G_{=}$ still has a $(k'\geq 4)$-edge-cut incident on $B$, so that the reference case or Proposition \ref{prop:4EdgeCuts} applies.

\item The last case to study, which is not treatable directly by either Proposition \ref{prop:4EdgeCuts} or the reference case, is when the four edges $e_1, e_2, e_3, e_4$ are part of a 6-edge-cut incident to $B$, because it becomes a 2-edge-cut in $G_{=}$ and $G_{\parallel}$. When removing the six edges of the cut, we have two connected components, both with six free vertices, and we denote $K$ the one which contains $B$. The boundary bubble $\partial K$ of $K$ is a bubble or a disjoint union of bubbles with six vertices total. All such possible objects are melonic bubbles or union of melonic bubbles, except one which as a graph is the complete bipartite graph $K_{3,3}$. Its canonical embedding, Corollary \ref{cor:2D}, as a combinatorial map with colored edges, which we still denote $K_{3,3}$, is homeomorphic to a torus. It has one face with colors $\{a, b\}$ for $1\leq a<b\leq 3$, so three faces total, and all of them of degree 6. The key point is that all four vertices $v_1, \bar{v}_2, v_3, \bar{v}_4$ of $G$ belong to $\partial K$, hence so do the edges of color 1 and 2 between them. Therefore, the face of colors $\{1,2\}$ and degree 4 is a submap of $\partial K$. This is impossible for $K_{3, 3}$ and this means that $\partial K$ cannot be $K_{3,3}$,
\begin{equation}
\partial K = \partial\quad \begin{array}{c} \includegraphics[scale=.4]{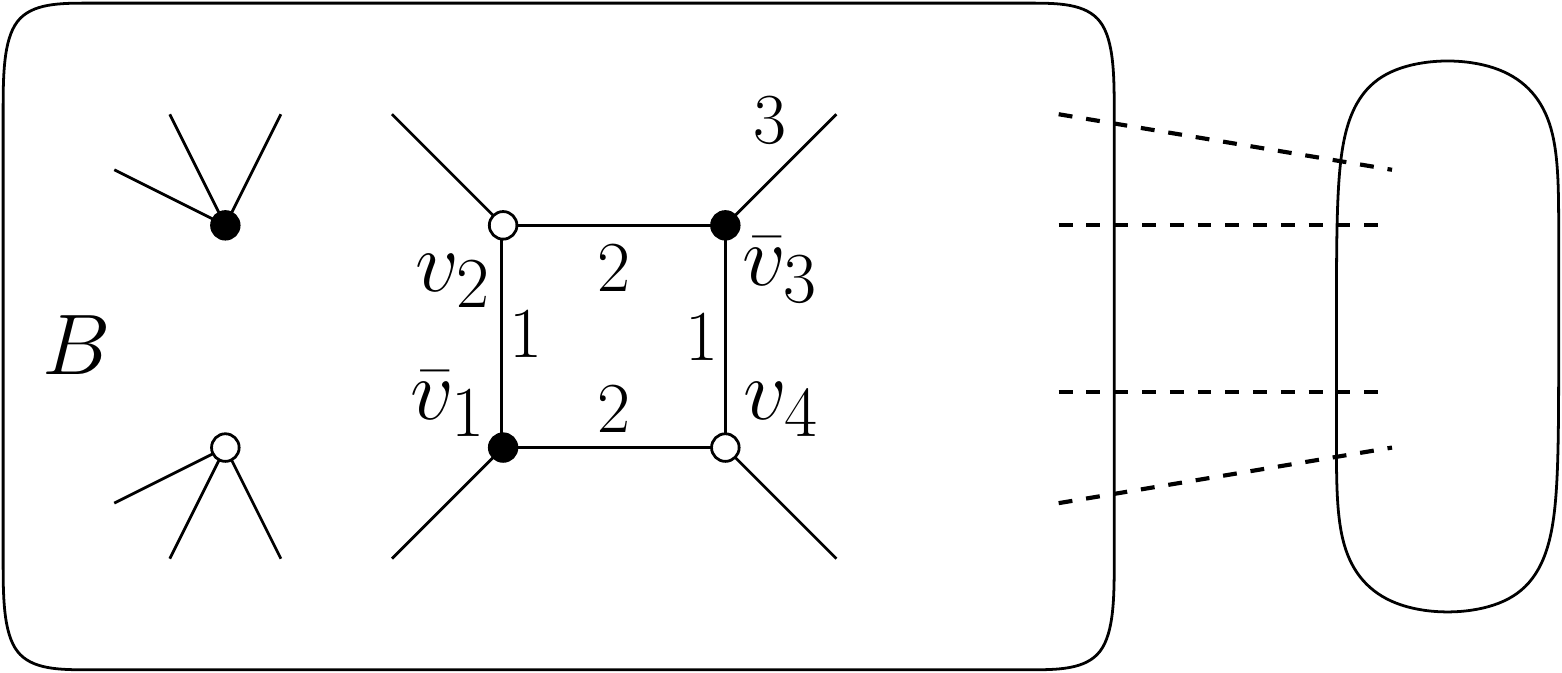} \end{array} \quad \neq \quad \begin{array}{c} \includegraphics[scale=.4]{6VertexBubble3dK33.pdf} \end{array} = K_{3, 3}.
\end{equation}
$\partial K$ is thus melonic or a union of melonic bubbles with six vertices total. It is a simple exercize to show that in order to maximize the number of bicolored cycles, $\partial K$ can only have incident 2-edge-cuts.
\end{itemize}
\end{description}
\end{proof} 

%%%%%%%%%%%%%%%
\subsubsection{Edge-cuts incident on an arbitrary bubble when all other bubbles are planar}
%%%%%%%%%%%%%%%

\begin{theorem} \label{thm:OnlyPlanar}
Let $\cG_{n_1, \dotsc, n_N}(B_1, \dotsc, B_N; B)$ (resp. $\cG^{\max}_{n_1, \dotsc, n_N}(B_1, \dotsc, B_N; B)$) be the set of connected colored graphs with $n_i$ bubbles $B_i$, $i=1, \dotsc, N$, and one marked bubble $B$ (resp. which maximize the total number of bicolored cycles). If the bubbles $B_1, \dotsc, B_N$ are planar, then $G\in \cG^{\max}_{n_1, \dotsc, n_N}(B_1, \dotsc, B_N; B)$ if and only if all its bubbles satisfy the maximal 2-cut property. The maximal number of bicolored cycles is
\begin{equation}
C_{n_1 \dotsb n_N}(B_1, \dotsc, B_N; B) = C_1(B) + \sum_{i=1}^N n_i \Bigl(C_1(B_i) - 3\Bigr)
\end{equation}
where we recall that $C_1(B_i)$ is the maximal number of bicolored cycles for 1-bubble graphs in $\cG^{\max}_1(B_i)$.
\end{theorem}

This Theorem has to be compared with Gurau's central limit theorem for large random tensors in \cite{Universality}. Although the latter has a different range of applicability, it does intersect our area of interest, due to the fact that some tensor integrals with unitary symmetry are generating functions of colored triangulations with prescribed CBBs \cite{Uncoloring}. In this context, Gurau's universality theorem \cite{Universality} is the same as Theorem \ref{thm:OnlyPlanar} except that the bubbles $B_1, \dotsc, B_N$ are melonic instead of planar ($B$ is still arbitrary -- this is the purpose of both theorems). Clearly melonicity is a much stronger constraint than planarity, which makes Theorem \ref{thm:OnlyPlanar} more powerful, as far as $d=3$ is concerned (\cite{Universality} holds for all $d\geq 3$).

\begin{proof}
We proceed by induction on the number of bubbles. If $B$ is the only bubble, there is nothing to prove as $G\in\cG^{\max}_1(B)$ must by definition satisfy the maximal 2-cut property. With only $n_1=1$, the set is $\cG_1(B_1; B)$ and its graphs only have $B$ and $B_1$ which must be connected. From Theorem \ref{thm:1Planar}, $B_1\subset G$ which is planar must satisfy the maximal 2-cut property for $G$ to be in $\cG^{\max}_1(B_1; B)$. The two bubbles are thus connected by a 2-edge-cut with edges of color 0. Flipping these two edges, we disconnect $G$ into two graphs $G_1\in \cG^{\max}_1(B_1)$ (from Theorem \ref{thm:1Planar}) and $\tilde{G}\in \cG_1(B)$ and $C_0(G) = C_0(G_1) + C_0(\tilde{G}) - 3$. It implies that the pairings defined by $\tilde{G}$ must maximize the number of bicolored cycles, i.e. $\tilde{G}\in\cG^{\max}_1(B)$. Therefore both $B_1$ and $B$ must satisfy the maximal 2-cut property and
\begin{equation}
C_{n_1}(B_1; B) = C_1(B) + C_1(B_1)-3
\end{equation}
for $n_1=1$.

Assume that the theorem is proved for $\sum_{i=1}^N n_i$ bubbles, including $B$, and consider $G\in \cG_{n_1, \dotsc, n_N}(B_1, \dotsc, B_N; B)$. Consider a copy of $B_1$ in $G$. Since it is planar, Theorem \ref{thm:1Planar} applies and we know that $B_1$ is connected to the rest of $G$ by a set of $L$ 2-edge-cuts which are parts of a pairing in $\cG^{\max}_1(B_1)$. We then flip the edges of each pair. This creates $L+1$ connected components $G_0, G_1, \dotsc, G_L$, $G_0\in \cG_1(B_1)$ is the 1-bubble graph which contains $B_1$. The number of bicolored cycles is
\begin{equation}
C_0(G) = C_0(G_0) + \sum_{l=1}^L \Bigl(C_0(G_l) - 3\Bigr).
\end{equation}
All connected components are made of planar bubbles except for one of them which contains $B$, and they all have fewer bubbles than $G$. From the induction hypothesis, all bubbles in those graphs satisfy the maximal 2-cut property.
\end{proof}

%%%%%%%%%%%%%
\subsection{Topology} \label{sec:Topology}
%%%%%%%%%%%%%

In order to study the topology of the pseudo-manifolds represented by elements of $\cG^{\max}_1(B)$, we need to extend some of the notions we already introduced as we have only needed them in particular circumstances so far. Those notions (bubbles and topological moves) are familiar to the colored graph community. 

\begin{definition} [$p$-bubbles]
If $G\in\cG_{n_1, \dotsc, n_N}(B_1, \dotsc, B_N)$ is a colored graph with colors $\{0, 1, \dotsc, d\}$, and $P=\{c_1, \dotsc, c_p\} \subset \{0, \dotsc, d\}$ is a subset of colors, we denote $G(P)$ the subgraph with colors $\{c_1, \dotsc, c_p\}$ only, i.e. obtained by removing all complementary colors. A $P$-bubble, or $p$-bubble with colors $\{c_1, \dotsc, c_p\}$ is a connected component of $G(P)$.
\end{definition}

The bubbles we have considered so far are thus the $\{1, \dotsc, d\}$-bubbles, or $d$-bubbles with all colors except 0. The $0$-bubbles, $1$-bubbles of color $\{c\}$ and $2$-bubbles of colors $\{a, b\}$ are respectively vertices, edges of color $\{c\}$ and bicolored cycles with colors $\{a, b\}$ of $G$. A $p$-bubble with colors $\{c_1, \dotsc, c_p\}$ is thus dual to a $(d-k)$-simplex with the same colors, see Theorem \ref{thm:ColoredGraphs}.

The next thing we need is to make sure that we can perform moves in $\cG^{\max}_1(B)$ which do not change the topology. There is a set of topological moves, analogous to Pachner moves for colored triangulations, known as the dipole moves. Here we will not need the theorem that two colored graphs represent the same PL-manifold if and only if there are related by a finite sequence of dipole moves. Rather, we will just need to know that those moves are topological.

\begin{theorem} [Dipole moves] \label{thm:Dipoles}
Let $G$ be a colored graph with colors $\{0, \dotsc, d\}$. Let $H = \{a_1, \dotsc, a_h\}$ be a subset of colors, and denote $P = \{0, \dotsc, d\} \setminus H= \{c_1, \dotsc, c_p\}$ the complement. An $h$-dipole is a set of $h$ edges with colors in $H$ connecting the same two vertices $v, \bar{v}$ and such that $v$ and $\bar{v}$ belong to two different $P$-bubbles $B(P)_L$ and $B(P)_R$. The $h$-dipole removal and its inverse are 
\begin{equation}
\begin{array}{c} \includegraphics[scale=.4]{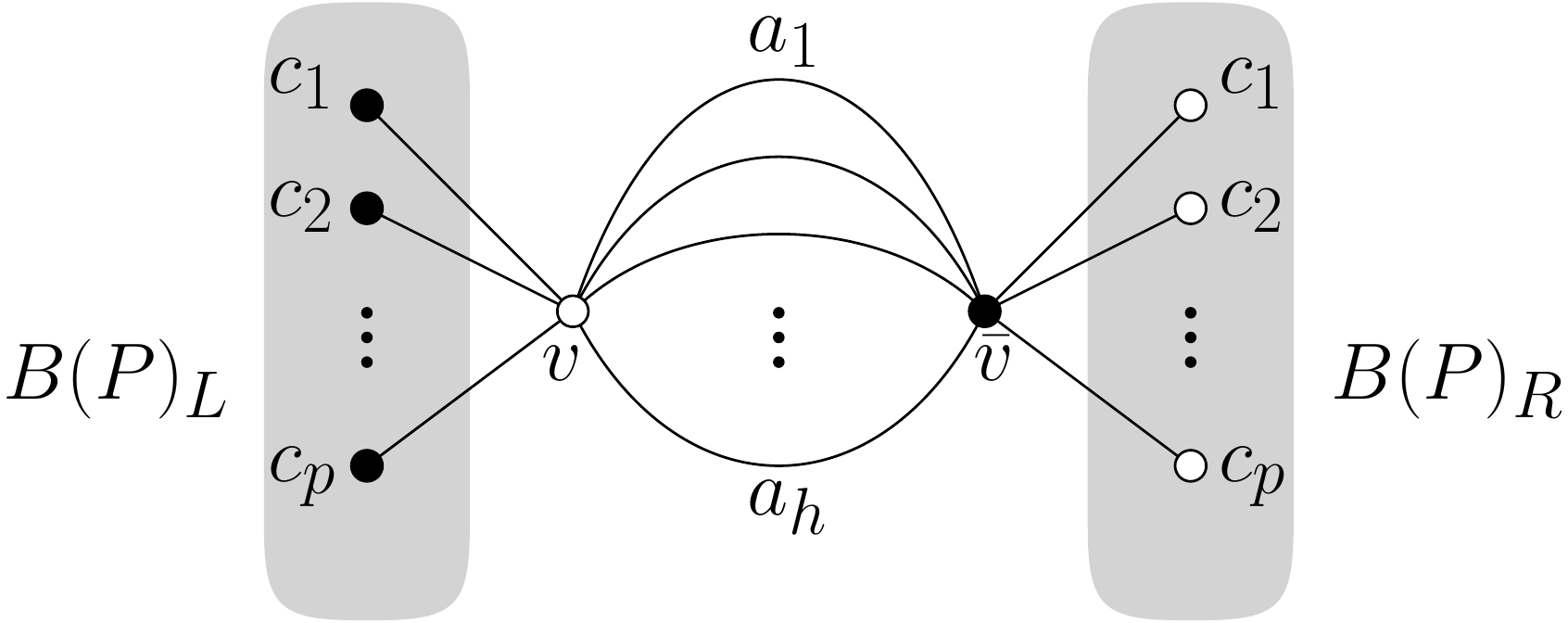} \end{array} \quad \leftrightarrow \quad \begin{array}{c} \includegraphics[scale=.4]{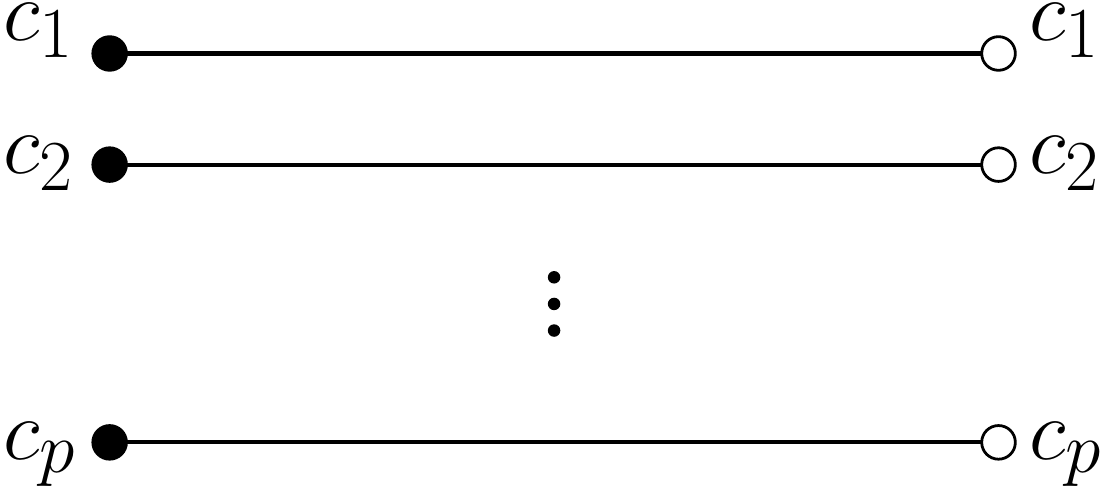} \end{array}
\end{equation}
They are topological moves if at least one of the $P$-bubbles $B(P)_L$ or $B(P)_R$ is a ball.
\end{theorem}

We do not prove this theorem and refer to \cite{DipoleMoves}. Notice that the grey areas above represent the $P$-bubbles and there are typically other paths between them than through $v$ and $\bar{v}$, which involve the colors $a_1, \dotsc, a_h$.

The contraction $G/e$ of $G$ by an edge of color 0, see Definition \ref{def:Contraction}, can lead to a dipole removal if $e$ is part of a dipole. But contractions are more generic and may not preserve the topology. In fact, the situations given in Proposition \ref{prop:Contraction} where topology is preserved are just special cases of the above more generic result to $d=2$ (for the bubble, since the 2-bubbles are bicolored cycles and always represent a disc).

Dipole moves can be used to show that edge flips can sometimes be topological. We restrict to the three-dimensional case which we will be useful for us.

\begin{corollary} \label{cor:TopologicalFlip}
Let $G$ be a connected colored graph with colors $\{0,1 ,2,3\}$ and consider $v$ and $\bar{v}$ connected by an edge of color $c\in\{1,2,3\}$. If $v$ and $\bar{v}$ belong to two different 3-bubbles with colors $\{0,c_1, c_2\}$ where at least one of them is a ball, and $\{c_1, c_2\} = \{1, 2, 3\}\setminus\{c\}$, then flipping the two edges of color 0 incident to $v$ and $\bar{v}$ is topological
\begin{equation}
\begin{array}{c} \includegraphics[scale=.4]{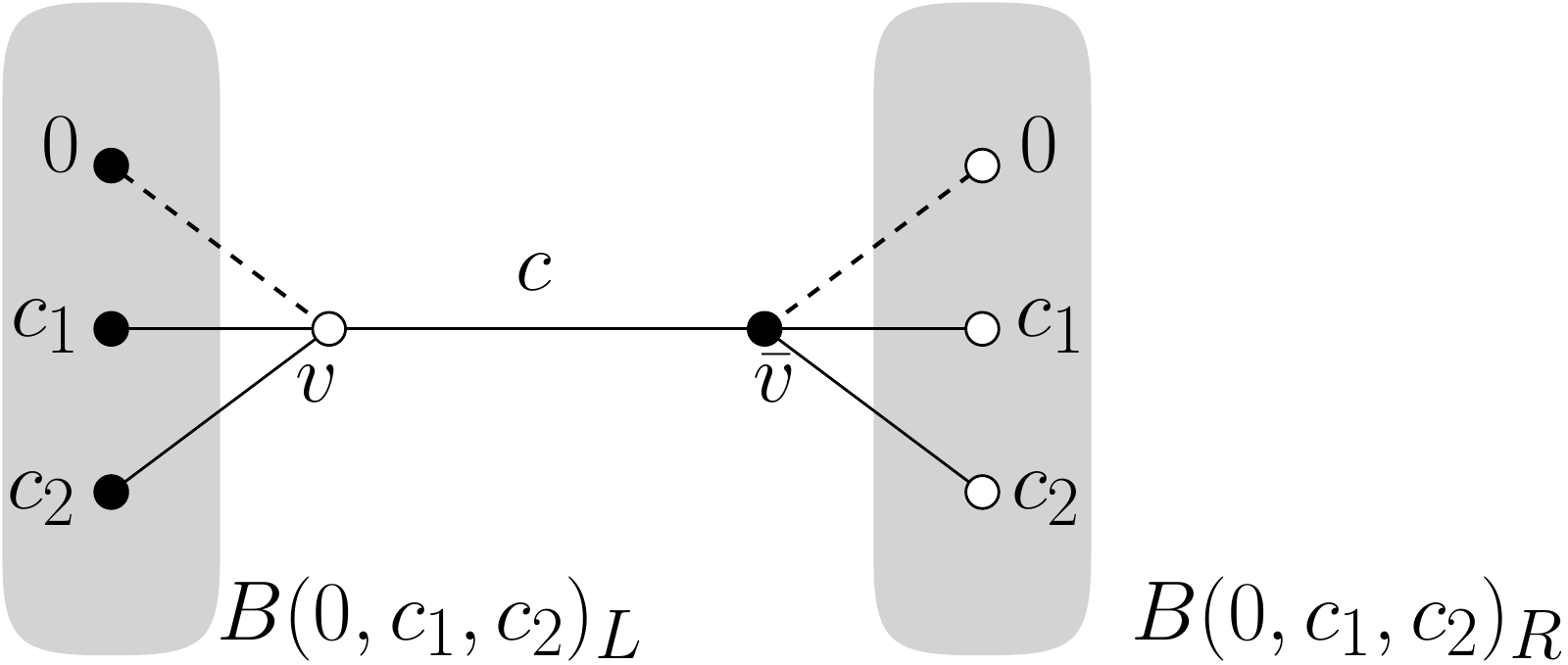} \end{array} \quad \leftrightarrow \quad \begin{array}{c} \includegraphics[scale=.4]{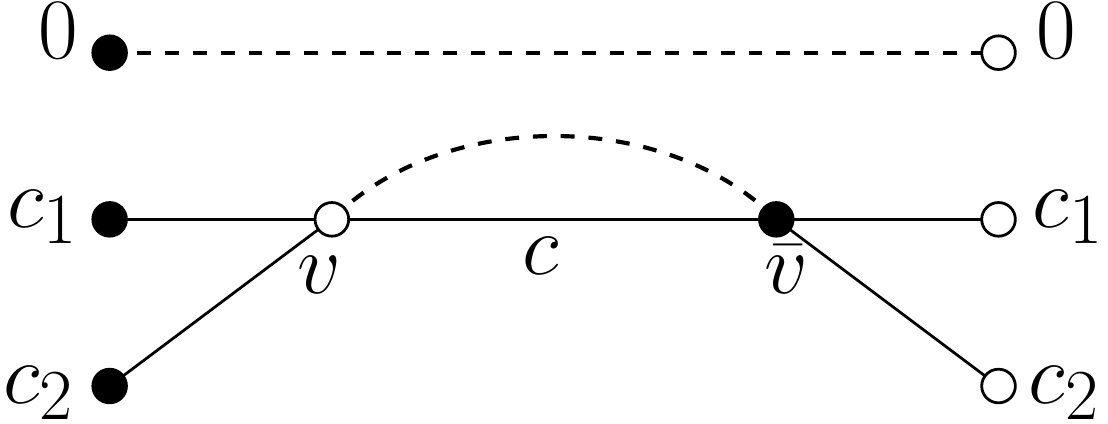} \end{array}
\end{equation}
\end{corollary}

\begin{proof}
We start by performing a 1-dipole removal for the edge of color $c$ since we have all the assumptions to do so,
\begin{equation}
\begin{array}{c} \includegraphics[scale=.4]{1DipoleFlip.pdf} \end{array} \quad \leftrightarrow\quad \begin{array}{c} \includegraphics[scale=.4]{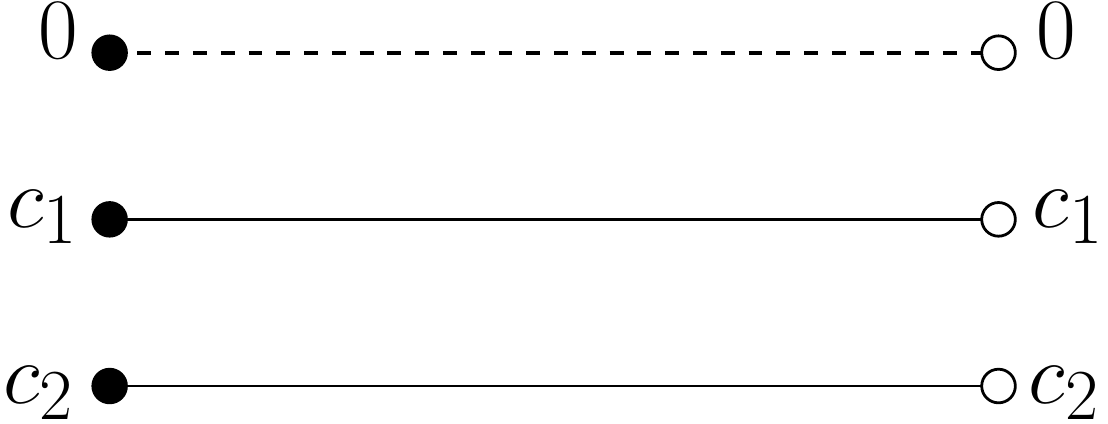} \end{array}
\end{equation}
Then we perform a 2-dipole insertion with colors $\{0,c\}$ on the edges of color $c_1$ and $c_2$. The new vertices can be named $v$ ad $\bar{v}$ again. The faces with colors $\{c_1, c_2\}$ incident to $v$ and $\bar{v}$ in the original graph and in the final graph are the same. Since they are distinct, the 2-dipole insertion preserves the topology. The overall result is a flip of edges of color 0.
\end{proof}

\begin{proposition} [Connected sums] \label{prop:ConnectedSums}
We recall that the connected sum of two $d$-manifolds $M$ and $N$ is a $d$-manifold obtained by removing a ball in $M$ and in $N$ and then identifying the two resulting $(d-1)$-spheres. For two colored graphs $G_1$, $G_2$, a connected sum $G_1\#_{\{v_1, v_2\}} G_2$ is obtained by removing a black vertex $v_1$ from $G_1$, a white vertex $v_2$ from $G_2$ and identifying the resulting hanging edges which have the same color,
\begin{equation}
\begin{array}{c} \includegraphics[scale=.4]{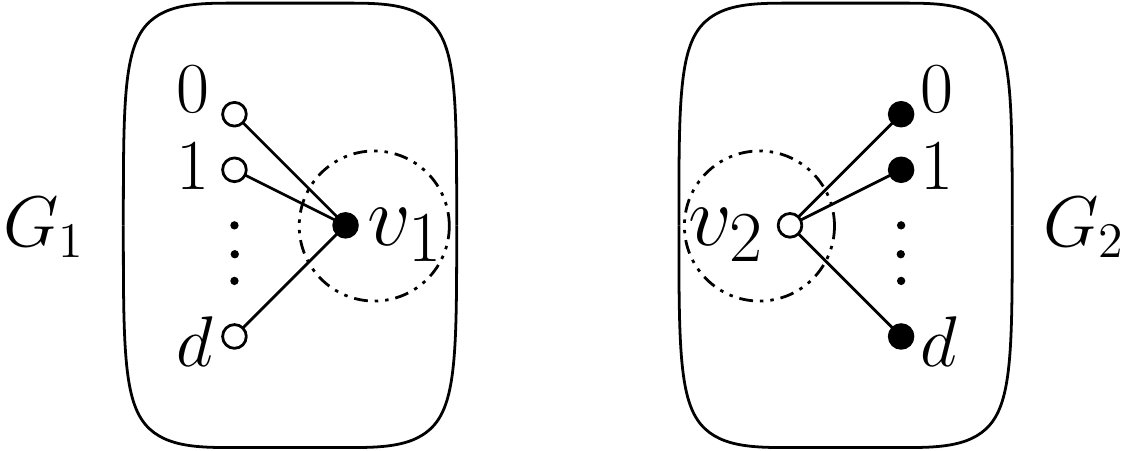} \end{array} \quad \to \quad G_1\#_{\{v_1, v_2\}} G_2 = \begin{array}{c} \includegraphics[scale=.4]{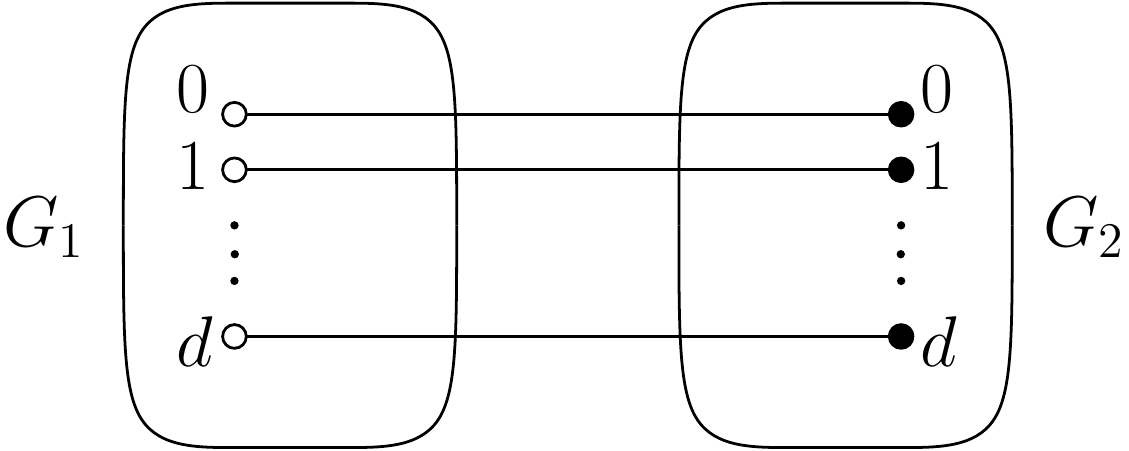} \end{array}
\end{equation}

If the connected sum is unique, then $G_1\#_{\{v_1, v_2\}} G_2$ is a colored graph for the connected sum of the manifolds represented by $G_1$ and $G_2$. 
\end{proposition}

We refer to \cite{ItalianSurvey} for multiple references to connected sums in the colored graph literature in topology. Notice that the sphere is a neutral element for the connected sum of manifolds (and obviously the connected sum with the sphere is unique).

\begin{lemma} \label{lemma:2CutConnected}
If $G$ is a connected colored graph with colors $\{0, \dotsc, d\}$ which has a 2-edge-cut formed by two edges $e, e'$ of color 0, see \eqref{2CutGraph}, then flipping $e$ with $e'$ gives two connected colored graphs $G_1, G_2$, see \eqref{2CutFlipped}, and $G$ is a connected sum of $G_1$ and $G_2$.
\end{lemma}

\begin{proof}
We simplify the proof which was used in \cite{Octahedra} in the context of octahedra. First add to $G_1$ a $d$-dipole with colors $\{1, \dotsc, d\}$ and choose $v_1$ and $v_2$ to perform the connected sum $\tilde{G}_1\#_{\{v_1, v_2\}} G_2$ as follows 
\begin{equation}
\tilde{G}_1 = \begin{array}{c} \includegraphics[scale=.45]{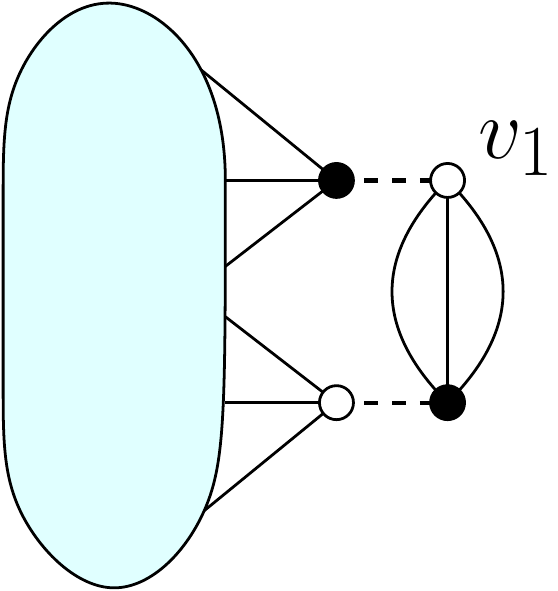} \end{array} \quad G_2 = \begin{array}{c} \includegraphics[scale=.45]{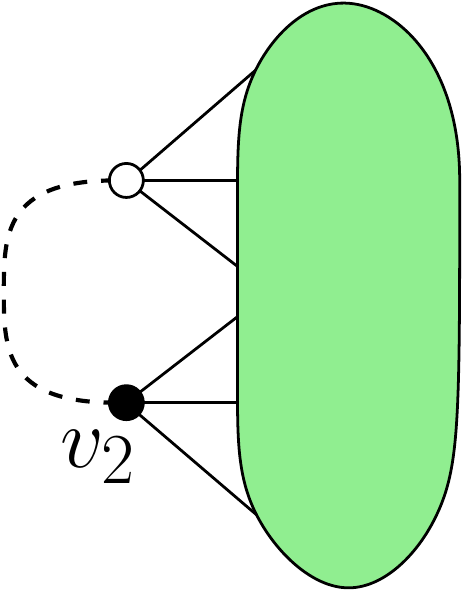} \end{array} \quad \to \quad \tilde{G}_1\#_{\{v_1, v_2\}} G_2 = \begin{array}{c} \includegraphics[scale=.45]{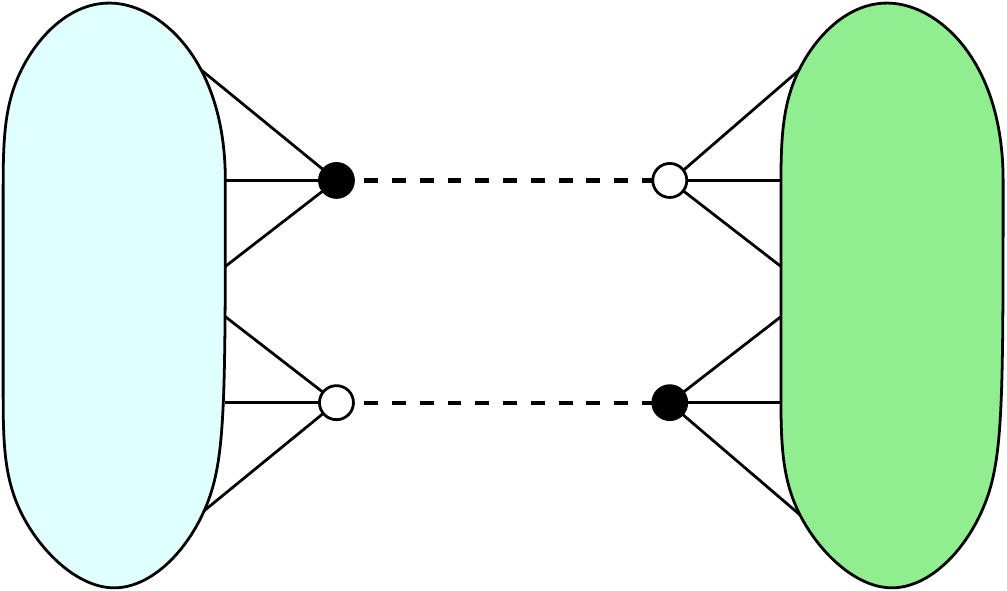} \end{array} = G
\end{equation}
Obviously $G_1$ and $\tilde{G}_1$ have the same topology. Therefore, if uniqueness holds, $G$ is indeed the connected sum of $G_1$ and $G_2$.
\end{proof}

\begin{theorem} \label{thm:Topology}
Let $\{B_1, \dotsc, B_N\}$ be a set of planar bubbles at $d=3$. Then $G\in \cG^{\max}_{n_1, \dotsc, n_N}(B_1, \dotsc, B_N)$ has the topology of the $3$-sphere.
\end{theorem}

Just like Theorem \ref{thm:1Planar}, it generalizes to all planar bubbles the same result previously obtained when all bubbles are melonic \cite{Melons} or octahedra \cite{Octahedra}.

It is more difficult than Theorem \ref{thm:1Planar} in a sense because contrary to the proof of Theorem \ref{thm:1Planar}, it requires a deeper investigation of the 1-bubble graphs $\cG^{\max}_1(B)$ (1-CBB triangulations in terms of tetrahedra) to prove that they are spheres. In particular, we actually use Theorem \ref{thm:1Planar} to prove that in one of several cases to analyze. %We also need to prove that all 3-bubbles are planar, for instance.

\begin{proof}
If $G$ has up to six vertices, its bubbles $B_1, \dotsc, B_N$ are melonic. Then $G$ maximizing the number of bicolored cycles is melonic too (with the additional color 0) and melonic graphs are well-known to be spheres (in arbitrary dimensions) \cite{Melons}.

We then proceed to an induction on the total number of vertices of $G$. We will have several cases to investigate. 
\begin{description}[style=unboxed,wide]
\item[$G$ has several bubbles] The easiest case is when $G$ is made of at least two bubbles, i.e. $\sum_{i=1}^N n_i\geq 2$. Then the gluing of those bubbles is combinatorially described by Theorem \ref{thm:1Planar}. We can pick up any bubble from $G$ and find a 2-edge-cut incident to it. This is exactly the situation of Lemma \ref{lemma:2CutConnected}. From Theorem \ref{thm:1Planar} we know that $G_1$ and $G_2$ (using the notations of Lemma \ref{lemma:2CutConnected}) maximize the number of bicolored cycles given their bubbles. The induction hypothesis then tells us that they are spheres. Lemma \ref{lemma:2CutConnected} then states that $G$ is a connected sum of spheres, hence a sphere itself.

\item[$G$ has a single bubble with a 2-edge-cut] We can now focus on $N=1$, i.e. $G\in\cG^{\max}_1(B)$ is made of a single bubble $B$. Let us focus moreover on the case where $B$ has two faces, say with colors $\{1, 2\}$ and $\{1, 3\}$, which share more than one edge (of color 1 obviously). It means that $B$ contains a piece of the following form,
\begin{equation}
B = \begin{array}{c} \includegraphics[scale=.35]{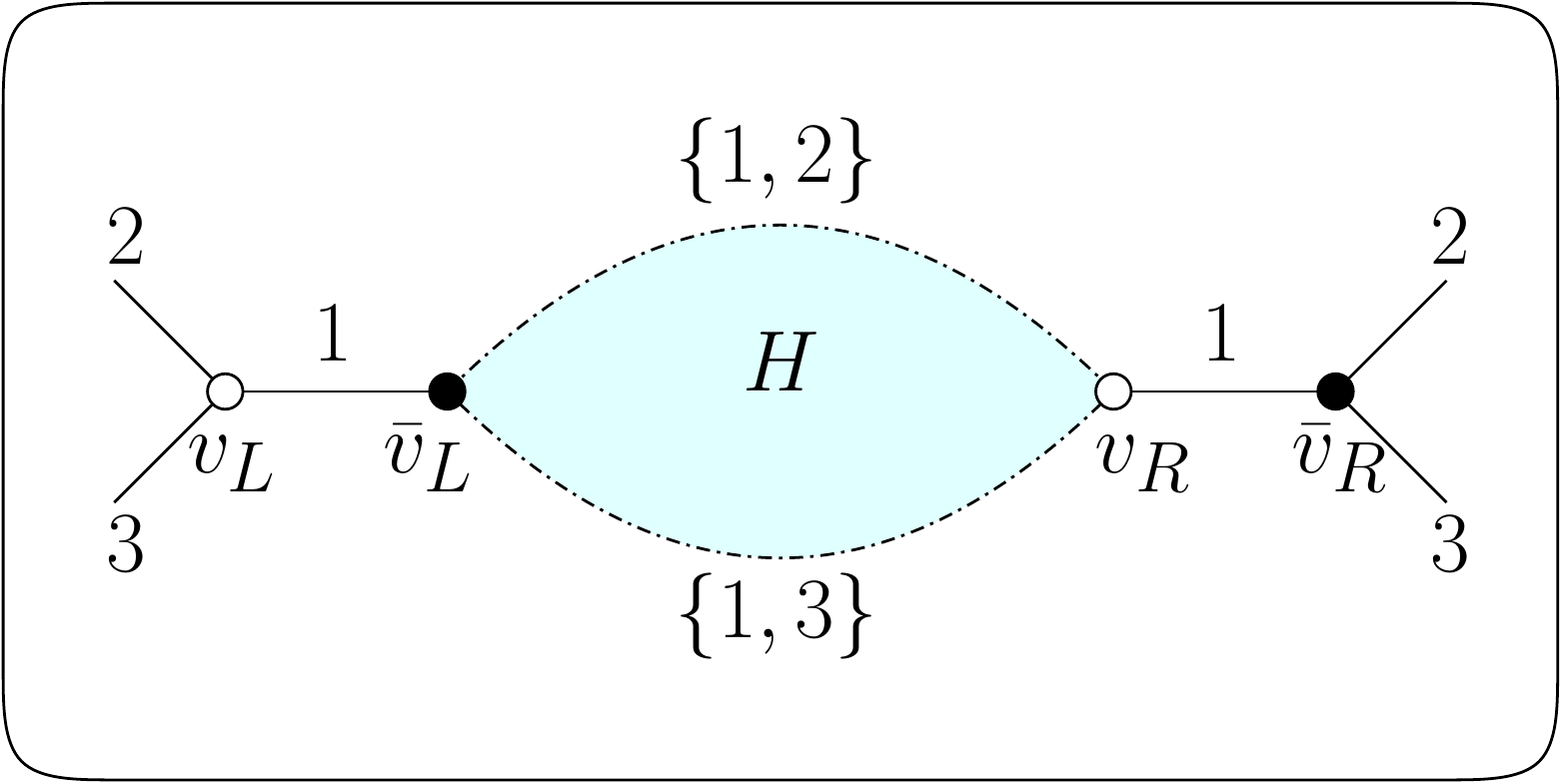} \end{array}
\end{equation}
where the dash-dotted lines represent paths of alternating colors $\{1, 2\}$ and $\{1, 3\}$. It can be obtained as the boundary bubble $B = \partial A$ where $A$ is made of two bubbles $B', B''$ defined as follows and connected by a single edge $e$ of color 0,
\begin{equation}
B' = \begin{array}{c} \includegraphics[scale=.28]{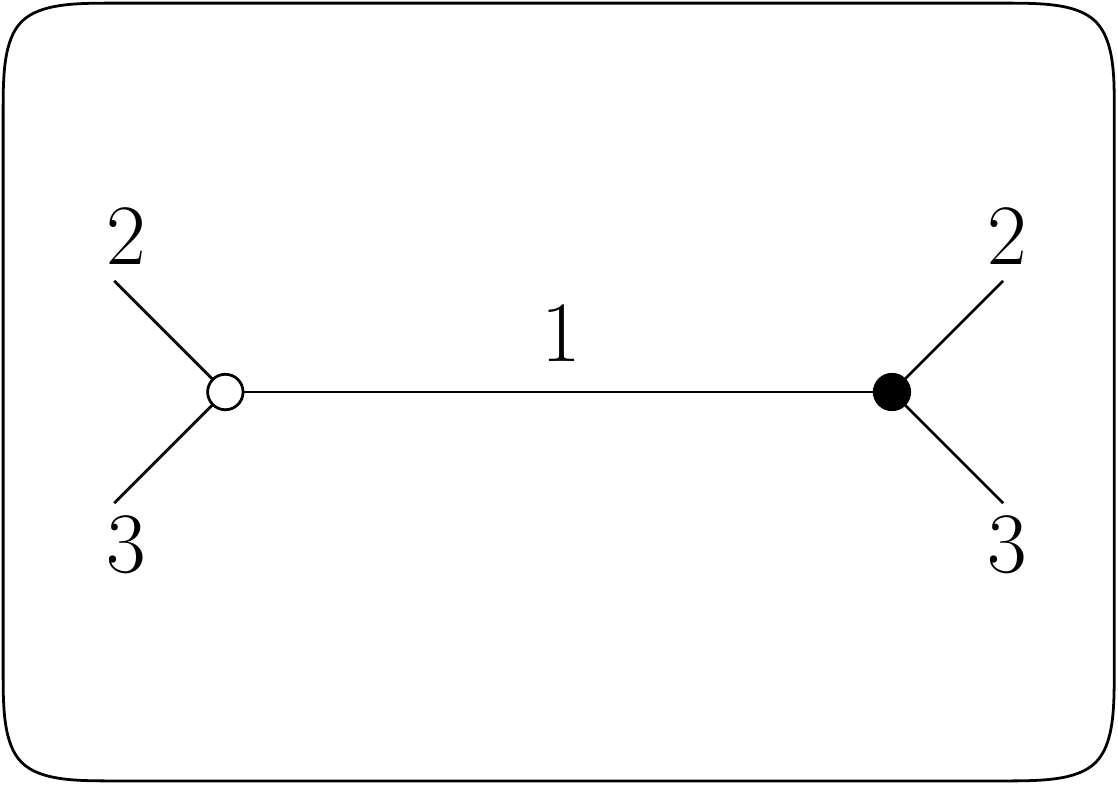} \end{array} \qquad \qquad B'' = \begin{array}{c} \includegraphics[scale=.28]{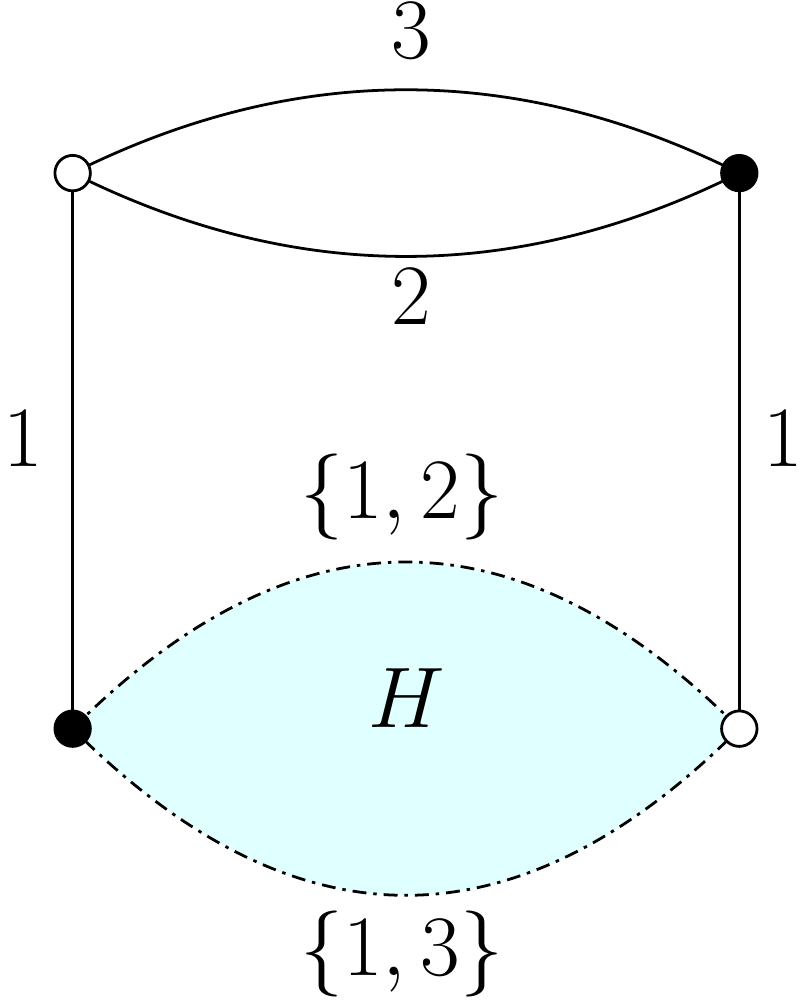} \end{array} \qquad \qquad A = \begin{array}{c} \includegraphics[scale=.28]{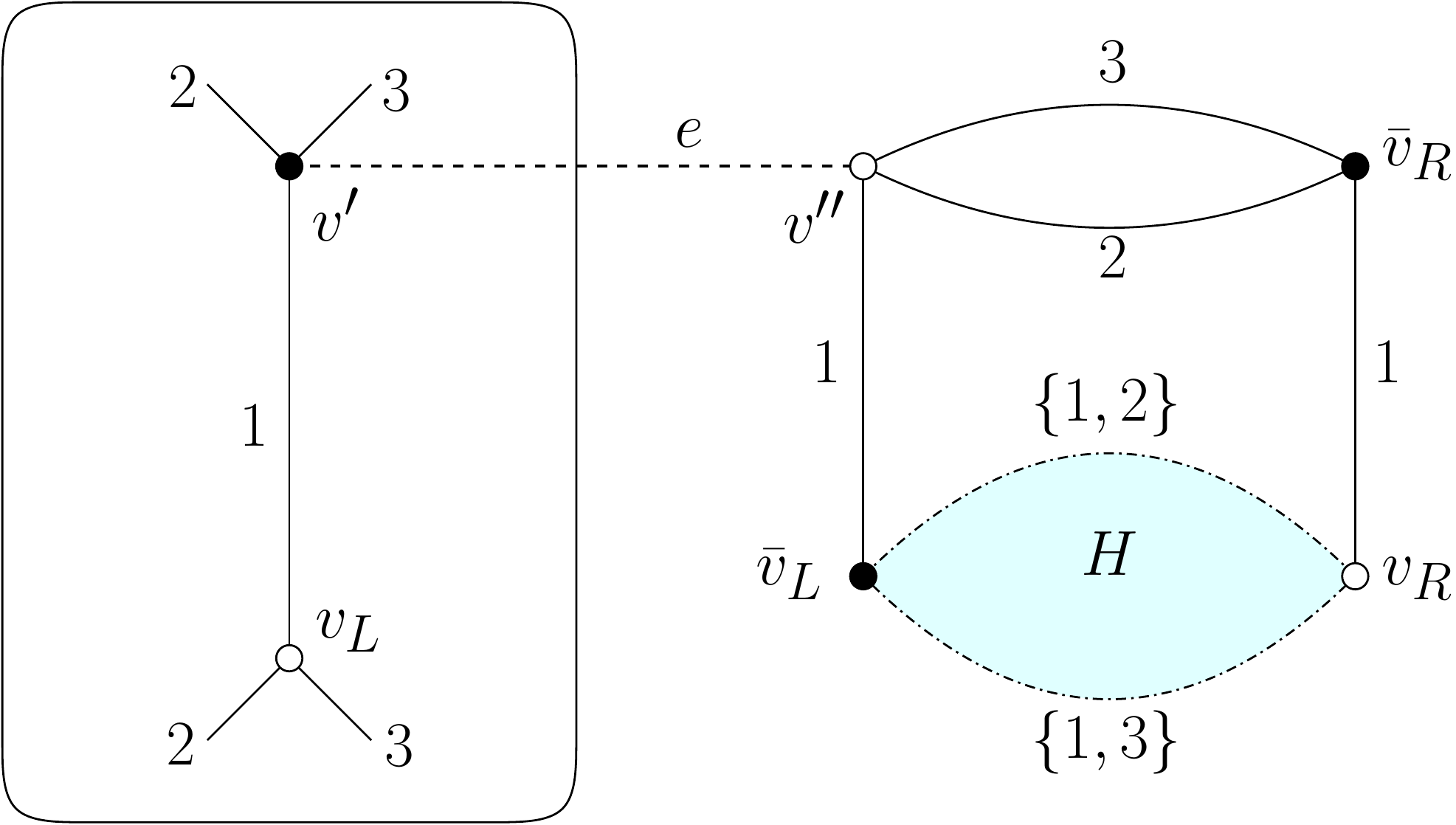} \end{array}
\end{equation}
The bubble $B$ is the boundary bubble, or equivalently, since there is a single edge of color 0, the contraction of $A$,
\begin{equation}
B = \partial A = A/e.
\end{equation}
Let us denote $\cG(A)$ the subset of $\cG_{1, 1}(B', B'')$ where $B'$ and $B''$ are precisely connected by an edge of color 0 between $v'$ and $v''$. There is a one-to-one correspondence between the vertices of $B$ and those of $A$ minus $v'$ and $v''$. By definition of the boundary bubble, and because $A$ does not have bicolored cycles with colors $\{0, c\}$, there is also a one-to-one correspondence between the bicolored cycles of a pairing $G\in \cG_1(B)$ and those of the same pairing equipped to $A$. Therefore the set $\cG_1(B)$ is equivalent to the set $\cG(A)$ of pairings on $A$ and this preserves the number of bicolored cycles with colors $\{0, c\}$ for $c\in\{1, 2, 3\}$.

The object $A$ is made of the two planar bubbles $B', B''$ connected by an edge. Maximizing the number of bicolored cycles thus follows Theorem \ref{thm:1Planar} for planar bubbles. It states that $\tilde{G}\in \cG^{\max}(A)$ has exactly two edges of color 0 connecting $B'$ and $B''$, one of them being set to be $e$, and that flipping these two edges produces two graphs $G'\in\cG^{\max}_1(B')$ and $G''\in\cG^{\max}_1(B'')$,
\begin{equation}
\tilde{G} = \begin{array}{c} \includegraphics[scale=.28]{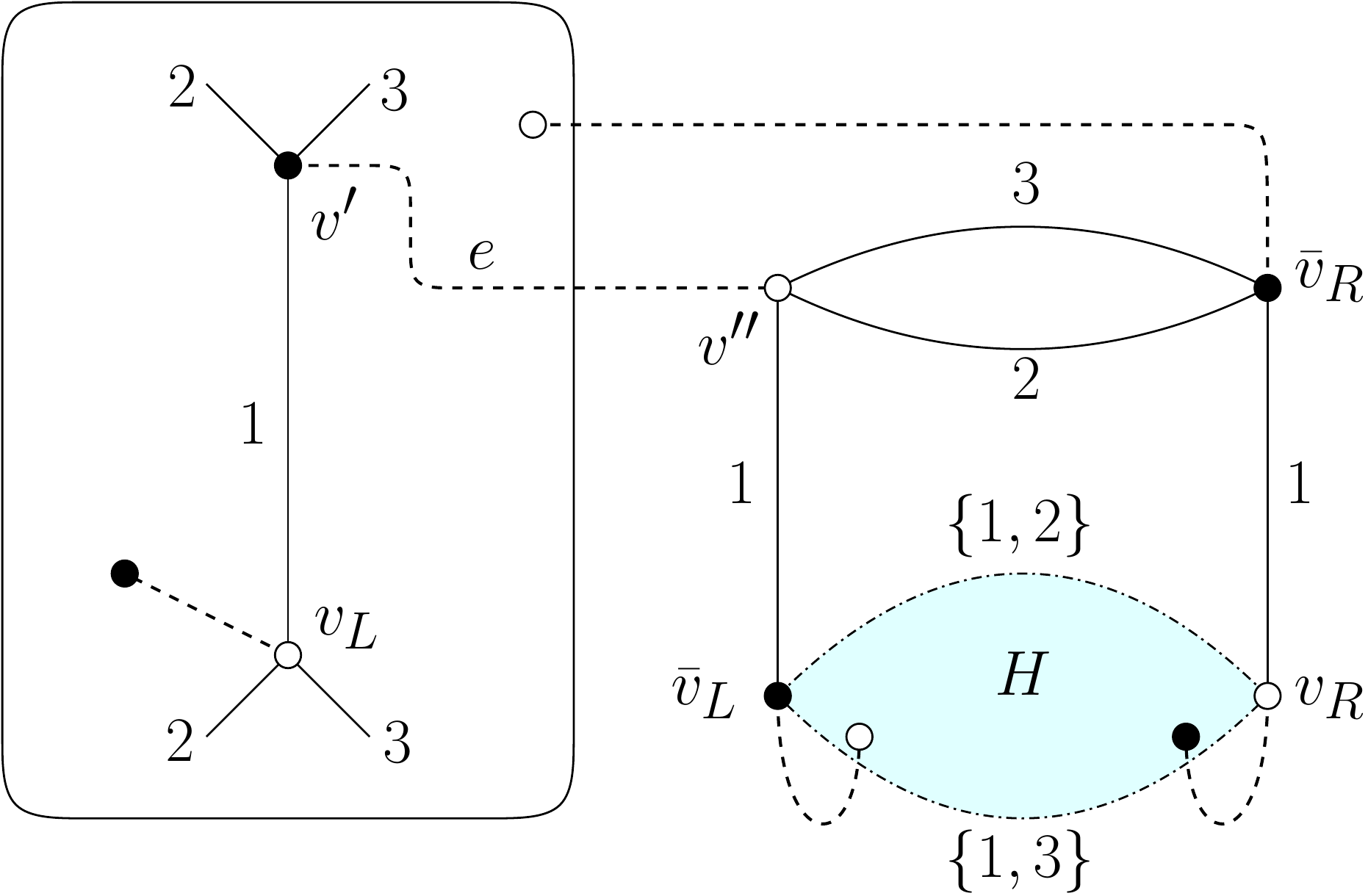} \end{array} \qquad G' = \begin{array}{c} \includegraphics[scale=.28]{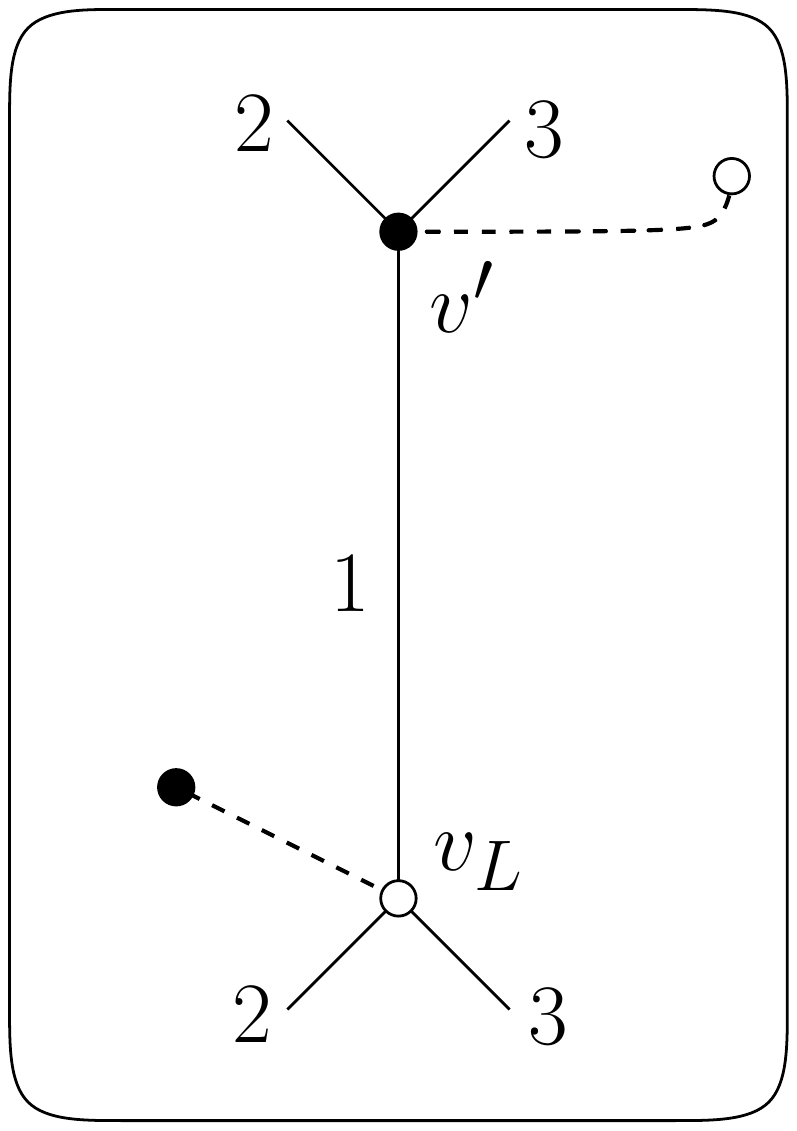} \end{array} \qquad G'' = \begin{array}{c} \includegraphics[scale=.28]{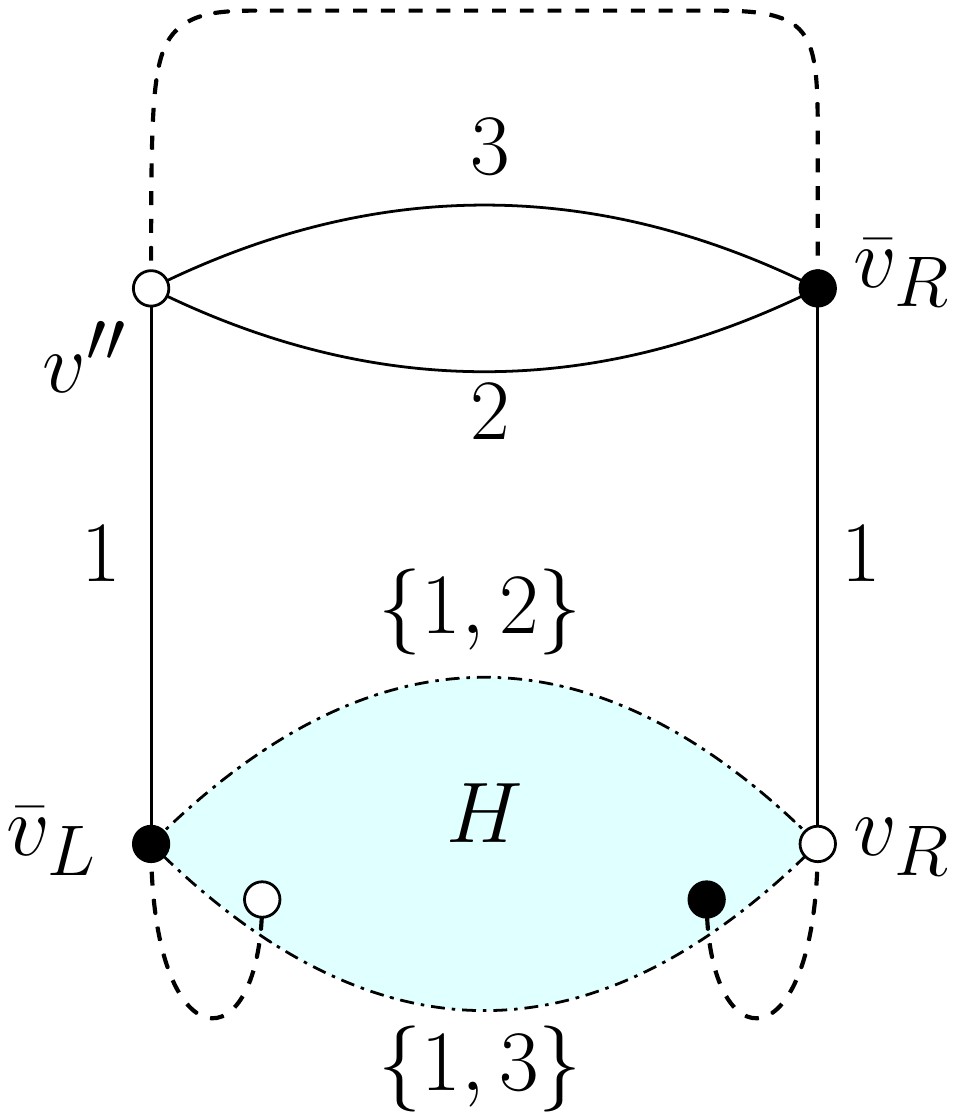} \end{array}
\end{equation}
We then get a graph made of the bubble $B$ and maximizing the number of bicolored cycles by contracting $e$,
\begin{equation}
G = \tilde{G}/e
\end{equation}
and all graphs $G\in\cG^{\max}_1(B)$ can be found in this form. We can finally track down the topology. From the induction hypothesis, $G'$ and $G''$ are spheres. From Lemma \ref{lemma:2CutConnected}, $\tilde{G}$ is thus a sphere (as a connected sum of spheres). Finally the contraction of $e$ in $\tilde{G}$ is a 1-dipole removal since the 3-bubbles incident to $v'$ and $v''$ are distinct (they are $B'$ and $B''$) and hence is topological. As a conclusion, $G$ is a sphere.

\item[$G$ has a single bubble with no 2-edge-cut, has a bicolored cycle of length 2] This narrows down the bubbles $B$ of interest to those with no two faces sharing more than one edge. In particular, $B$ does not have a 2-dipole, hence it has at least six faces of degree 4 from Lemma \ref{lemma:FacesPlanarBubble}.

Moreover, if $G$ has a bicolored cycle of length 2, i.e. an edge of color 0 connecting two adjacent vertices of $B$, then we can apply the induction as follows. The situation is exactly like in \eqref{1DipoleMove} with $f_{c_1 c_2} \neq f'_{c_1 c_2}$ or else it would share two edges with other faces of $B$. After the contraction, we have a bubble $B'$ and a graph $G'=G/e$. $B'$ is planar and $G'$ must be in $\cG^{\max}_1(B')$ from Lemma \ref{lemma:NotMax}. The induction hypothesis then tells us that $G'$ is a sphere. The edge $e$ in $G$ is in fact part of a 2-dipole and its contraction is the dipole removal, which preserves the topology so that $G$ also is a sphere.

We can thus assume that $G$ has no bicolored cycle of length 2. The final case to consider is the following.

\item[$G$ has no bicolored cycles of length 2 and a single bubble with no 2-edge-cut] Since $B$ has no 2-edge-cut, it does not have any faces of degree 2. Following Lemma \ref{lemma:FacesPlanarBubble}, it however has faces of degree 4. Let us consider one with, say, colors $\{1, 2\}$,
\begin{equation}
G = \begin{array}{c} \includegraphics[scale=.45]{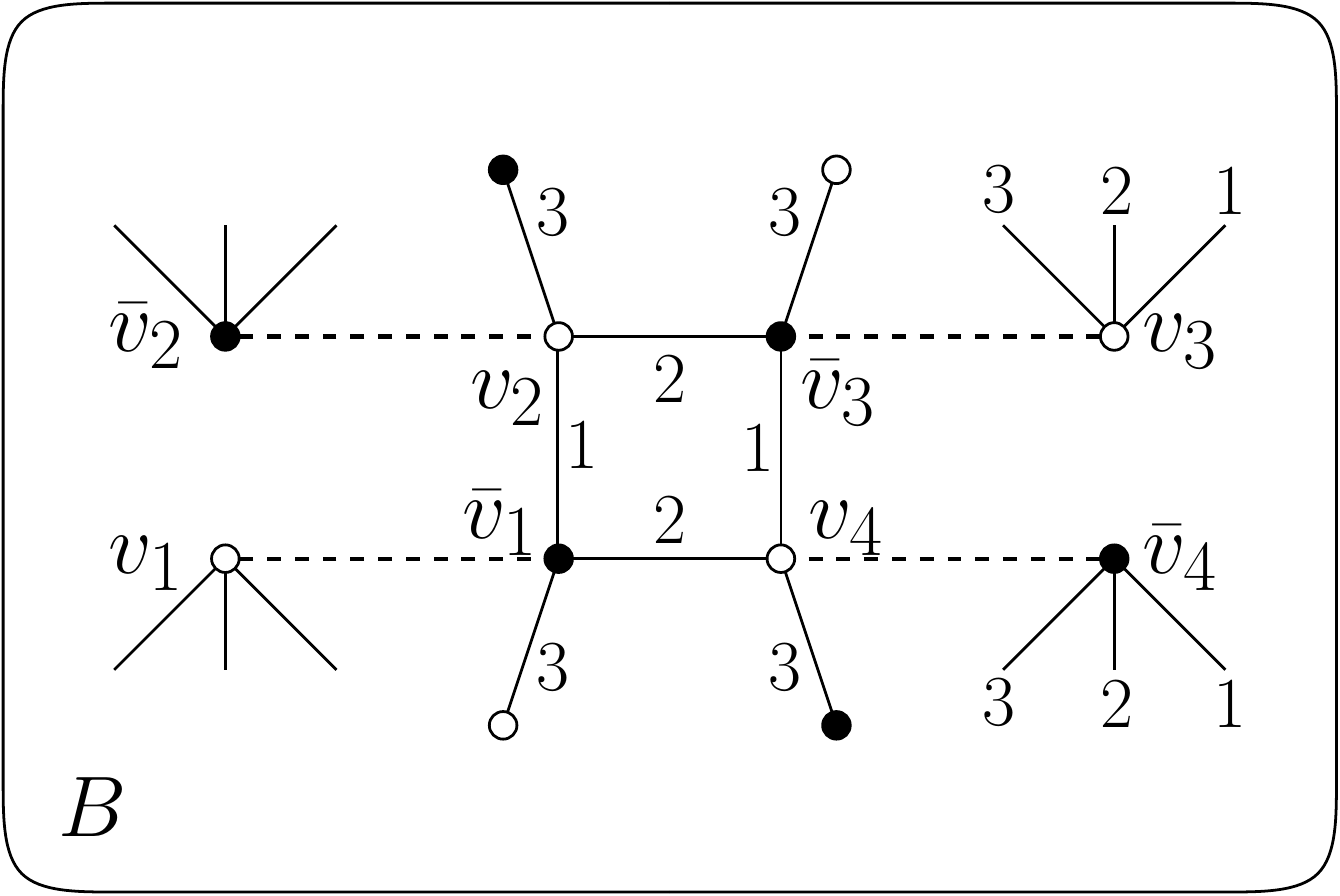} \end{array}
\end{equation}
In particular, the two faces with colors $\{2,3\}$ adjacent to the face of degree 4 are distinct, or else those two faces would share two edges (of color 2). Also, the edges of color 0 drawn above do not form bicolored cycles of length 2.

We perform the same set of moves as in \eqref{OneFlip}, \eqref{TwoFlips}: flip the edges of color 0 incident to $\bar{v}_3$ and $v_4$, then the edges between $\bar{v}_1$ and $v_2$
\begin{equation}
G_{\mid} = \begin{array}{c} \includegraphics[scale=.45]{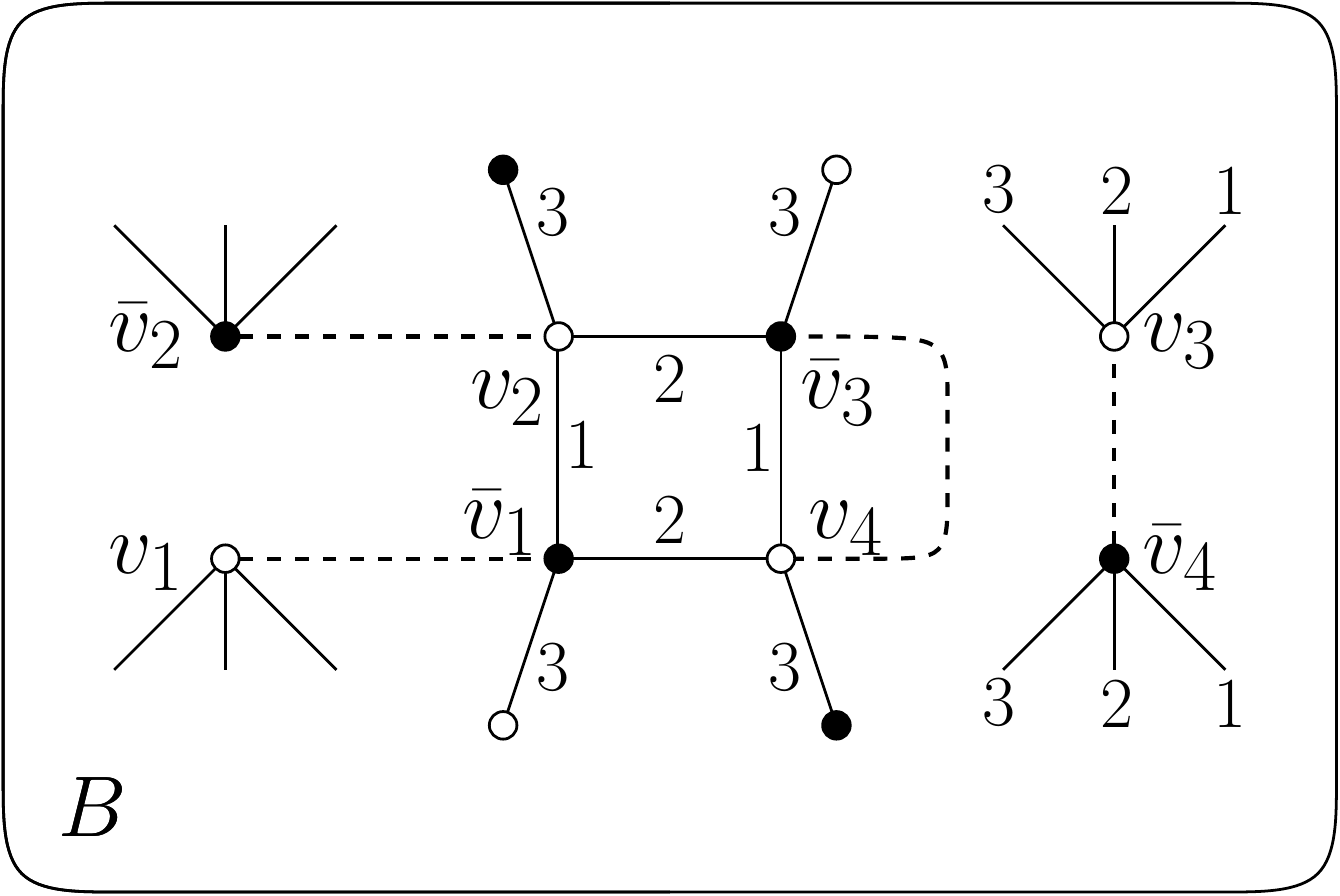} \end{array} \hspace{.5cm}\text{and}\hspace{.5cm}
G_{\parallel} = \begin{array}{c} \includegraphics[scale=.45]{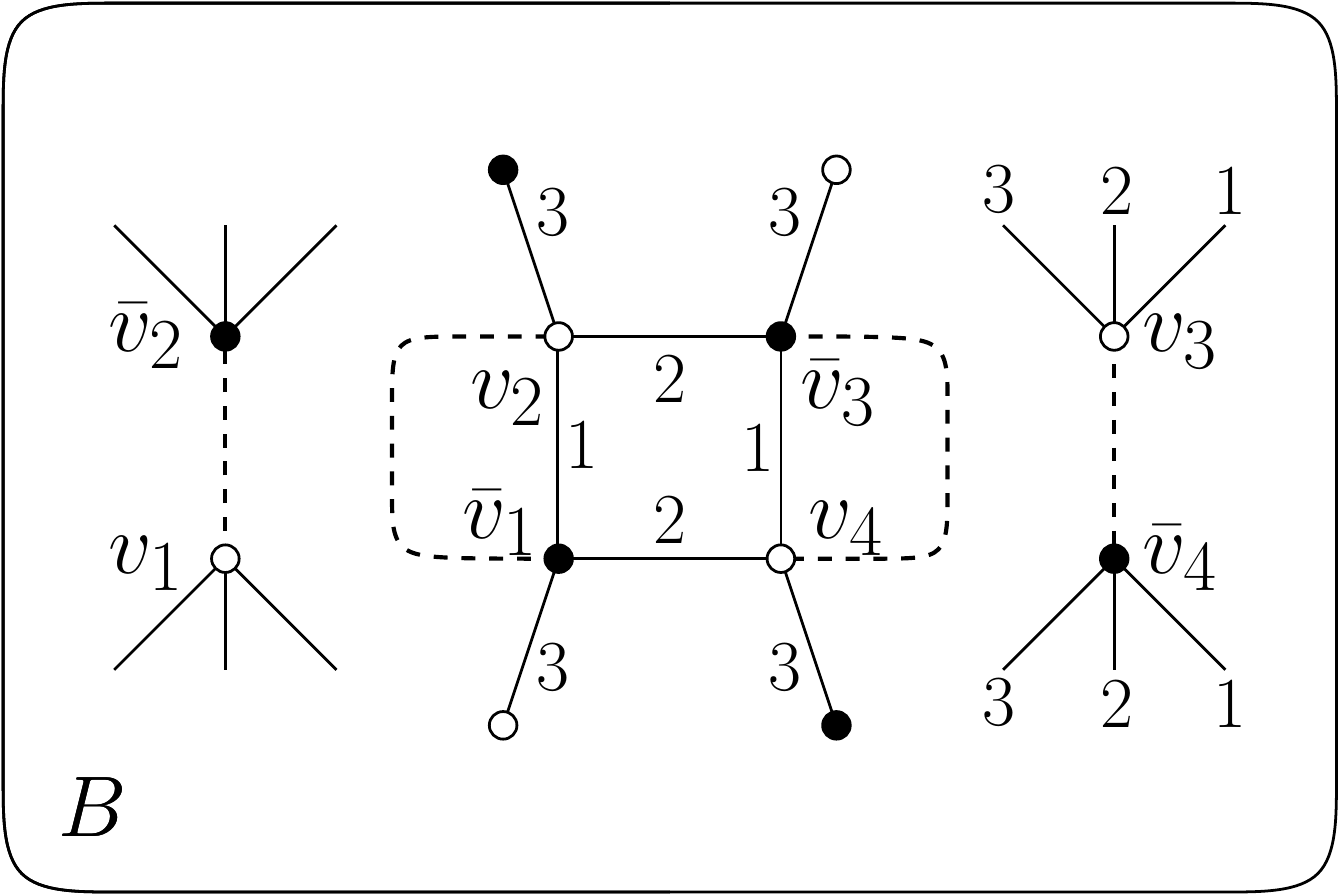} \end{array}
\end{equation}
The bound \eqref{TwoFlipsBound}, $C_0(G_{\parallel}) \geq C_0(G)$ applies, and since $G$ is assumed to maximize the number of bicolored cycles we find
\begin{equation} \label{CyclesGParallel}
C_0(G_{\parallel}) = C_0(G).
\end{equation}
We can then contract the edges of color 0, say like
\begin{equation}
G' = \begin{array}{c} \includegraphics[scale=.45]{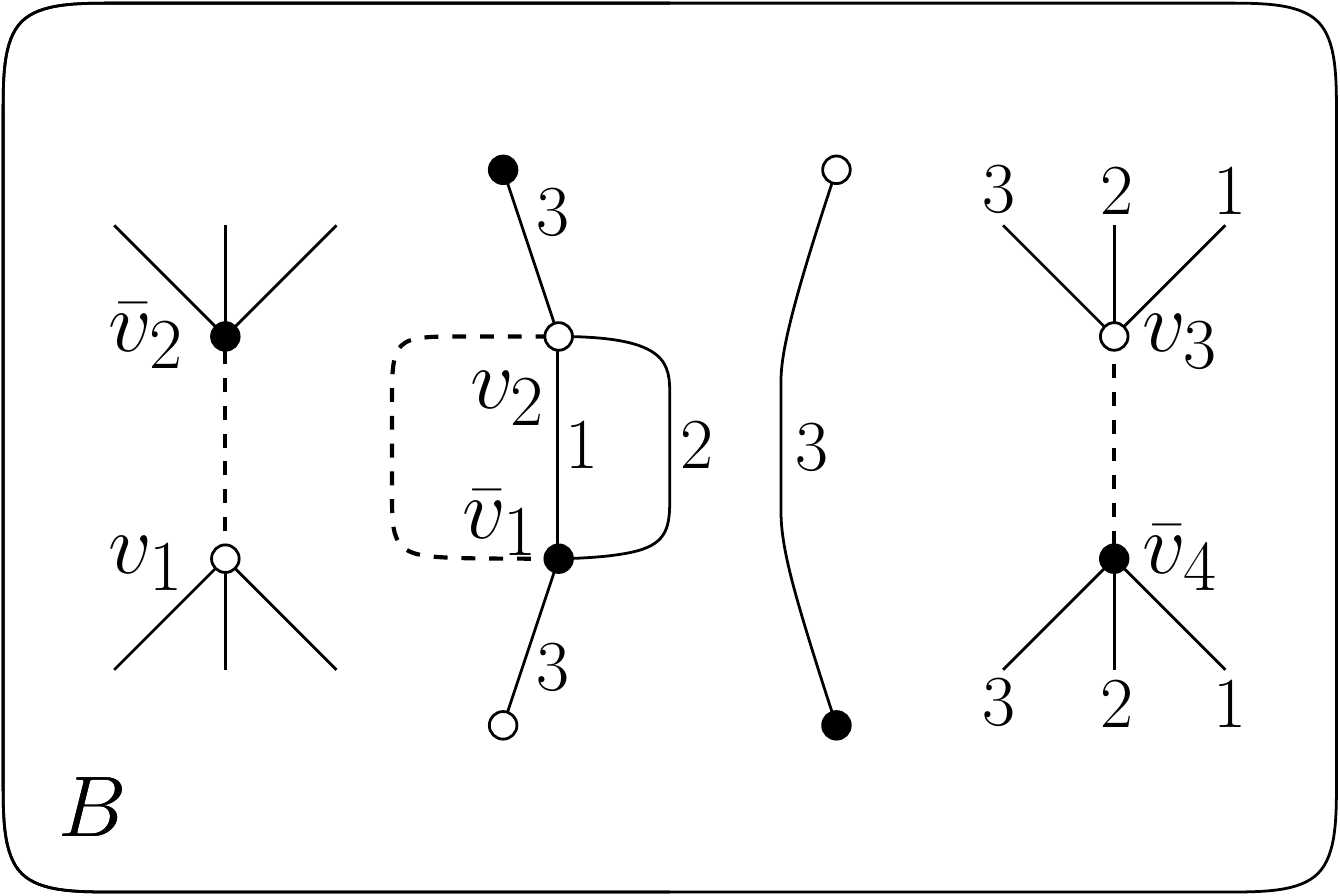} \end{array}
\hspace{.5cm}\text{then}\hspace{.5cm}
G'' = \begin{array}{c} \includegraphics[scale=.45]{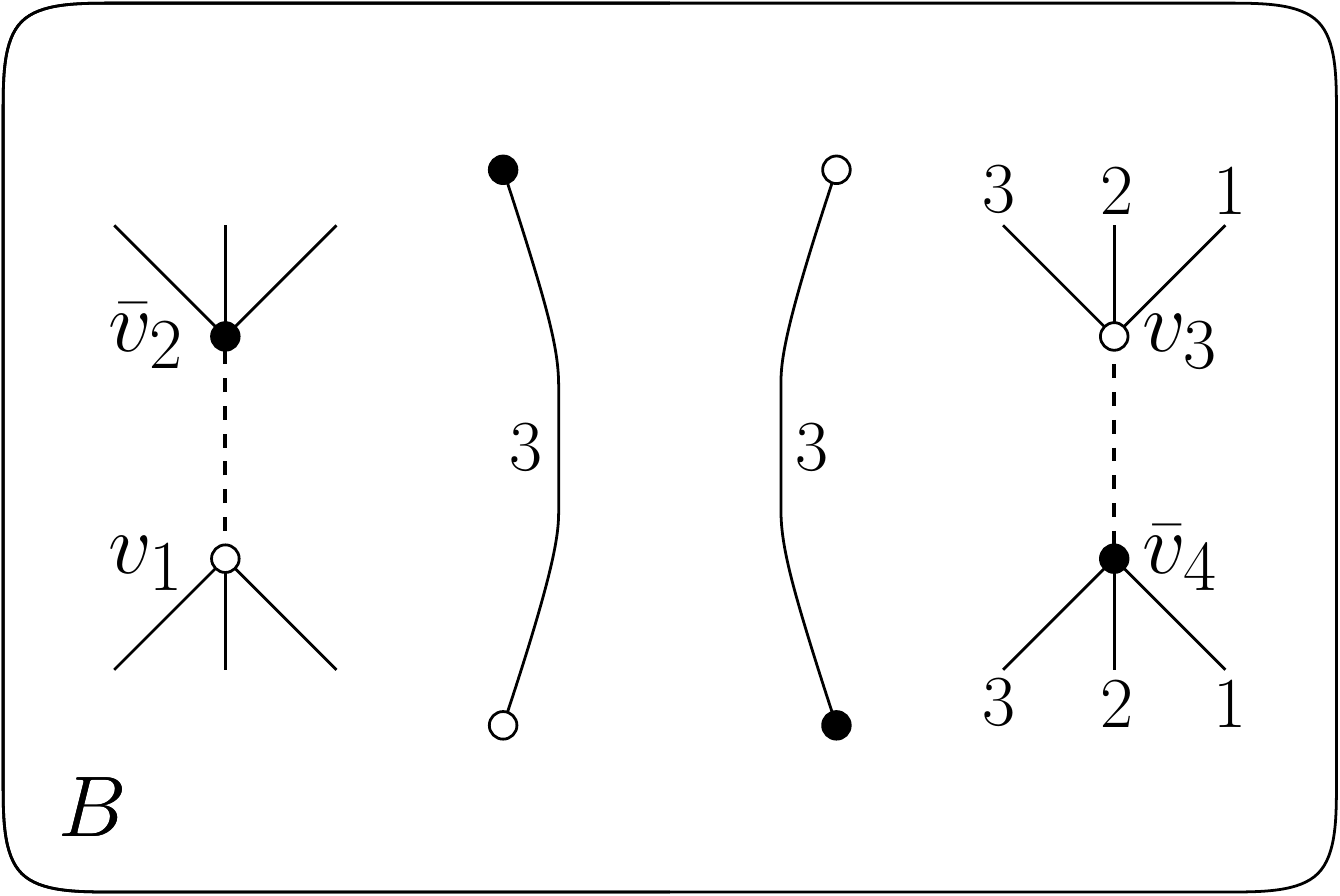} \end{array}
\end{equation}
Since $B$ has no faces which share more than one edge, those contractions turn $B$ into connected bubbles $B'\subset G'$ and $B''\subset G''$. The graph $G$ becomes $G'\in\cG^{\max}_1(B')$ and $G''\in\cG^{\max}_1(B'')$, see Lemma \ref{lemma:NotMax}. From the induction hypothesis, $G''$ is a sphere. It will thus be sufficient to show that $G$ and $G''$ are connected by a sequence of moves which preserves topology.

As is well known in colored graph theory, a colored graph encodes a 3-manifold if and only if all its 3-bubbles represent 3-balls (a 3-bubble being dual to a vertex of the triangulation, it means that the neighborhoods of all vertices are balls). $G''$ being a sphere, all its 3-bubbles represent balls which means that they are planar, in the sense of Corollaries \ref{cor:2D} and \ref{cor:PlanarBubbles}.

Let us focus on the 3-bubbles with colors $\{0, 2, 3\}$ in $G_{\parallel}$ and in $G''$, i.e. $G_{\parallel}(0,2,3)$ and $G''(0,2,3)$. Since the face of degree 4 in $G_{\parallel}$ separates two distinct faces of colors $\{2,3\}$, the move from $G_{\parallel}(0,2,3)$ to $G''(0,2,3)$ is topological -- this is twice Proposition \ref{prop:Contraction} for the bubbles with colors $\{0,2,3\}$ instead of $\{1,2,3\}$ and contracting the two edges of color 1 instead of 0. This implies that the number of connected components and their genera are unchanged. In particular, the 3-bubbles of $G_{\parallel}$ with colors $\{0,2,3\}$ are planar.

Let us denote $(G^{(\rho)}(0,2,3))_\rho$ the 3-bubbles, i.e. the connected components of $G(0,2,3)$, and $g^{(\rho)}(0,2,3)$ their genera, and similarly with $G_{\parallel}$. The total number of bicolored cycles with colors $\{0,2\}$ in $G(0,2,3)$ is the same as in $G$ itself, same for bicolored cycles with colors $\{0,3\}$ and $\{2,3\}$. The number of edges of $G(0,2,3)$ is $3V(B)/2$, $V(B)$ being the number of vertices of $B$. We denote $K_{23}(G)$ the number of 3-bubbles with colors $\{0,2,3\}$. Euler's relation for $G(0,2,3)$ thus reads
\begin{equation}
C_{02}(G) + C_{03}(G) + C_{23}(G) - V(B)/2 = 2K_{23}(G) - 2\sum_\rho g^{(\rho)}(0,2,3).
\end{equation}

The reason for the bound \eqref{TwoFlipsBound} is that the flips from $G$ to $G_{\parallel}$ produces two new bicolored cycles with colors $\{0,1\}$, but they may not increase the number of bicolored cycles with colors $\{0,2\}$ and may decrease by two the number of bicolored cycles with colors $\{0,3\}$. Since equality \eqref{CyclesGParallel} holds here, it means we are in the ``worst'' case, i.e. 
\begin{equation}
C_{02}(G_{\parallel}) = C_{02}(G) \qquad \text{and} \qquad C_{03}(G_{\parallel}) = C_{03}(G) -2.
\end{equation}
Euler's relation for $G_{\parallel}(0,2,3)$ then gives
\begin{equation}
C_{02}(G) + C_{03}(G) + C_{23}(G) - V(B)/2 = 2K_{23}(G_{\parallel}) + 2,
\end{equation}
since its 3-bubbles are planar. Therefore
\begin{equation}
K_{23}(G) = K_{23}(G_{\parallel}) + 1 + \sum_\rho g^{(\rho)}(0,2,3) > K_{23}(G_{\parallel}),
\end{equation}
i.e. the number of 3-bubbles with colors $\{0,2,3\}$ decreases from $G$ to $G_{\parallel}$. Obviously, only the 3-bubbles which contain $\bar{v}_1, v_2, \bar{v}_3, v_4$ are affected by the flips. Therefore, if the four of them are in the same 3-bubble with colors $\{0,2,3\}$ in $G$, then the number of those 3-bubbles cannot decrease. If $v_2, \bar{v}_3$ are in a different 3-bubble with colors $\{0,2,3\}$ than $\bar{v}_1, v_2$, then the number of those bubbles can decrease by at most one. We can thus only be in the latter case, i.e.
\begin{equation}
K_{23}(G) = K_{23}(G_{\parallel}) + 1,
\end{equation}
which also implies the 3-bubbles are planar,
\begin{equation}
g^{(\rho)}(0,2,3) = 0.
\end{equation}
These two equations provide the two conditions needed to apply Corollary \ref{cor:TopologicalFlip} and prove that the flip from $G$ to $G_{\mid}$ is topological: that $\bar{v}_3$ and $v_4$ lie in different 3-bubbles with colors $\{0,2,3\}$ and that at least one of these bubbles is planar.

Finally notice now that in $G_{\mid}$ the two edges of color 1 and 0 between $\bar{v}_3$ and $v_4$ form a 2-dipole, because they separate two distinct faces of colors $\{2, 3\}$ of $B$. This 2-dipole can thus be contracted to give
\begin{equation}
G_{\mid}' = \begin{array}{c} \includegraphics[scale=.45]{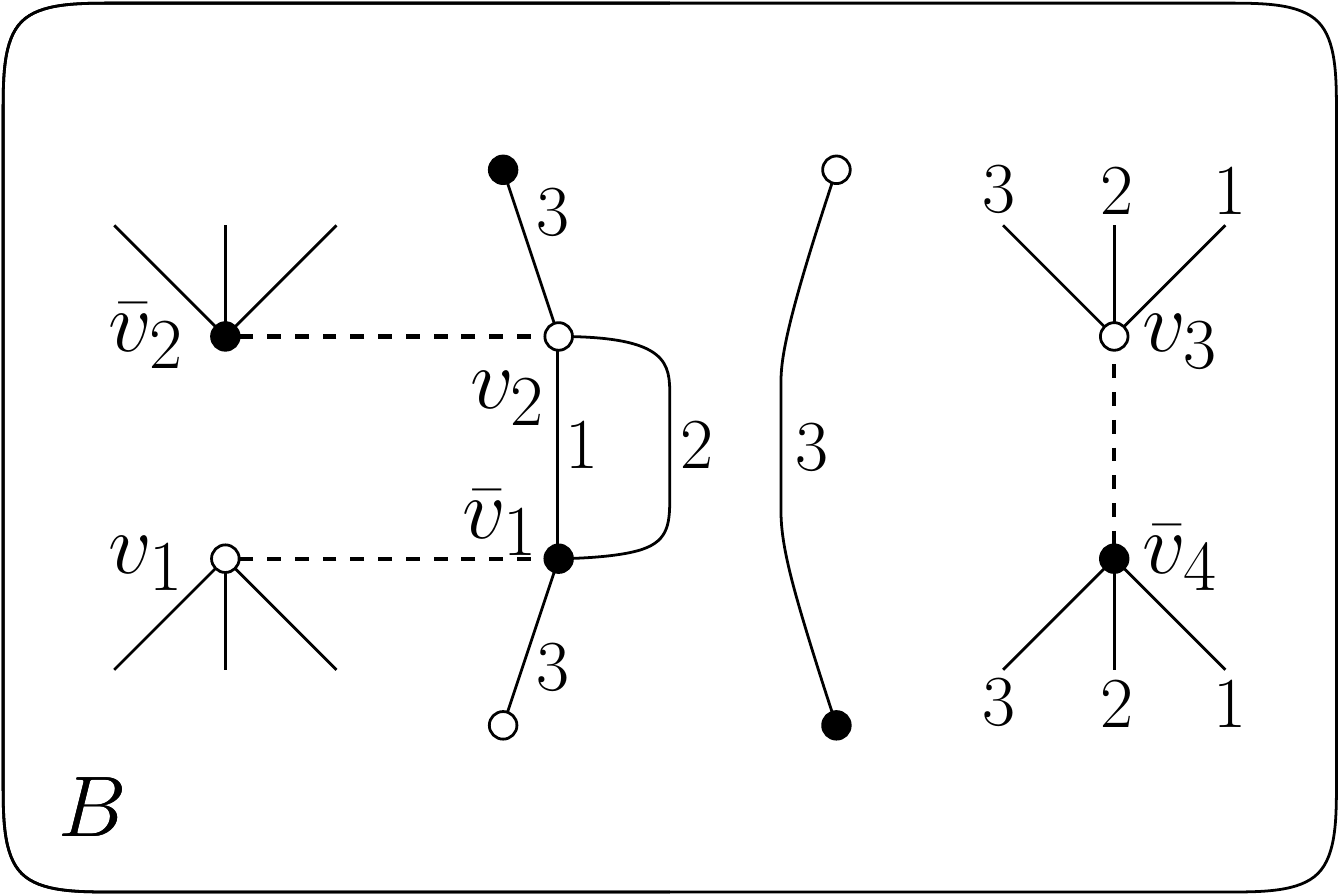} \end{array}
\end{equation}
The bicolored cycles with colors $\{0, 3\}$ incident on $\bar{v}_1$ and $v_2$ in $G_{\mid}$ are different, or else a bicolored cycle of colors $\{0, 3\}$ would be created from $G_{\mid}$ to $G_{\parallel}$ instead of destroyed. They are thus also different in $G'_{\mid}$. That implies that the edges of colors 1 and 2 between $\bar{v}_1$ and $v_2$ form a 2-dipole. Contracting this 2-dipole produces the graph $G''$. Since we have found a sequence of topological moves from $G$ to $G''$, we find that $G$ is a sphere.
\end{description}
\end{proof}

%%%%%%%%%%%%%%%%%%%
\section*{Acknowledgements}

This research was supported by the ANR MetACOnc project ANR-15-CE40-0014.

%%%%%%%%%%%%%%%%%%%%

\end{document}